\documentclass[11pt]{article}

\usepackage{amsmath,amssymb,amsthm}
\usepackage[a4paper,margin=1in]{geometry}
\usepackage{booktabs}
\usepackage{enumitem}
\usepackage{microtype}
\usepackage[title,titletoc]{appendix}
\usepackage[hidelinks]{hyperref}

\newtheorem{theorem}{Theorem}[section]
\newtheorem{proposition}[theorem]{Proposition}
\newtheorem{lemma}[theorem]{Lemma}
\newtheorem{corollary}[theorem]{Corollary}
\newtheorem{conjecture}[theorem]{Conjecture}
\theoremstyle{definition}

\theoremstyle{remark}
\newtheorem{remark}[theorem]{Remark}

\DeclareMathOperator{\tr}{tr}

\DeclareMathOperator{\diag}{diag}

\DeclareMathOperator{\Res}{Res}
\DeclareMathOperator{\pp}{pp}

\newcommand{\R}{\mathbb{R}}

\newcommand{\norm}[1]{\lVert #1 \rVert}
\newcommand{\M}{\mathcal M}
\newcommand{\one}{\mathbf 1}
\newcommand{\ph}{\varphi}
\allowdisplaybreaks[1]
\hypersetup{pdftitle={The Converse Problem for the Morley Tetrahedron},pdfauthor={Quang Hung Tran}}

\title{The Converse Problem for the Morley Tetrahedron:\\Counterexamples, Conjectures, and Partial Results}
\author{Quang Hung Tran}
\date{}

\begin{document}
\maketitle

\begin{abstract}
In a paper in Acta Mathematica Hungarica the author proved that the Morley tetrahedron of an isosceles tetrahedron, obtained by trisecting the six dihedral angles, is again isosceles, and proposed two converse conjectures. We show that both are false. There is a nonisosceles tetrahedron $T_1$ and an isosceles, nonregular tetrahedron $T_2$ whose Morley tetrahedra are regular, and there are nonisosceles tetrahedra, even a two-parameter family of tetrahedra without any symmetry, whose Morley tetrahedra are isosceles. In $T_1$ and in $T_2$ there is a pair of opposite edges such that the other four edges are equal, and we conjecture that a regular Morley tetrahedron always forces this. We prove the conjecture for every tetrahedron with a nontrivial symmetry, and we show that, up to similarity, the regular tetrahedron, $T_1$ and $T_2$ are the only tetrahedra with this edge pattern and a regular Morley tetrahedron. The tetrahedron $T_2$ has $AB=CD=1$ and $AC=AD=BC=BD=\sqrt{(21+4\sqrt6)/45}$, while $T_1$ is given by a root of a sextic with Galois group $S_6$ and cannot be expressed by radicals. We also prove that a tetrahedron with a regular Morley tetrahedron is regular if it is orthocentric, if it is isodynamic, if its three sums of opposite edges are equal, if it has three equal edges at a vertex, if it has an equilateral face, or if none of its dihedral angles is larger than $95^\circ$; the tetrahedron $T_2$ has two dihedral angles of about $98.7^\circ$. For isosceles Morley tetrahedra we conjecture that $(AB^2-CD^2)(AC^2-BD^2)(AD^2-BC^2)\ge0$ and that $AB=CD$ forces a second pair of equal opposite edges. We also conjecture that a tetrahedron with $AC=BD$ whose Morley tetrahedron satisfies $A'B'=B'C'=C'D'=D'A'$ has a nontrivial symmetry. Some proofs are computer assisted; they use exact rational arithmetic or interval arithmetic with outward rounding.
\end{abstract}



\tableofcontents

\section{Introduction}

Morley's trisector theorem states that the intersections of adjacent angle trisectors of a triangle form an equilateral triangle. The theorem goes back to Morley \cite{Morley}; see also Child \cite{Child}, Coxeter and Greitzer \cite{CoxeterGreitzer}, Oakley and Baker \cite{OakleyBaker}, Hashimoto \cite{Hashimoto}, and Grinberg and Orhon \cite{GrinbergOrhon} for various proofs and viewpoints. Several extensions of Morley's configuration have also been investigated; see, for example, \cite{Braude,DergiadesTran}. A recent synthetic proof is given in \cite{TranSynthetic}.

A three-dimensional analogue was considered by Svrtan and Veljan \cite{SvrtanVeljan}. In \cite{TranActa} the author defined a Morley tetrahedron by trisecting the six dihedral angles of a tetrahedron, and proved that the Morley tetrahedron of an isosceles tetrahedron is again isosceles. We refer to this as the direct theorem. It follows from the fact that every symmetry of a tetrahedron is also a symmetry of its Morley tetrahedron, and the same argument shows that two pairs of equal opposite edges already give two pairs of equal opposite edges of the Morley tetrahedron (Corollary~\ref{cor:inherited}).

A tetrahedron $ABCD$ is called \emph{isosceles} when its three pairs of opposite edges have equal lengths, that is, $AB=CD$, $AC=BD$ and $AD=BC$. Here regular tetrahedra are included; in \cite{TranActa} they were excluded. Such tetrahedra are also called \emph{equifacial tetrahedra} or \emph{disphenoids}; see \cite{BrandtsKrizek,HajjaWalker,Leech}. Related tetrahedra with congruent pairs of facets are studied in \cite{Klain,TranVolume}.

In \cite{TranActa} the author also proposed two converse statements.

\begin{conjecture}[The original converse conjecture]
\label{conj:old1} If the Morley tetrahedron of $ABCD$ is regular, respectively isosceles, then $ABCD$ is regular, respectively isosceles.
\end{conjecture}

For $p>1$, a plane through an edge $XY$ that meets the opposite edge is called $p$-secting in \cite{TranActa} if it makes the angle $\theta_{XY}/p$ with one of the two faces through $XY$, where $\theta_{XY}$ is the dihedral angle at $XY$; for $p=3$ these are the trisecting planes. Theorem~6 of \cite{TranActa} replaces the trisecting planes in the Morley construction by $p$-secting planes, with $p=e$ at the edges $AB$ and $CD$, $p=f$ at $BC$ and $DA$, and $p=g$ at $AC$ and $BD$, where $e,f,g>1$, and shows that the tetrahedron obtained from an isosceles tetrahedron is again isosceles. For $e=f=g=3$ this is the direct theorem \cite[Remark~7]{TranActa}. The corresponding converse was proposed as well.

\begin{conjecture}[The original converse for $p$-secting planes]
\label{conj:old2} Let $A'B'C'D'$ be defined as in \cite[Theorem 6]{TranActa}, for an arbitrary tetrahedron $ABCD$ and numbers $e,f,g>1$. If $A'B'C'D'$ is regular, respectively isosceles, then $ABCD$ is regular, respectively isosceles.
\end{conjecture}

\subsection*{Results}

The paper is divided into three parts.

Part~\ref{part:counterexamples} contains the counterexamples. Theorem~\ref{thm:T1} gives a nonisosceles tetrahedron $T_1$ with a regular Morley tetrahedron, and Theorem~\ref{thm:T2} gives an isosceles one, $T_2$, in closed form. Theorem~\ref{thm:curve} gives a curve of nonisosceles tetrahedra whose Morley tetrahedra are isosceles but not regular. Hence both converse conjectures fail.

In Part~\ref{part:conjectures} we look for a statement that could replace them. Both $T_1$ and $T_2$ have a pair of opposite edges, say $AB$ and $CD$, such that the other four edges, which we call the cross edges, are equal. We conjecture that every tetrahedron with a regular Morley tetrahedron has this property (Conjecture~\ref{conj:four}), and we discuss several variants.

Part~\ref{part:results} collects what we can prove. Suppose that the Morley tetrahedron of $ABCD$ is regular. If $ABCD$ is isosceles, or has two pairs of equal opposite edges, or has a reflection symmetry, then it has four equal cross edges (Theorems~\ref{thm:disphenoid}, \ref{thm:twopairs} and~\ref{thm:reflection}). If it has four equal cross edges, it is similar to the regular tetrahedron, to $T_1$ or to $T_2$ (Theorem~\ref{sym:thm:classification}). Together these results show that a tetrahedron with a nontrivial symmetry and a regular Morley tetrahedron is one of these three (Corollary~\ref{cor:symmetric}). A counterexample to Conjecture~\ref{conj:four}, if there is one, therefore has no symmetry at all. We also show that $T_1$ cannot be expressed by radicals (Proposition~\ref{prop:T1galois}) and that the three known solutions are isolated (Theorem~\ref{thm:localregular} and Section~\ref{sec:intervals}). Finally, a tetrahedron with a regular Morley tetrahedron is regular if it is orthocentric, if it is isodynamic, if its three sums of opposite edges are equal, if it has three equal edges at a vertex, or if it has an equilateral face (Theorem~\ref{thm:special}); such tetrahedra have no symmetry in general, so these are cases of Conjecture~\ref{conj:four} not covered by the symmetry results. The same conclusion holds for every tetrahedron whose dihedral angles are at most $95^\circ$ (Theorem~\ref{thm:acute}). This bound is not far from the best possible, since $T_2$ has two dihedral angles of about $98.69^\circ$. The general case remains open.

For isosceles Morley tetrahedra the situation is different. Near the three known solutions a tetrahedron with an isosceles Morley tetrahedron has a half-turn symmetry, but there is a two-parameter family of such tetrahedra without any symmetry (Proposition~\ref{prop:isolocal} and Theorem~\ref{thm:asymmetric}). What seems to survive is a sign condition. We conjecture that $(AB^2-CD^2)(AC^2-BD^2)(AD^2-BC^2)\ge0$ and that $AB=CD$ forces a second pair of equal opposite edges (Conjecture~\ref{conj:isoproduct}), and we propose a sign rule for the Morley tetrahedron when $AB=CD$ that implies the second statement (Conjecture~\ref{conj:isosign}). Proposition~\ref{prop:isonormal} gives the first conjecture near every isosceles tetrahedron at which at least two of the three transversal derivatives of the Morley map are nonzero.

Section~\ref{sec:skew} treats Morley tetrahedra with four equal edges in a skew quadrilateral, $A'B'=B'C'=C'D'=D'A'$. Near the regular tetrahedron this forces $AB=BC=CD=DA$ (Proposition~\ref{prop:skewlocal}), but in general it forces no symmetry. If in addition $AC=BD$, every solution we found has a symmetry, and Conjecture~\ref{conj:skew} says that this is always the case. Possible approaches and open problems are discussed in Sections~\ref{tet:directions} and~\ref{sec:openproblems}.

The same questions in higher dimensions are studied in \cite{TranSimplices}. There the answer depends on the dimension: the analogue of Conjecture~\ref{conj:four} is conjectured in dimensions four and five, and it fails in dimensions six and seven and in infinitely many other dimensions.

Several proofs are computer assisted. Most of them use only exact arithmetic with rational numbers: polynomial identities, resultants and inequalities proved by rational interval arithmetic. The proof of Theorem~\ref{thm:special} uses floating-point interval arithmetic with outward rounding. The scripts are described in Appendix~\ref{sec:reproducibility}.

\part{Counterexamples to the original converse conjectures}
\label{part:counterexamples}

Both counterexamples lie in a family of tetrahedra with two reflection symmetries, introduced in Section~\ref{sec:family}. The computations they rely on are in Section~\ref{sec:intervals}.

\section{The Morley tetrahedron}\label{sec:morley}

Let $ABCD$ be a nondegenerate tetrahedron. For an edge $XY$, the two faces containing $XY$ form a dihedral angle. If $Z$ is one of the two vertices not on $XY$, we denote by
\[
(XY)_Z
\]
the trisecting plane through $XY$ which is closer to $Z$. Equivalently, if $\theta_{XY}$ is the interior dihedral angle along $XY$, then $(XY)_Z$ makes the angle $\theta_{XY}/3$ with the face containing $XY$ and $Z$.

Following \cite{TranActa}, define the Morley tetrahedron
\[
\M(ABCD)=A'B'C'D'
\]
by
\begin{align}
A'&=(BC)_D\cap(CD)_B\cap(DB)_C, \notag\\
B'&=(AC)_D\cap(CD)_A\cap(DA)_C, \notag\\
C'&=(AB)_D\cap(BD)_A\cap(DA)_B, \notag\\
D'&=(AB)_C\cap(BC)_A\cap(CA)_B. \label{eq:morleydefinition}
\end{align}
In other words, $A'$ is the intersection of the three trisecting planes attached to the three edges of the face $BCD$ and lying closer to that face, and similarly for the other three vertices.

\begin{lemma}
\label{lem:equivariance} The Morley construction commutes with Euclidean isometries and with relabelings of the four vertices.
\end{lemma}

\begin{proof}
The construction uses only dihedral angles, their trisecting planes, the choice of the trisector closer to a specified face, and intersections of planes. All these operations commute with Euclidean isometries.

The description of a Morley vertex as the intersection of the three trisectors attached to the three edges of the opposite face is symmetric under relabeling.
\end{proof}

In particular, every symmetry of $ABCD$ is a symmetry of its Morley tetrahedron, with the same permutation of the labels.

\begin{corollary}
\label{cor:inherited}
Let $g$ be an isometry that maps the vertex set of $ABCD$ onto itself, and let $\pi$ be the permutation of $\{A,B,C,D\}$ that it induces. Then $g$ maps $\M(ABCD)$ onto itself, and it maps the vertex attached to the face opposite $X$ to the vertex attached to the face opposite $\pi(X)$. In particular:
\begin{enumerate}[label=\textup{(\alph*)}]
\item if $AC=BD$ and $AD=BC$, then $A'C'=B'D'$ and $A'D'=B'C'$, and in the same way for the other two choices of two pairs of opposite edges;
\item if $AC=BC$ and $AD=BD$, then $A'C'=B'C'$ and $A'D'=B'D'$;
\item if $ABCD$ is isosceles, then so is $\M(ABCD)$.
\end{enumerate}
\end{corollary}

\begin{proof}
By Lemma~\ref{lem:equivariance}, $g(\M(ABCD))$ is the Morley tetrahedron of $g(A)g(B)g(C)g(D)$, which is $ABCD$ with its labels permuted by $\pi$. This gives the first statement. If $AC=BD$ and $AD=BC$, the permutation that interchanges $A$ with $B$ and $C$ with $D$ preserves the six edge lengths. The vertices are affinely independent, so it is induced by an isometry $g$, the half-turn about the line through the midpoints of $AB$ and $CD$. Then $g(A')=B'$, $g(B')=A'$, $g(C')=D'$ and $g(D')=C'$, and therefore $A'C'=B'D'$ and $A'D'=B'C'$. In (b) the permutation interchanging $A$ and $B$ preserves the edge lengths and is induced by the reflection in the plane through $C$, $D$ and the midpoint of $AB$. Statement (c) follows from (a), applied to two of the three pairs.
\end{proof}

Statement (c) is the direct theorem of \cite{TranActa}, and (a) shows that two pairs of equal opposite edges are already enough for the corresponding two pairs of the Morley tetrahedron. The converse of (a) fails: the tetrahedra of Theorem~\ref{thm:asymmetric} have isosceles Morley tetrahedra, but their six edges are all different.

\section{The Gram matrix formulation}
\label{sec:gram}

Write the vertices as $A_1,A_2,A_3,A_4$, and let $F_i$ be the face opposite $A_i$. Denote its area by $S_i$, its inward unit normal by $n_i$, and the volume of the tetrahedron by $V$. Let $\theta_{ij}$ be the interior dihedral angle between $F_i$ and $F_j$; it is the angle along the edge $F_i\cap F_j$, which is the edge opposite to $A_iA_j$. Set $q_{ii}=1$ and $q_{ij}=2\cos(\theta_{ij}/3)$ for $i\ne j$, let $K=(q_{ij})$, and put
\[
\ph(x)=\frac{3x-x^3}{2}.
\]
Since $0<\theta_{ij}<\pi$, we have $1<q_{ij}<2$ for $i\ne j$. Since $\ph(2\cos\alpha)=-\cos3\alpha$, the normal Gram matrix is
\begin{equation}
H=(n_i\cdot n_j)=\ph[K]=\frac{3K-K^{\circ3}}2,
\label{eq:normalgram}
\end{equation}
where the brackets and the superscript $\circ3$ indicate entrywise operations. The equilibrium of the face area vectors gives
\begin{equation}
\begin{gathered}
H\succeq0\ \text{ of rank }3,\\
HS=0\ \text{ with }\ S=(S_1,S_2,S_3,S_4)^T>0.
\end{gathered}
\label{eq:positivekernel}
\end{equation}

Let $\ell_j(X)$ be the inward signed distance to $F_j$, and let $M_i$ be the Morley vertex corresponding to $F_i$. The sine rule in a normal section gives
\begin{equation}
\begin{aligned}
\ell_j(M_i)&=t_iq_{ij},\\
t_i&=\ell_i(M_i)=\frac{3V}{(KS)_i}>0.
\end{aligned}
\label{eq:distances}
\end{equation}
The distance ratio is $\sin(2\theta/3)/\sin(\theta/3)=2\cos(\theta/3)$. The second equality follows from the identity
\[
\sum_j S_j\ell_j(X)=3V.
\]
Conversely, the positive numbers in \eqref{eq:distances} are the distances of the interior point with barycentric coordinates $S_jt_iq_{ij}/(3V)$ to the faces. In particular, the three trisecting planes defining $M_i$ do meet in one point.

\begin{lemma}
\label{lem:negative}
Let $K$ be the matrix of a nondegenerate tetrahedron. Then $K$ is strictly negative definite on $\one^\perp$, where $\one=(1,1,1,1)^T$. Its inertia is $(1,3)$; in particular, it is invertible and $\det K<0$.
\end{lemma}

\begin{proof}
Define $f(h)=2\cos(\arccos(-h)/3)$. We claim that
\[
f(h)=\sqrt3-\sum_{m\ge1}c_mh^m,
\]
where $c_m>0$ and $c_1=\frac13$, with absolute convergence on $[-1,1]$. Put $y=\sqrt3-f(h)$. The equation $\ph(f(h))=h$ becomes
\[
h=3y\left(1-\frac{y}{\sqrt3}\right)
       \left(1-\frac{y}{2\sqrt3}\right).
\]
Lagrange inversion gives
\[
c_m=\frac{1}{m3^m}[u^{m-1}]
\left(1-\frac{u}{\sqrt3}\right)^{-m}
\left(1-\frac{u}{2\sqrt3}\right)^{-m}>0.
\]
The branch is analytic for $|h|<1$, since the only possible finite branch values of the cubic are $h=\pm1$. Letting $h\uparrow1$ gives $\sum_m c_m=\sqrt3-1$, which proves the asserted endpoint convergence.

Each $H^{\circ m}$ is the Gram matrix of the tensors $n_i^{\otimes m}$ and is positive semidefinite. For $0\ne z\perp\one$, therefore,
\[
z^TKz=-\sum_{m\ge1}c_mz^TH^{\circ m}z
\le-\frac13z^THz<0.
\]
The last inequality follows from $\ker H=\R S$ and $S>0$. Finally, $\one^TK\one>0$, so the remaining eigenvalue is positive.
\end{proof}

\begin{proposition}[Regularity equations]
\label{prop:regularity}
If the Morley tetrahedron is regular with edge length $L$, set $d_i=t_i/(\sqrt2L)$, $d=(d_1,\dots,d_4)^T$ and $D=\diag(d_1,\dots,d_4)$. Then
\begin{equation}
H=K(4D^2-dd^T)K.
\label{eq:regularity}
\end{equation}
Conversely, suppose $K$ is symmetric with $q_{ii}=1$ and $1<q_{ij}<2$ for $i\ne j$, and that \eqref{eq:normalgram}, \eqref{eq:positivekernel} and \eqref{eq:regularity} hold for some $d>0$. Then $H$ is the normal Gram matrix of a nondegenerate tetrahedron, $K=(2\cos(\theta_{ij}/3))$ for this tetrahedron, and its Morley tetrahedron is regular.
\end{proposition}

\begin{proof}
Take the centroid of the regular Morley tetrahedron as origin, and write its vertices as $m_i$. Then $\sum_i m_i=0$, $m_i\cdot m_j=\frac{L^2}{2}\delta_{ij}-\frac{L^2}{8}$ and $\sum_i m_im_i^T=\frac{L^2}{2}I_3$.
Using \eqref{eq:distances} gives
\[
n_j=\frac{2}{L^2}\sum_i t_iq_{ij}m_i.
\]
Taking scalar products proves \eqref{eq:regularity}.

For the converse, a positive semidefinite matrix $H$ with diagonal $1$, rank $3$ and a positive kernel vector determines four unit normals which positively span $\R^3$. The half-spaces $n_i\cdot X\ge-1$ bound a nondegenerate tetrahedron with these inward normals; this is the simplex case of the standard facet-normal reconstruction, cf.~\cite{Schneider}. Its shape is uniquely determined by $H$: translations and a common dilation account for the four support parameters. The positivity of the kernel also implies that every three normals are independent. The dihedral angles satisfy $\cos\theta_{ij}=-H_{ij}$, and $\ph$ maps $(1,2)$ bijectively and decreasingly onto $(-1,1)$. Since $\ph(2\cos(\theta_{ij}/3))=-\cos\theta_{ij}=H_{ij}=\ph(q_{ij})$, the given $K$ is the matrix of this tetrahedron.

By Lemma~\ref{lem:negative}, $K$ is invertible. Since $HS=0$, equation \eqref{eq:regularity} gives $(4D^2-dd^T)KS=0$, that is, $4d_i^2(KS)_i=d_i\,(d^TKS)$ for every $i$. Hence $KS=c\,(1/d_i)_i$ with $c=d^TKS/4$, and $c>0$ because all entries of $K$ and $S$ are positive. Consequently the distances in \eqref{eq:distances} satisfy $t_i=\alpha d_i$ with $\alpha=3V/c>0$.

The matrix $TK$, where $T=\diag(t_i)$, has the entries $\ell_j(M_i)=t_iq_{ij}$. It is invertible, so the Morley vertices are affinely independent and the Morley tetrahedron is nondegenerate. Let $\mu_i$ be its barycentric coordinates. The affine functions $\ell_j$ and $\sum_it_iq_{ij}\mu_i$ agree at the four Morley vertices, so they are equal, and taking gradients gives $n_j=\sum_it_iq_{ij}\nabla\mu_i$. Hence the Gram matrix of the gradients $\nabla\mu_i$ is
$T^{-1}BT^{-1}$, where $B=K^{-1}HK^{-1}$, and by \eqref{eq:regularity} it equals $\alpha^{-2}(4I_4-\one\one^T)$. This is, up to a factor, the Gram matrix of the barycentric gradients of a regular tetrahedron, so the Morley tetrahedron is regular.
\end{proof}

The proof shows that, for every tetrahedron, $B=K^{-1}HK^{-1}$ is the Gram matrix of the vectors $t_i\nabla\mu_i$, so $B_{ii}>0$. Since the normals of a tetrahedron are proportional to the gradients of its barycentric coordinates, the Morley tetrahedron is regular if and only if
\begin{equation}
\frac{B_{ij}}{\sqrt{B_{ii}B_{jj}}}=-\frac13\ \text{ for all }i\ne j.
\label{eq:normalizedgram}
\end{equation}
Conversely, \eqref{eq:normalizedgram} gives $B=4D^2-dd^T$ with $d_i=\sqrt{B_{ii}/3}$. The polynomial form of \eqref{eq:normalizedgram} is $9B_{ij}^2=B_{ii}B_{jj}$ together with $B_{ij}<0$.

We shall often use the edge formula: if $\{i,j,k,l\}=\{1,2,3,4\}$, then
\begin{equation}
|A_kA_l|=\frac{2S_iS_j\sin\theta_{ij}}{3V}.
\label{eq:edge}
\end{equation}
It follows by expressing the volume in terms of the two face altitudes to their common edge. Also,
\[
\sin^2\theta_{ij}=1-\ph(q_{ij})^2
=\frac{(q_{ij}^2-1)^2(4-q_{ij}^2)}4.
\]

\section{A symmetric two-parameter family}\label{sec:family}

For $\rho>0$ and $0<\sigma<\pi/2$, let $T(\rho,\sigma)=ABCD$ be the tetrahedron with vertices
\[
A=\Bigl(-\frac12,0,0\Bigr),\ B=\Bigl(\frac12,0,0\Bigr),\ C=(0,\rho\cos\sigma,\rho\sin\sigma),\ D=(0,\rho\cos\sigma,-\rho\sin\sigma).
\]
Its edge lengths are $AB=1$, $CD=\mu=2\rho\sin\sigma$ and
\[
AC=AD=BC=BD=\lambda=\sqrt{\tfrac14+\rho^2}.
\]
The reflections $(x,y,z)\mapsto(-x,y,z)$ and $(x,y,z)\mapsto(x,y,-z)$ preserve the vertex set, so by Lemma~\ref{lem:equivariance} they are also symmetries of the Morley tetrahedron. Consequently $|A'C'|=|B'C'|=|A'D'|=|B'D'|$. Define
\[
f(\rho,\sigma)=|A'B'|^2-|C'D'|^2\ \text{ and }\ g(\rho,\sigma)=|A'B'|^2-|A'C'|^2.
\]
The Morley tetrahedron of $T(\rho,\sigma)$ is regular exactly when $f(\rho,\sigma)=g(\rho,\sigma)=0$, and it is isosceles exactly when $f(\rho,\sigma)=0$.

\section{Counterexamples to the original conjectures}
\label{sec:counterexamples}

\subsection{A nonisosceles counterexample}

\begin{theorem}
\label{thm:T1}
There is a nondegenerate tetrahedron $T_1$ with
\[
AB=1,\ AC=AD=BC=BD=\lambda_*,\ CD=\mu_*,
\]
where $1.579392054<\lambda_*<1.579392055$ and $2.351892333<\mu_*<2.351892334$, whose Morley tetrahedron is regular. The tetrahedron $T_1$ is neither regular nor isosceles.
\end{theorem}

\begin{proof}
In Section~\ref{sec:intervals} we show that the regularity equations have exactly one solution in a small box around the values labeled $T_1$, that this solution comes from a nondegenerate tetrahedron, and that the tetrahedron has two reflections interchanging disjoint pairs of vertices. After relabeling and scaling it belongs to the family $T(\rho,\sigma)$ of Section~\ref{sec:family}, with $\rho\approx1.4981586235$ and $\sigma\approx0.9025793893$. The interval computation of Section~\ref{sec:intervals} gives $1.57<\lambda_*<1.59$ and $2.35<\mu_*<2.36$, so $T_1$ is neither regular nor isosceles.

The sharper enclosures in the statement come from the exact description of $T_1$ in Theorem~\ref{sym:thm:classification}, through the root of the sextic \eqref{sym:eq:P}.
\end{proof}

\subsection{An isosceles counterexample}

For $c>0$ consider the tetrahedron with vertices
\[
A=\Bigl(-\frac12,0,0\Bigr),\ B=\Bigl(\frac12,0,0\Bigr),\ C=\Bigl(0,c,\frac12\Bigr),\ D=\Bigl(0,c,-\frac12\Bigr).
\]
Then $AB=CD=1$ and $AC=AD=BC=BD=\sqrt{c^2+1/2}$, so every tetrahedron in this family is isosceles.

\begin{theorem}
\label{thm:T2} Let
\[
c_0^2=\frac{8\sqrt6-3}{90}.
\]
The isosceles tetrahedron $T_2$ of the family above with $c=c_0$ has a regular Morley tetrahedron. Its edges are $AB=CD=1$ and
\[
AC=AD=BC=BD=\lambda_\dagger=\sqrt{\frac{21+4\sqrt6}{45}}\approx0.8272841641,
\]
so $T_2$ is isosceles but not regular.
\end{theorem}

\begin{proof}
Put $r=\sqrt6$ and
\begin{equation}\label{t2:eq:exactpq}
q_0=2r-3\ \text{ and }\ p_0=r-\frac52+\frac12\sqrt{60r-135}.
\end{equation}
The number $p_0$ is the positive root of
\begin{equation}\label{t2:eq:pquadratic}
p^2+(5-2r)p+46-20r=0.
\end{equation}
The quadratic is negative at $1$ and positive at $2$, and its constant term is negative, so $1<p_0<2$. Also $1<q_0<2$ because $2<r<5/2$. Reducing powers of $p_0$ by \eqref{t2:eq:pquadratic} gives $p_0^3=3p_0+470-192r$, and a direct expansion gives $q_0^3=102r-243$. Hence
\begin{equation}\label{t2:eq:phi}
\ph(p_0)=96r-235\ \text{ and }\ \ph(q_0)=117-48r.
\end{equation}

In the coordinates above with $c=c_0$, the inward unit normals of the faces opposite $A$ and $B$ have scalar product
\[
\frac{1/4-c_0^2}{1/4+c_0^2}=96\sqrt6-235=\ph(p_0),
\]
and the same holds for the faces opposite $C$ and $D$. For a face opposite $A$ or $B$ and a face opposite $C$ or $D$ the scalar product is
\[
-\frac{1/4}{1/4+c_0^2}=117-48\sqrt6=\ph(q_0).
\]
Since $\ph$ is strictly decreasing on $(1,2)$, formula \eqref{eq:normalgram} shows that the matrix $K$ of Section~\ref{sec:gram} is
\begin{equation}\label{t2:eq:K}
K=\begin{pmatrix}
1&p_0&q_0&q_0\\p_0&1&q_0&q_0\\q_0&q_0&1&p_0\\q_0&q_0&p_0&1
\end{pmatrix}.
\end{equation}
The row sums of $H=\ph[K]$ vanish, because
\begin{equation}\label{t2:eq:row}
1+\ph(p_0)+2\ph(q_0)=0.
\end{equation}
On $(1,1,-1,-1)^T$ the eigenvalues of $K$ and $H$ are $1+p_0-2q_0$ and $1+\ph(p_0)-2\ph(q_0)=-4\ph(q_0)$. On the plane spanned by $(1,-1,0,0)^T$ and $(0,0,1,-1)^T$ they are $1-p_0$ and $1-\ph(p_0)=(p_0-1)^2(p_0+2)/2$. Put $\tau_0^2=(p_0+2)/8$. Reduction by \eqref{t2:eq:pquadratic} gives
\begin{equation}\label{t2:eq:tau}
(p_0+2)(1+p_0-2q_0)^2=384r-936=-8\ph(q_0).
\end{equation}
Hence $H$ and $\tau_0^2K(4I_4-\one\one^T)K$ agree on all four eigenvectors, the constant vector being annihilated by both. This is \eqref{eq:regularity} with $d_1=\dots=d_4=\tau_0$. Proposition~\ref{prop:regularity} shows that the Morley tetrahedron is regular. Finally
\[
AC^2=c_0^2+\frac12=\frac{21+4\sqrt6}{45}\ne1.
\]
The interval argument in Section~\ref{sec:intervals} shows in addition that this solution is isolated among all tetrahedron shapes.
\end{proof}

\subsection{A curve with isosceles Morley tetrahedra}

\begin{theorem}
\label{thm:curve} There exists a real analytic one-parameter family of nonisosceles tetrahedra whose Morley tetrahedra are isosceles and nonregular.
\end{theorem}

\begin{proof}
Use the equivalent parameters
\[
A=(-a,0,0),\ B=(a,0,0),\ C=(0,c,b),\ D=(0,c,-b),
\]
where $a=1/2$, $b=\rho\sin\sigma$ and $c=\rho\cos\sigma$. By Section~\ref{sec:intervals}, $\partial f/\partial b\ne0$ at $T_1$, so the implicit function theorem gives a real analytic curve $f=0$ through $T_1$.

Section~\ref{sec:intervals} also shows that no other tetrahedron close to $T_1$ has a regular Morley tetrahedron. The numbers $q_{ij}$ of the tetrahedron with parameters $(b,c)$ depend continuously on $(b,c)$, and so do the normalized distances $d_i$ when the Morley tetrahedron is regular; a regular example close to $T_1$ would therefore give a second zero in the uniqueness box. Hence $g\ne0$ at all points of the curve near $T_1$ other than $T_1$ itself, and the Morley tetrahedra at these points are isosceles but not regular. Finally, $AB\ne CD$ near $T_1$, so these tetrahedra are not isosceles.
\end{proof}

By Proposition~\ref{prop:isolocal}, the curve lies on a surface of nonisosceles tetrahedra with a half-turn symmetry and isosceles Morley tetrahedra, and Theorem~\ref{thm:asymmetric} gives such tetrahedra without any symmetry.

\begin{corollary}
The two converse conjectures proposed in \cite{TranActa} are false.
\end{corollary}

\begin{proof}
Theorem~\ref{thm:T1} disproves the regular case of Conjecture~\ref{conj:old1}, and Theorem~\ref{thm:curve} disproves the isosceles case. For $e=f=g=3$ the construction of \cite[Theorem 6]{TranActa} is the Morley construction \cite[Remark~7]{TranActa}. So Conjecture~\ref{conj:old2} contains Conjecture~\ref{conj:old1}, and the same examples disprove it.
\end{proof}

\begin{remark}
The direct theorem of \cite{TranActa}, that $\M(ABCD)$ is isosceles whenever $ABCD$ is, is not affected; see Corollary~\ref{cor:inherited}. Theorem~\ref{thm:T2} shows that the regular case of Conjecture~\ref{conj:old1} fails even among isosceles tetrahedra, although only for one similarity class (Theorem~\ref{t2:thm:classification}).
\end{remark}

\section{Certified existence and local uniqueness}
\label{sec:intervals}

This section contains the computations used in Theorems~\ref{thm:T1}--\ref{thm:curve}. All decimal numbers below are exact rational numbers. The argument is the usual combination of the contraction mapping theorem with interval arithmetic; see \cite{Krawczyk,Neumaier,Rump}.

\subsection{The equations and the three centers}

Use the coordinate order
\[
x=(q_{12},q_{13},q_{14},q_{23},q_{24},q_{34},d_1,d_2,d_3,d_4).
\]
Set $u=Kd$ and, for $i\le j$, define
\begin{equation}
F_{ij}(x)=4\sum_{k=1}^4d_k^2q_{ik}q_{jk}-u_i u_j-\ph(q_{ij}).
\label{eq:tenequations}
\end{equation}
Thus $F(x)=0$ is precisely \eqref{eq:regularity}. Order its components as
\[
(F_{11},F_{12},F_{13},F_{14},F_{22},F_{23},F_{24},F_{33},F_{34},F_{44}).
\]
The three rational centers have the common form
\begin{equation}
x_0=(u_0,u_0,v_0,w_0,u_0,u_0,a_0,b_0,b_0,a_0),
\label{eq:centerform}
\end{equation}
with the entries given in the following table.

\begin{center}
\small
\begin{tabular}{lrrrrr}
\toprule
 & $u_0$ & $v_0$ & $w_0$ & $a_0$ & $b_0$\\
\midrule
Regular & 1.833986596704 & 1.833986596704 & 1.833986596704 & 0.692277635482 & 0.692277635482\\
$T_1$ & 1.844867969052 & 1.648726879892 & 1.892477377433 & 0.675344993308 & 0.697538294418\\
$T_2$ & 1.898979485567 & 1.679329660736 & 1.679329660736 & 0.678171222916 & 0.678171222916\\
\bottomrule
\end{tabular}
\end{center}

Around each center take the closed box $\mathcal I=\{x:\norm{x-x_0}_\infty\le r\}$ with $r=10^{-5}$. The matrix $J_0=DF(x_0)$ is nonsingular at all three centers, and we put $C=J_0^{-1}$. Its entries are rational numbers determined by \eqref{eq:tenequations} and the table. The derivatives are as follows. For $a<b$, let $E^{ab}$ have entries $1$ in positions $(a,b),(b,a)$ and zero elsewhere. Then
\begin{align}
\frac{\partial F_{ij}}{\partial q_{ab}}
={}&4\sum_k d_k^2\bigl(E^{ab}_{ik}q_{jk}+q_{ik}E^{ab}_{jk}\bigr)
-(E^{ab}d)_i u_j-u_i(E^{ab}d)_j
-\ph'(q_{ij})E^{ab}_{ij},\notag\\
\frac{\partial F_{ij}}{\partial d_a}
={}&8d_aq_{ia}q_{ja}-q_{ia}u_j-u_iq_{ja}.
\label{eq:tenjacobian}
\end{align}
These formulas give $J_0$ by rational substitution. On substituting the intervals $[x_{0,k}-r,x_{0,k}+r]$ into \eqref{eq:tenjacobian}, ordinary interval addition and multiplication give a matrix of intervals $J(\mathcal I)$ containing every $DF(x)$ for $x\in\mathcal I$. For all three rows of the table, this calculation gives
\begin{equation}
\norm{I-CJ(\mathcal I)}_\infty<\frac7{100}\ \text{ and }\ \norm{CF(x_0)}_\infty<10^{-11}.
\label{eq:contractionbounds}
\end{equation}
For a matrix of intervals, the norm is the largest row sum of the largest absolute endpoints of the entries. Thus \eqref{eq:contractionbounds} is a finite list of comparisons between rational numbers determined by the table and the formulas above.

The map $x\mapsto x-CF(x)$ is a contraction on $\mathcal I$, and its displacement from the center is at most
\[
10^{-11}+\frac7{100}r<r.
\]
It therefore has a unique fixed point $x_*$ in $\mathcal I$, and since $C$ is invertible, $x_*$ is the unique zero of $F$ in $\mathcal I$. Also
\begin{equation}
\norm{x_*-x_0}_\infty\le\frac{\norm{CF(x_0)}_\infty}{1-7/100}<1.1\cdot10^{-11}.
\label{eq:sharpbox}
\end{equation}

The same interval operations, using the determinant formula and Cramer's rule, give throughout each box
\begin{equation}
1<q_{ij}<2,\ \ d_i>0,\ \ \det K<0\ \text{ and }\ K^{-1}(1/d_i)_i>0.
\label{eq:intervalgeometry}
\end{equation}
At the zero of $F$, the matrix $H=K(4D^2-dd^T)K$ is positive semidefinite of rank $3$. The last vector in \eqref{eq:intervalgeometry} belongs to its kernel and has positive coordinates. Proposition~\ref{prop:regularity} therefore gives a nondegenerate tetrahedron whose Morley tetrahedron is regular.

Every permutation of the vertices preserves the regular center. The centers of $T_1$ and $T_2$ are preserved by $(1\,4)$ and $(2\,3)$, and that of $T_2$ also by $(1\,2)(3\,4)$. These permutations preserve the boxes and the equations, so the unique solution in each box has the same symmetries. Such a symmetry of the solution is a symmetry of the tetrahedron: if a permutation $\pi$ fixes $K$, it fixes $H=\ph[K]$ and hence the positive kernel vector $S$, which is unique up to a positive factor; by \eqref{eq:edge} the vertex permutation $\pi$ then preserves all six edge lengths, so it is induced by an isometry, and a transposition fixing two vertices is induced by a reflection. In particular, up to similarity and relabeling, each of the three tetrahedra is the only example in its neighborhood.

\subsection{The edge lengths and the analytic curve}

Relabel the vertices so that the reflections $(2\,3)$ and $(1\,4)$ become $(1\,2)$ and $(3\,4)$, and put $p=q_{12}$ and $s=q_{34}$. Denote the common value of the four cross entries by $q$, and set $h_p=\ph(p)$ and $h_s=\ph(s)$. The original tetrahedron has vertices $A=(-a,0,0)$, $B=(a,0,0)$, $C=(0,c,b)$ and $D=(0,c,-b)$, where $a=1/2$ and
\begin{equation}
c^2=a^2\,\frac{1-h_p}{1+h_p}\ \text{ and }\ b^2=c^2\,\frac{1+h_s}{1-h_s}.
\label{eq:reconstruction}
\end{equation}
For $T_1$, this means $p=q_{23}$ and $s=q_{14}$ in \eqref{eq:centerform}; for $T_2$, take $p=q_{14}=q_{23}$. The edge $CD$ has length $2b$, and each cross edge has length $\sqrt{a^2+b^2+c^2}$. By \eqref{eq:sharpbox}, the true values of $p$ and $s$ lie within $1.1\cdot10^{-11}$ of the entries of the table. Substituting these intervals gives $1.57<\lambda_*<1.59$ and $2.35<\mu_*<2.36$ for $T_1$, and $0.82<\lambda_\dagger<0.84$ for $T_2$. In the latter case $p=s$, and \eqref{eq:reconstruction} gives $b=a$. These are the bounds used in Theorem~\ref{thm:T1}; for $T_2$ they agree with the exact value in Theorem~\ref{thm:T2}.

We finally check the derivative used in Theorem~\ref{thm:curve}. Put $U=(a^2+c^2)^{1/2}$ and $W=(b^2+c^2)^{1/2}$. The two distinct face areas are $bU$ and $aW$. Formula \eqref{eq:distances} and the two reflections give
\[
\begin{aligned}
t_A&=\frac{2abc}{(1+p)bU+2qaW},&
t_C&=\frac{2abc}{2qbU+(1+s)aW},\\
|A'B'|^2&=\frac{(p-1)^2t_A^2U^2}{c^2},&
|C'D'|^2&=\frac{(s-1)^2t_C^2W^2}{c^2}.
\end{aligned}
\]
With $a,c$ fixed, the normal Gram entries are
\[
h_p=\frac{a^2-c^2}{a^2+c^2},\ \ h_s=\frac{b^2-c^2}{b^2+c^2},\ \ h_q=-\frac{ab}{UW},
\]
and their derivatives with respect to $b$ give
\[
p_b=0,\ \ s_b=\frac{4bc^2}{(b^2+c^2)^2\ph'(s)},\ \ q_b=-\frac{ac^2}{UW^3\ph'(q)}.
\]
We differentiate the formulas for $|A'B'|^2$ and $|C'D'|^2$ with respect to $b$ and substitute the intervals given by \eqref{eq:sharpbox} for $T_1$. This gives
\[
-0.001840<\frac{\partial f}{\partial b}<-0.001839.
\]
Only rational operations and square roots occur, and a square root of a positive rational interval is enclosed by rational endpoints whose squares bound the interval. In particular $\partial f/\partial b\ne0$, as needed in the proof of Theorem~\ref{thm:curve}.

\part{New conjectures}
\label{part:conjectures}

The counterexamples suggest a weaker converse, in which regularity of the Morley tetrahedron forces a symmetry of the original tetrahedron but not regularity. We state this and related conjectures here; partial results are proved in Part~\ref{part:results}.

\section{What the counterexamples suggest}
\label{sec:suggest}

In both exceptional examples there is a pair of opposite edges, $AB$ and $CD$, such that the other four edges have a common length. The plane through $C$, $D$ and the midpoint of $AB$ is then the perpendicular bisector plane of $AB$, and the reflection in it interchanges $A$ and $B$ and fixes $C$ and $D$. In the same way the perpendicular bisector plane of $CD$ contains $A$ and $B$. The two planes are perpendicular, because $AB\perp CD$. Conversely, two such reflections force the four cross edges to be equal. In $T_2$ the edges $AB$ and $CD$ are equal as well; in $T_1$ they are not.

Numerical experiments point in the same direction. We minimized the differences of the six Morley edge lengths by least squares, starting from several hundred random tetrahedra. Every run that converged ended, up to similarity and relabeling, at one of the following:
\begin{enumerate}[label=(\roman*)]
\item the regular tetrahedron;
\item the nonisosceles tetrahedron $T_1$;
\item the isosceles nonregular tetrahedron $T_2$.
\end{enumerate}
Such searches cannot exclude further solutions, and they are not used in any proof. Section~\ref{sec:intervals} proves existence and uniqueness near the three examples, but says nothing about the rest of the shape space.

\section{The conjectures}
\label{sec:conjectures}

\begin{conjecture}[Four equal cross edges]
\label{conj:four} Let $ABCD$ be a nondegenerate tetrahedron. If its Morley tetrahedron is regular, then, after relabeling the vertices,
\[
AC=AD=BC=BD.
\]
In other words, there is a pair of opposite edges such that the remaining four edges, which form a skew quadrilateral, all have the same length. Equivalently, the tetrahedron has two reflection symmetries interchanging the endpoints of these two opposite edges.
\end{conjecture}

The numerical evidence suggests a complete classification.

\begin{conjecture}[Classification]\label{conj:classification}
Up to similarity and relabeling, the only tetrahedra with a regular Morley tetrahedron are the regular tetrahedron and the two exceptional similarity classes represented by $T_1$ and $T_2$.
\end{conjecture}

A tetrahedron with four equal cross edges has two reflection symmetries. The next conjecture asks only for one symmetry of any kind.

\begin{conjecture}[Symmetry conjecture]\label{conj:symmetry}
If the Morley tetrahedron of $ABCD$ is regular, then some isometry other than the identity maps the vertex set $\{A,B,C,D\}$ onto itself.
\end{conjecture}

\begin{conjecture}[One pair of equal opposite edges]
\label{conj:onepair}
Let $ABCD$ be a nondegenerate tetrahedron satisfying $AB=CD$. If its Morley tetrahedron is regular, then $ABCD$ has four equal cross edges with respect to some pair of opposite edges, which need not be the pair $\{AB,CD\}$.
\end{conjecture}

Conjecture~\ref{conj:onepair} is a special case of Conjecture~\ref{conj:four}. Its conclusion does not say that the two remaining edges are equal: $T_1$ can be labeled so that $AB=CD$, but it is not isosceles.

In the notation of Section~\ref{sec:gram}, Conjecture~\ref{conj:four} can be stated purely in terms of matrices.

\begin{conjecture}[Matrix form]
\label{conj:positivity}
Let $K$ be symmetric with $q_{ii}=1$ and $1<q_{ij}<2$ for $i\ne j$. Suppose that $H=\ph[K]$ satisfies \eqref{eq:positivekernel}, and that $B=K^{-1}HK^{-1}$ satisfies \eqref{eq:normalizedgram}. Then, after a permutation of the indices,
\[
q_{13}=q_{14}=q_{23}=q_{24}.
\]
\end{conjecture}

Conjecture~\ref{conj:positivity} is equivalent to Conjecture~\ref{conj:four}. The displayed equalities make $H$ invariant under $(1\,2)$ and $(3\,4)$, so its positive kernel vector satisfies $S_1=S_2$ and $S_3=S_4$, and formula \eqref{eq:edge} gives four equal cross edges. Conversely, four equal cross edges give two reflections, and the displayed equalities follow by equivariance. Proposition~\ref{prop:regularity} provides the passage from matrices to tetrahedra.

Finally, we state the analogue for the construction of \cite[Theorem 6]{TranActa}.

\begin{conjecture}[Four equal cross edges for $p$-secting planes]
\label{conj:psector} Let $\M_{e,f,g}(ABCD)$ denote the tetrahedron $A'B'C'D'$ of \cite[Theorem 6]{TranActa}, for numbers $e,f,g>1$. If $\M_{e,f,g}(ABCD)$ is regular, then, after relabeling the vertices, $AC=AD=BC=BD$.
\end{conjecture}

\subsection*{Relations between the conjectures}

The results of Part~\ref{part:results} give the following relations:
\begin{itemize}
\item Conjectures~\ref{conj:four} and~\ref{conj:classification} are equivalent (Corollary~\ref{sym:cor:equivalence});
\item Conjectures~\ref{conj:four} and~\ref{conj:symmetry} are equivalent (Corollary~\ref{cor:symmetric});
\item Conjectures~\ref{conj:four} and~\ref{conj:positivity} are equivalent (see above);
\item under $AB=CD$, Conjecture~\ref{conj:onepair} is equivalent to the equality of the dihedral angles along $AB$ and $CD$ (Proposition~\ref{prop:cotangent}).
\end{itemize}

\part{Results, possible approaches, and open problems}
\label{part:results}

In this part we prove the conjectures of Part~\ref{part:conjectures} under additional hypotheses. Section~\ref{sec:summary} lists what is known, and Sections~\ref{tet:directions} and~\ref{sec:openproblems} discuss possible approaches and open problems.

\section{Bounds and a reflection criterion}
\label{sec:bounds}

We begin with bounds on the distances $t_i$ of \eqref{eq:distances}.

\begin{proposition}
\label{prop:distancebounds}
If $r_{\mathrm{in}}$ is the inradius of the original tetrahedron, then
\[
\frac{r_{\mathrm{in}}}{\sqrt3}<t_i<\frac{2r_{\mathrm{in}}}{3}\ \text{ and }\ \frac{\max_i t_i}{\min_i t_i}<\frac2{\sqrt3}.
\]
These inequalities do not require the Morley tetrahedron to be regular.
\end{proposition}

\begin{proof}
Normalize the face areas to sum to $1$ and put $y_i=(KS)_i$. Since $HS=0$, we have $\sum_jS_jq_{ij}^3=3y_i$. Strict convexity gives $y_i^3<3y_i$, hence $y_i<\sqrt3$. On the other hand,
\[
x^3-7x+6=(x-1)(x-2)(x+3)<0\ \text{ for }1<x<2.
\]
Applying this to the off-diagonal terms gives $3y_i<7y_i-6$, so $y_i>3/2$. Now $t_i=r_{\mathrm{in}}/y_i$.
\end{proof}

\begin{proposition}[A reflection criterion]
\label{prop:defect}
If the Morley tetrahedron is regular with edge length $L$, then
\begin{equation}
\varepsilon_{ij}:=q_{ij}+2-\frac{(t_i+t_j)^2}{L^2}\ge0.
\label{eq:defect}
\end{equation}
Equality holds if and only if the original tetrahedron has a reflection interchanging $A_i,A_j$ and fixing the other two vertices. We call $\varepsilon_{ij}$ the defect of the pair $\{i,j\}$. Moreover,
\begin{equation}
\frac{(t_i+t_j)^2}{L^2(q_{ij}+2)}
+\frac{(t_i-t_j)^2}{L^2(2-q_{ij})}\le1.
\label{eq:ellipse}
\end{equation}
In particular, $d_i\le1/\sqrt2$.
\end{proposition}

\begin{proof}
Set $v=M_j-M_i$ and $q=q_{ij}$. Formula \eqref{eq:distances} gives
\[
(n_i-n_j)\cdot v=(q-1)(t_i+t_j)\ \text{ and }\ (n_i+n_j)\cdot v=(q+1)(t_j-t_i).
\]
The two normal sums are orthogonal and have squared lengths
\[
\|n_i-n_j\|^2=(q-1)^2(q+2)\ \text{ and }\ \|n_i+n_j\|^2=(2-q)(q+1)^2.
\]
Cauchy's inequality proves \eqref{eq:defect}; summing the squared orthogonal projections of $v$ proves \eqref{eq:ellipse}. Minimizing its left side with respect to $t_j$ gives $t_i^2/L^2\le1$, or $d_i\le1/\sqrt2$.

If $\varepsilon_{ij}=0$, then $v$ is parallel to $n_i-n_j$, and the second scalar product implies $t_i=t_j$. The plane $\ell_i=\ell_j$ is the perpendicular bisector of $M_iM_j$. The other two Morley vertices lie in that plane, whence $q_{ik}=q_{jk}$ for $k\ne i,j$. Also $n_k\cdot v=0$ for these $k$. Reflection in the plane therefore interchanges $F_i,F_j$ and preserves the other two faces. It induces exactly the stated permutation of the original vertices. Conversely, such a reflection is preserved by the Morley construction, and $v$ is parallel to $n_i-n_j$, giving equality in Cauchy's inequality.
\end{proof}

When equality holds, the distances $t_i$ and $t_j$ can be computed.

\begin{corollary}\label{cor:reflectiondistance}
Suppose that the Morley tetrahedron is regular with edge length $L$. If the original tetrahedron has a reflection interchanging $A_i,A_j$ and fixing the other two vertices, then $t_i=t_j=L\cos(\theta_{ij}/6)$, or equivalently $d_i^2=d_j^2=(q_{ij}+2)/8$.
\end{corollary}

\begin{proof}
By Proposition~\ref{prop:defect}, $\varepsilon_{ij}=0$ and $t_i=t_j$. Hence $4t_i^2=L^2(q_{ij}+2)$. Since $q_{ij}+2=2\cos(\theta_{ij}/3)+2=4\cos^2(\theta_{ij}/6)$, the claim follows.
\end{proof}

The relation $8\tau^2=p+2$ in the proof of Theorem~\ref{t2:thm:classification}, and the relation $8a^2=p+2$ in the proofs of Theorems~\ref{thm:reflection} and~\ref{sym:thm:classification}, are instances of this identity.

\section{The isosceles case and equal face distances}
\label{sec:isosceles}

\begin{theorem}
\label{thm:disphenoid}
Let $ABCD$ be an isosceles tetrahedron. If its Morley tetrahedron is regular, then, after relabeling,
\[
AC=AD=BC=BD.
\]
Thus two of the three lengths of the pairs of opposite edges coincide. The remaining two edges are also equal, because $ABCD$ is isosceles.
\end{theorem}

\begin{proof}
All four faces of an isosceles tetrahedron are congruent. The three rotations through $\pi$ that permute the vertices give
\[
K=\begin{pmatrix}
1&a&b&c\\ a&1&c&b\\ b&c&1&a\\ c&b&a&1
\end{pmatrix}
\]
and $S_1=S_2=S_3=S_4$, where $a,b,c>1$. Put $\sigma=1+a+b+c$. By \eqref{eq:distances}, all $t_i$ are equal, so all $d_i$ equal some $\tau>0$.

On the vectors $(1,1,-1,-1)^T$, $(1,-1,1,-1)^T$ and $(1,-1,-1,1)^T$, the eigenvalues of $K$ are $2a+2-\sigma$, $2b+2-\sigma$ and $2c+2-\sigma$, respectively. Since $H\one=0$, the corresponding eigenvalues of $H$ are $2+3a-a^3$, $2+3b-b^3$ and $2+3c-c^3$. Equation \eqref{eq:regularity} therefore gives
\[
2+3u-u^3=4\tau^2(2u+2-\sigma)^2\ \text{ for }u=a,b,c.
\]
If $a,b,c$ were distinct, they would be all three roots of the monic cubic
\[
X^3-3X-2+4\tau^2(2X+2-\sigma)^2.
\]
Vi\`ete's formula would give $a+b+c=-16\tau^2<0$, contrary to $a,b,c>1$. Hence two of $a,b,c$ coincide. Together with the equality of the face areas, formula \eqref{eq:edge} now gives four equal cross edges.
\end{proof}

In fact it suffices to know that the four distances $t_i$ are equal.

\begin{theorem}
\label{thm:equaldistances}
Suppose the Morley tetrahedron is regular and $t_1=t_2=t_3=t_4$. Then the original tetrahedron is isosceles and has four equal cross edges after relabeling.
\end{theorem}

\begin{proof}
Write
\[
\begin{aligned}
q_{12}&=a+x,& q_{34}&=a-x,\\
q_{13}&=b+y,& q_{24}&=b-y,\\
q_{14}&=c+z,& q_{23}&=c-z.
\end{aligned}
\]
Here $a,b,c>1$. In the orthonormal Hadamard basis, with the constant vector first, $K$ becomes
\[
\widehat K=\begin{pmatrix}
\lambda_0&x&y&z\\
x&\lambda_1&-z&-y\\
y&-z&\lambda_2&-x\\
z&-y&-x&\lambda_3
\end{pmatrix},
\]
where $\lambda_0=1+a+b+c$, $\lambda_1=1+a-b-c$, $\lambda_2=1-a+b-c$ and $\lambda_3=1-a-b+c$.
By Lemma~\ref{lem:negative}, $\mu_i=-\lambda_i>0$ and
\begin{equation}
A=\begin{pmatrix}\mu_1&z&y\\z&\mu_2&x\\y&x&\mu_3\end{pmatrix}\succ0.
\label{eq:equalA}
\end{equation}

Let $d_i=\tau>0$. In this basis, \eqref{eq:regularity} reads
\[
\widehat H=4\tau^2\widehat K\diag(0,1,1,1)\widehat K.
\]
The four diagonal entries of $H$ are equal. Equivalently,
\[
\widehat H_{01}+\widehat H_{23}
=\widehat H_{02}+\widehat H_{13}
=\widehat H_{03}+\widehat H_{12}=0.
\]
Multiplication gives
\begin{equation}
\begin{aligned}
(\mu_2+\mu_3-\mu_1)x&=yz,\\
(\mu_1+\mu_3-\mu_2)y&=xz,\\
(\mu_1+\mu_2-\mu_3)z&=xy.
\end{aligned}
\label{eq:equalxyz}
\end{equation}

Suppose first that $xyz\ne0$, and denote the three coefficients on the left by $\eta_1,\eta_2,\eta_3$. They have the same sign, and their sum is positive, so all are positive. Also $x^2=\eta_2\eta_3$, $y^2=\eta_1\eta_3$, $z^2=\eta_1\eta_2$ and $xyz=\eta_1\eta_2\eta_3$.
Since $\mu_1=(\eta_2+\eta_3)/2$ and cyclically,
\[
\det A=-\frac38\left(\sum_{i\ne j}\eta_i^2\eta_j
-6\eta_1\eta_2\eta_3\right)\le0
\]
by the inequality between the arithmetic and geometric means. This contradicts \eqref{eq:equalA}. Thus one of $x,y,z$ vanishes, and \eqref{eq:equalxyz} implies that at least two vanish. Relabel so that $x=y=0$.

Assume $z\ne0$. The last equation in \eqref{eq:equalxyz} gives
\begin{equation}
3c=a+b+1.
\label{eq:equalc}
\end{equation}
Now compute $\widehat H_{11}-\widehat H_{22}$ and $\widehat H_{03}$ in two ways, from \eqref{eq:normalgram} and from \eqref{eq:regularity}. This gives
\[
(a-b)(3-a^2-ab-b^2)=8\tau^2(a-b)(\lambda_1+\lambda_2)\ \text{ and }\ z(3-3c^2-z^2)=8\tau^2z\lambda_3.
\]
By \eqref{eq:equalc}, $\lambda_1+\lambda_2=2-2c=\lambda_3$. So if $a\ne b$, then
\[
a^2+ab+b^2=3c^2+z^2.
\]
Put $m=(a+b)/2>1$ and $e=(a-b)/2$. The preceding equality and \eqref{eq:equalc} yield
\[
z^2-\lambda_1\lambda_2
=\frac{(m-1)(11m+7)}9+5e^2>0,
\]
contradicting the positive principal minor $\lambda_1\lambda_2-z^2$ in \eqref{eq:equalA}.

It remains to consider $a=b$. Set $r=c-1>0$ and $w=z^2/r^2>0$. Then $a=b=1+\frac32r$, $\lambda_1=\lambda_2=-r$ and $\lambda_3=-2r$.
The trace equation gives $\tau^2=1/[3r^2(2+w)]$. Here $\ph(a)=1-27r^2(r+2)/16$, and the two ways of computing $\widehat H_{00}-\widehat H_{33}$ and $\widehat H_{03}$ give $4\ph(a)=4\tau^2(z^2-4r^2)$ and $z(3-3c^2-z^2)=8\tau^2z\lambda_3$. These equations become
\[
r^2(r+2)=\frac{32(5+w)}{81(2+w)}\ \text{ and }\ 3r^2(2+w)\bigl(6+r(3+w)\bigr)=16.
\]
Eliminating $r^2$ gives
\[
r(2w^2+16w+3)+12w+6=0,
\]
which is impossible. Therefore $z=0$ as well.

We have proved equality of the three pairs of opposite entries of $K$. Its row sums are constant, and \eqref{eq:regularity} gives $H\one=0$. Since the kernel of $H$ has dimension one, the four face areas are equal. Formula \eqref{eq:edge} proves that the original tetrahedron is isosceles, and Theorem~\ref{thm:disphenoid} completes the proof.
\end{proof}

\subsection{Classification of the isosceles solutions}
\begin{theorem}\label{t2:thm:classification}
Up to similarity and relabeling, exactly two isosceles tetrahedra have
a regular Morley tetrahedron: the regular tetrahedron and the tetrahedron
$T_2$ with $AB=CD=1$ and $AC=AD=BC=BD=\sqrt{(21+4\sqrt6)/45}$.
\end{theorem}
\begin{proof}
By Theorem~\ref{thm:disphenoid}, after relabeling we may write
\begin{equation}\label{t2c:eq:K}
 K=\begin{pmatrix}
 1&p&q&q\\p&1&q&q\\q&q&1&p\\q&q&p&1
 \end{pmatrix},\ \text{ where }\ 1<p,q<2.
\end{equation}

The kernel condition is
\begin{equation}\label{t2c:eq:row}
 1+\ph(p)+2\ph(q)=0.
\end{equation}
On $(1,1,-1,-1)^T$, the eigenvalues of $K,H$ are
$1+p-2q,-4\ph(q)$. On the two-dimensional subspace spanned by
$(1,-1,0,0)^T,(0,0,1,-1)^T$, they are $1-p,1-\ph(p)$.
Since
\[
 1-\ph(p)=\frac{(p-1)^2(p+2)}2,
\]
regularity forces
\begin{equation}\label{t2c:eq:tau}
 \tau^2=\frac{p+2}{8}\ \text{ and }\ (p+2)(1+p-2q)^2=-8\ph(q).
\end{equation}

Equations \eqref{t2c:eq:row} and \eqref{t2c:eq:tau} give
\begin{align*}
 A(p,q)&=p^3+2q^3-3p-6q-2=0,\\
 B(p,q)&=(p+2)(p+1-2q)^2+12q-4q^3=0.
\end{align*}
Their resultant in $p$ has the exact factorization
\begin{equation}\label{t2:eq:resultant}
 \Res_p(A,B)
 =-8q^2(q^2+6q-15)^2(3q^3-9q-2).
\end{equation}
It can be checked by computing the determinant of the Sylvester
matrix of the two cubics (script \texttt{check\_T2.py}).

Since $1<q<2$, there are only two possibilities.
If $3q^3-9q-2=0$, then $\ph(q)=-1/3$, and
\eqref{t2c:eq:row} gives $\ph(p)=-1/3$. The function $\ph$ is strictly
decreasing on $(1,2)$, so $p=q$; the original tetrahedron is regular.

Otherwise $q^2+6q-15=0$, and its only root in $(1,2)$ is $q_0=2\sqrt6-3$.
By \eqref{t2c:eq:row} and the strict monotonicity of $\ph$, the value
of $p$ is then uniquely determined; by \eqref{t2:eq:phi} it is the
number $p_0$ of \eqref{t2:eq:exactpq}. The proof of
Theorem~\ref{thm:T2} shows that this pair is realized by the
tetrahedron $T_2$.
\end{proof}

\section{Symmetry results}\label{sec:symmetry}

\begin{theorem}[Two pairs of equal opposite edges]
\label{thm:twopairs}
Suppose the Morley tetrahedron of a nondegenerate tetrahedron is regular. If the original tetrahedron has two pairs of equal opposite edges, then it has four equal cross edges after relabeling.
\end{theorem}

\begin{proof}
Relabel so that $A_1A_3=A_2A_4$ and $A_1A_4=A_2A_3$. The permutation $(1\,2)(3\,4)$ preserves all six edge lengths, hence is induced by an isometry. Equivariance gives
\begin{equation}
\begin{gathered}
K=\begin{pmatrix}1&p&q&r\\p&1&r&q\\q&r&1&s\\r&q&s&1\end{pmatrix},\\
d=(a,a,b,b)^T,\ \ S=(\alpha,\alpha,\beta,\beta)^T,
\end{gathered}
\label{eq:twopairsK}
\end{equation}
with $1<p,q,r,s<2$ and $a,b,\alpha,\beta>0$. On the plus and minus eigenspaces of $(1\,2)(3\,4)$, the blocks of $K$ are
\[
K_+=\begin{pmatrix}1+p&q+r\\q+r&1+s\end{pmatrix}\ \text{ and }\ K_-=\begin{pmatrix}1-p&q-r\\q-r&1-s\end{pmatrix}.
\]
Equation \eqref{eq:regularity} becomes
\begin{equation}
H_+=2K_+\begin{pmatrix}a\\-b\end{pmatrix}\begin{pmatrix}a&-b\end{pmatrix}K_+\ \text{ and }\ H_-=4K_-\diag(a^2,b^2)K_-.
\label{eq:twopairsblocks}
\end{equation}

We show that $p=s$ or $q=r$. Otherwise, interchange the two vertex pairs so that $p>s$, and put $P=p-1$, $R=s-1$, $h=q-r$, $k=q+r$, $u=P+R$, $v=PR$ and $w=h^2/(PR)$.
Then $0<R<P<1$, $h\ne0$ and $h^2<PR$ by Lemma~\ref{lem:negative}. The diagonal entries of the minus equation in \eqref{eq:twopairsblocks} give
\[
P^2(P+3)=8a^2P^2+8b^2h^2\ \text{ and }\ R^2(R+3)=8b^2R^2+8a^2h^2.
\]
Solving,
\begin{align}
8a^2&=\frac{R^2\{P^2(P+3)-h^2(R+3)\}}{P^2R^2-h^4},\notag\\
8b^2&=\frac{P^2\{R^2(R+3)-h^2(P+3)\}}{P^2R^2-h^4}.
\label{eq:twopairsab}
\end{align}
In particular,
\begin{equation}
0<w<\frac{R(R+3)}{P(P+3)}<1.
\label{eq:wbound}
\end{equation}
The off-diagonal minus equation, divided by $h$, gives
\begin{equation}
k^2=4+\frac{4(u^2+3u-2v)}{3(1+w)}-\frac{vw}{3}.
\label{eq:ksquared}
\end{equation}
The two diagonal plus equations are
\[
[a(P+2)-bk]^2=\frac{(1-P)(P+2)^2}{4}\ \text{ and }\ [ak-b(R+2)]^2=\frac{(1-R)(R+2)^2}{4}.
\]
Divide them by $P+2$ and $R+2$ respectively and subtract; this cancels the terms with $ab$. Let $\Delta=0$ be the resulting equation, where
\[
\Delta=a^2(P+2)-b^2(R+2)+k^2\Bigl(\frac{b^2}{P+2}-\frac{a^2}{R+2}\Bigr)-\frac{(1-P)(P+2)-(1-R)(R+2)}4 .
\]
Substituting \eqref{eq:twopairsab} and \eqref{eq:ksquared}, one finds
\[
24v(P+2)(R+2)(1-w)(1+w)^2\,\Delta=-(P-R)\,\mathcal E(u,v,w),
\]
where $\mathcal E(u,v,w)$ is the right-hand side of \eqref{eq:Ezero} below. Since $P\ne R$, $v>0$ and $0<w<1$ by \eqref{eq:wbound}, we obtain
\begin{equation}
\begin{aligned}
0={}&[2u(2u+3-3w)-v(w^2+4w+11)]\\
&\cdot[v(u+5+w(u+1))+2wu(u+3)]\\
&-6(u+1)(v+2u+4)v(1-w)(1+w)^2.
\end{aligned}
\label{eq:Ezero}
\end{equation}

The matrix $H_+$ has positive diagonal entries and the positive kernel vector $(\alpha,\beta)^T$, so its off-diagonal entry satisfies $\ph(q)+\ph(r)<0$. Since
\[
\ph(q)+\ph(r)=\frac{k(12-k^2-3h^2)}8,
\]
equation \eqref{eq:ksquared} makes this equivalent to
\begin{equation}
\mathcal B(u,v,w):=u^2+3u-6-6w+2v(w^2+w-1)>0.
\label{eq:Bpositive}
\end{equation}
The polynomial inequality in Appendix~\ref{app:bernstein} shows that \eqref{eq:wbound}, \eqref{eq:Ezero} and \eqref{eq:Bpositive} are incompatible. Thus $p=s$ or $q=r$.

If $q=r$, formula \eqref{eq:edge} and the paired areas in \eqref{eq:twopairsK} give the equality of the four cross edges. If $p=s$, the row sums of $H$ are equal. Their common value is zero, because $HS=0$ and $\one^TS>0$. Hence $S$ is constant, and \eqref{eq:edge} makes the tetrahedron isosceles. Theorem~\ref{thm:disphenoid} applies.
\end{proof}

\begin{theorem}[A single reflection]
\label{thm:reflection}
Suppose the Morley tetrahedron is regular. If the original tetrahedron has a reflection interchanging $A_1,A_2$ and fixing $A_3,A_4$, then it also has a reflection interchanging $A_3,A_4$ and fixing $A_1,A_2$. Consequently,
\[
A_1A_3=A_1A_4=A_2A_3=A_2A_4.
\]
\end{theorem}

\begin{proof}
The first reflection gives $d=(a,a,b,c)^T$ and
\[
K=\begin{pmatrix}1&p&q&r\\p&1&q&r\\q&q&1&s\\r&r&s&1\end{pmatrix}.
\]
Appendix~\ref{app:reflection} excludes $b\ne c$ by exact polynomial elimination and the positivity of the face areas. If $b=c$, the identity
\[
(n_3+n_4)\cdot(n_3-n_4)=0
\]
in the coordinates used there gives $(q+r-(1+s)b/a)(q-r)=0$. The first factor cannot vanish: it would make $n_3+n_4=0$, contrary to $1<s<2$. Thus $q=r$. The matrix $H$ and its positive kernel are then also invariant under $(3\,4)$. The corresponding vertex permutation is an isometry by \eqref{eq:edge}, and is a reflection because it fixes two vertices and interchanges the other two.
\end{proof}

In terms of Proposition~\ref{prop:defect}, Theorem~\ref{thm:reflection} says that $\varepsilon_{12}=0$ implies $\varepsilon_{34}=0$. To prove Conjecture~\ref{conj:four} it would therefore be enough to show that one of the six defects $\varepsilon_{ij}$ vanishes. Neither Theorem~\ref{thm:twopairs} nor Theorem~\ref{thm:reflection} covers the case of only one pair of equal opposite edges.

\subsection{Classification with four equal cross edges}\label{sym:section}

\begin{theorem}
\label{sym:thm:classification}
Suppose that $AC=AD=BC=BD$ and that $\M(ABCD)$ is regular.
Up to similarity and relabeling, precisely the following three
possibilities occur:
\begin{enumerate}
\item the regular tetrahedron;
\item the isosceles tetrahedron $T_2$ with $AB=CD=1$ and $AC=AD=BC=BD=\sqrt{(21+4\sqrt6)/45}$;
\item a unique nonisosceles class $T_1$, with normalized edges $AB=1$, $CD=\mu_*$ and $AC=AD=BC=BD=\lambda_*$, where
\[
\lambda_*=1.579392054290007551\ldots\ \text{ and }\ \mu_*=2.351892333661864551\ldots.
\]
An exact algebraic description of $T_1$ is given below.
\end{enumerate}
\end{theorem}

The edge assumption gives the reflections $(1\,2)$ and $(3\,4)$,
which are inherited by the Morley construction. Consequently
\begin{equation}
K=\begin{pmatrix}
1&p&q&q\\p&1&q&q\\q&q&1&s\\q&q&s&1
\end{pmatrix}\ \text{ and }\ d=(a,a,b,b)^T,
\label{sym:eq:two}
\end{equation}
where $1<p,q,s<2$ and $a,b>0$.
Place the Morley vertices at $M_i=Le_i/(2\sqrt2)$, with the vectors $e_i$ of Appendix~\ref{app:reflection}; the proof of Proposition~\ref{prop:regularity} then gives $n_i=\sum_jq_{ij}d_je_j$. Write $U=a(1+p)-2bq$ and $V=2aq-b(1+s)$.
The four normals are
\[
\begin{aligned}
n_1&=(\sqrt2a(1-p),0,U),&
n_2&=(-\sqrt2a(1-p),0,U),\\
n_3&=(0,\sqrt2b(1-s),V),&
n_4&=(0,-\sqrt2b(1-s),V).
\end{aligned}
\]
Comparison of $|n_1-n_2|^2$ and $|n_3-n_4|^2$ gives
\begin{equation}
8a^2=p+2\ \text{ and }\ 8b^2=s+2.
\label{sym:eq:ab}
\end{equation}
Their unit lengths and their cross scalar products give
\begin{equation}
U^2=\frac{(p+1)^2(2-p)}4,\ \ V^2=\frac{(s+1)^2(2-s)}4\ \text{ and }\ UV=\ph(q).
\label{sym:eq:UV}
\end{equation}

The positive area vector has the form $(\alpha,\alpha,\beta,\beta)$,
so $\alpha U+\beta V=0$. Thus $UV<0$ and $q>\sqrt3$.
Also $a(1+p)<3/\sqrt2$ and $2bq>2\sqrt{3/8}\,\sqrt3=3/\sqrt2$. Hence $U<0$ and $V>0$; these signs fix the square roots taken below.

Set
\begin{equation}
x=\sqrt{\frac{2(2-p)}{p+2}}\ \text{ and }\ y=\sqrt{\frac{2(2-s)}{s+2}}.
\label{sym:eq:xy}
\end{equation}
Then $0<x,y<\sqrt{2/3}$ and
\[
p=\frac{2(2-x^2)}{2+x^2},\ \ s=\frac{2(2-y^2)}{2+y^2},\ \ a=\frac1{\sqrt{2+x^2}},\ \ b=\frac1{\sqrt{2+y^2}}.
\]
Equations \eqref{sym:eq:UV}, with the signs just established, become
\[
2bq=a(p+1)(1+x)\ \text{ and }\ 2aq=b(s+1)(1+y).
\]
Define
\[
h(z)=\frac{(6-z^2)(1+z)}{(2+z^2)^2}.
\]
We therefore obtain
\begin{equation}
h(x)=h(y)=\gamma\ \text{ and }\ q=\frac{\gamma}{2}\sqrt{(2+x^2)(2+y^2)}.
\label{sym:eq:h}
\end{equation}
The remaining equation $UV=\ph(q)$ is equivalent to
\begin{equation}
q^2-3=\frac{4\gamma xy}{(1+x)(1+y)}.
\label{sym:eq:cross}
\end{equation}
All divisions have positive denominators. Conversely,
\eqref{sym:eq:h} through \eqref{sym:eq:cross} with the indicated ranges reconstruct
the normals and a positive kernel: take face areas proportional
to $(V,V,-U,-U)$.

\subsubsection*{The unequal-parameter branch}

Assume $p\ne s$, or equivalently $x\ne y$. Put $u=x+y$ and $v=xy$, so that $0<v<2/3$.
Subtracting $h(x)=h(y)$ and dividing by $x-y$ gives
\begin{equation}
\begin{split}
F(u,v)={}&-6u^3-(6v+4)u^2+(v^2+12v-28)u\\
&+v^3+2v^2-20v+24=0.
\end{split}
\label{sym:eq:F}
\end{equation}
To record the other equation compactly, set $W=4+2u^2-4v+v^2$, $Z=36-6u^2+12v+v^2$, $T=1+u+v$ and
\[
\begin{split}
C(u,v)={}&6u^4+(6v-4)u^3+(-v^2-24v+20)u^2\\
&+(-v^3-22v^2+36v+24)u+2v^3+4v^2-40v+48.
\end{split}
\]
Here $W=(2+x^2)(2+y^2)$, $Z=(6-x^2)(6-y^2)$,
$T=(1+x)(1+y)$, and
\[
C=(6-x^2)(1+x)(2+y^2)^2
+(6-y^2)(1+y)(2+x^2)^2.
\]
Equation \eqref{sym:eq:cross} is therefore
\begin{equation}
G(u,v):=ZT^2W-12W^2T-8vC(u,v)=0.
\label{sym:eq:G}
\end{equation}

The following identities are checked in \texttt{check\_T1.py}. The resultant is
\begin{equation}
\Res_u(F,G)
=3538944(v-2)^4(v+2)^2(v^2-6)P(v),
\label{sym:eq:resultant}
\end{equation}
where
\begin{equation}
P(v)=v^6-132v^5+4116v^4-33792v^3
+116736v^2-212736v+20736.
\label{sym:eq:P}
\end{equation}
Every factor apart from $P$ is nonzero on $(0,2/3)$.

\begin{lemma}
\label{sym:lem:root}
The polynomial $P$ has exactly one root in $(0,2/3)$.
For this value of $v$, equations $F=G=0$ have exactly one common
value of $u$, and that value is positive.
\end{lemma}

\begin{proof}
For $0<v<2/3$, dropping the negative terms in $P'$ gives
\[
\begin{split}
P'(v)
&<6(2/3)^5+16464(2/3)^3+233472(2/3)-212736\\
&=-\frac{4228928}{81}<0.
\end{split}
\]
Also $P(0)=20736$ and $P(2/3)=-57169664/729$.
This proves the first assertion.

The linear subresultant of $F$ and $G$ is
\[
-6144(v-2)(v+2)\{A(v)u+B(v)\},
\]
where
\[
\begin{split}
A(v)={}&64v^6-2900v^5+31506v^4-98472v^3\\
&+73368v^2+395904v-852672,
\end{split}
\]
\[
\begin{split}
B(v)={}&64v^7-2737v^6+25186v^5-71268v^4\\
&+33144v^3+247008v^2-683136v+614400.
\end{split}
\]
Throughout $(0,2/3)$,
\[
A(v)<64(2/3)^6+31506(2/3)^4+73368(2/3)^2
+395904(2/3)-852672
=-\frac{400876352}{729}<0.
\]
Thus the linear subresultant is nonzero at the root of $P$.
Since the resultant vanishes, the common root is unique and is
\begin{equation}
u=-B(v)/A(v).
\label{sym:eq:u}
\end{equation}
Rational substitution gives $P(0.1031379)>0$ and $P(0.1031380)<0$.
In particular, $0.10<v<0.11$. On this interval
\[
B(v)>614400-683136(0.11)-71268(0.11)^4
-2737(0.11)^6>0.
\]
Hence $u>0$.
\end{proof}

The unordered pair $\{x,y\}$ is now uniquely determined as the roots
of $z^2-uz+v=0$. Exact interval substitution in \eqref{sym:eq:u} gives
\[
0.67384<u<0.67386,\ \ 0.2350<\min(x,y)<0.2351\ \text{ and }\ 0.4387<\max(x,y)<0.4389.
\]
In particular, the roots are real, positive, distinct and less than
$\sqrt{2/3}$. Equation \eqref{sym:eq:h} gives $1.84<q<1.85$.
Thus the construction is admissible and proves existence.

This gives an exact description of $T_1$:
take the unique root $v$ of \eqref{sym:eq:P} in $(0,2/3)$, set
$u=-B(v)/A(v)$, let $x<y$ be the roots of $z^2-uz+v$, and recover
$p,s,q$ from \eqref{sym:eq:xy} and \eqref{sym:eq:h}.
Numerically,
\[
\begin{aligned}
v&=0.103137987399601104094765707121\ldots,\\
u&=0.673845625206396229272045732691\ldots,\\
p&=1.892477377432828212380883568822\ldots,\\
s&=1.648726879892260334958056852158\ldots,\\
q&=1.844867969052390261064873920274\ldots .
\end{aligned}
\]
To recover the original edges, put $h_p=\ph(p)$, $h_s=\ph(s)$ and
\[
c_1^2=\frac{1-h_p}{4(1+h_p)},\ \ b_1^2=c_1^2\,\frac{1+h_s}{1-h_s}.
\]
The vertices $A=(-1/2,0,0)$, $B=(1/2,0,0)$, $C=(0,c_1,b_1)$ and $D=(0,c_1,-b_1)$ have the required normal Gram matrix, and $\mu_*=2b_1$ and $\lambda_*=(1/4+c_1^2+b_1^2)^{1/2}$.
Interchanging $x,y$ interchanges the two distinguished opposite edges,
so it gives the same similarity class.

\subsubsection*{The equal-parameter branch}
If $p=s$, the matrix $K$ in \eqref{sym:eq:two} is also invariant under $(1\,3)(2\,4)$, so the positive kernel vector of $H$ has four equal entries, and \eqref{eq:edge} shows that opposite edges are equal. The original tetrahedron is isosceles, and Theorem~\ref{t2:thm:classification} gives precisely the regular tetrahedron and $T_2$. Together with the unequal-parameter branch, this completes the proof of Theorem~\ref{sym:thm:classification}.

\begin{corollary}\label{sym:cor:equivalence}
Conjectures~\ref{conj:four} and~\ref{conj:classification} are equivalent.
\end{corollary}
\begin{proof}
Each of the three classes has four equal cross edges. Conversely, Theorem~\ref{sym:thm:classification} classifies all solutions with that edge property.
\end{proof}

\subsection{Tetrahedra with a nontrivial symmetry}

Theorems~\ref{thm:disphenoid}, \ref{thm:twopairs}, \ref{thm:reflection} and~\ref{sym:thm:classification} together cover every tetrahedron that has a symmetry.

\begin{corollary}\label{cor:symmetric}
Suppose that $\M(ABCD)$ is regular and that some isometry other than the identity maps the vertex set of $ABCD$ onto itself. Then $ABCD$ has four equal cross edges, and up to similarity it is the regular tetrahedron, $T_1$ or $T_2$. Consequently Conjectures~\ref{conj:four} and~\ref{conj:symmetry} are equivalent, and a counterexample to Conjecture~\ref{conj:four} must have trivial symmetry group.
\end{corollary}

\begin{proof}
The isometry induces a nonidentity permutation $\pi$ of the vertices preserving all six edge lengths. Conversely, every permutation of the vertices preserving the six edge lengths is induced by an isometry, since the edge lengths determine the tetrahedron up to congruence. After relabeling, $\pi$ is one of $(1\,2)$, $(1\,2)(3\,4)$, $(1\,2\,3)$ or $(1\,2\,3\,4)$.

If $\pi=(1\,2)$, the isometry fixes $A_3$, $A_4$ and the midpoint of $A_1A_2$. These three points are not collinear, so the isometry is the reflection in the plane through them. Theorem~\ref{thm:reflection} gives four equal cross edges.

If $\pi=(1\,2)(3\,4)$, then $A_1A_3=A_2A_4$ and $A_1A_4=A_2A_3$. Theorem~\ref{thm:twopairs} applies.

If $\pi=(1\,2\,3)$, then $A_1A_2=A_2A_3=A_3A_1$ and $A_1A_4=A_2A_4=A_3A_4$. The transposition $(1\,2)$ then preserves all edge lengths. It is induced by an isometry, which is a reflection by the first case.

If $\pi=(1\,2\,3\,4)$, then $A_1A_2=A_2A_3=A_3A_4=A_4A_1$ and $A_1A_3=A_2A_4$. All three pairs of opposite edges are equal, and Theorem~\ref{thm:disphenoid} applies.

In every case the tetrahedron has four equal cross edges, and Theorem~\ref{sym:thm:classification} completes the proof.
\end{proof}

\subsection{The arithmetic of the three classes}

The three solutions differ arithmetically. For the regular tetrahedron all $q_{ij}$ equal the root of $3q^3-9q-2=0$ in $(1,2)$, and $T_2$ is given by square roots (Theorem~\ref{thm:T2}). The tetrahedron $T_1$ cannot be expressed by radicals.

\begin{proposition}\label{prop:T1galois}
The polynomial $P$ in \eqref{sym:eq:P} is irreducible over $\mathbb Q$ and its Galois group is the symmetric group $S_6$. Consequently the number $v$ specifying $T_1$ is not expressible by radicals, and no tetrahedron similar to $T_1$ has vertex coordinates expressible by radicals.
\end{proposition}

\begin{proof}
We use Dedekind's theorem: if a prime $\ell$ divides neither the leading coefficient nor the discriminant of $P$, the degrees of the irreducible factors of $P$ modulo $\ell$ form the cycle type of an element of the Galois group. None of the primes $7$, $23$ and $31$ divides the discriminant of $P$, since $P$ is squarefree modulo each of them. Modulo $23$, $P$ is irreducible. Modulo $7$ it factors as a quintic times a linear factor, and modulo $31$ as a cubic, a quadratic and a linear factor. The first fact shows that $P$ is irreducible over $\mathbb Q$ and that the Galois group $G$, as a transitive group of degree six, contains a $6$-cycle. The second gives a $5$-cycle fixing a point, so the point stabilizer is transitive on the remaining points; hence $G$ is $2$-transitive and primitive. The cube of an element of type $(3,2,1)$ is a transposition. A primitive permutation group containing a transposition is the full symmetric group. Thus $G=S_6$, which is not solvable, and $v$ is not expressible by radicals.

Suppose a tetrahedron similar to $T_1$ had coordinates expressible by radicals. Its unit normals, and hence $\cos\theta_{ij}=-\ph(q_{ij})$, would be expressible by radicals. Each $q_{ij}$ is a root of $x^3-3x-2\cos\theta_{ij}$, so the numbers $p,s$ of Theorem~\ref{sym:thm:classification} would be expressible by radicals, and so would $x,y$ in \eqref{sym:eq:xy} and $v=xy$, a contradiction.
\end{proof}

The factorizations modulo $7$, $23$ and $31$ are checked in \texttt{check\_galois\_T1.py}. Thus the root of the sextic $P$ in Theorem~\ref{sym:thm:classification} cannot be replaced by an expression in radicals.

\section{The case with one pair of equal opposite edges}\label{sec:onepair}

\begin{proposition}
\label{prop:cotangent}
If $AB=CD=\ell$, then
\begin{equation}
24V\ell(\cot\theta_{CD}-\cot\theta_{AB})
=(AC^2-BD^2)(AD^2-BC^2),
\label{eq:cotangent}
\end{equation}
where $\theta_{XY}$ denotes the interior dihedral angle along $XY$. In particular,
\[
\theta_{AB}=\theta_{CD}\iff AC=BD\ \text{ or }\ AD=BC.
\]
\end{proposition}

\begin{proof}
Heron's formula for the four faces, using $AB^2=CD^2$, gives
\[
8(S_A^2+S_B^2-S_C^2-S_D^2)
=(AC^2-BD^2)(AD^2-BC^2).
\]
For a direct check, put $a=\ell^2$, $x=AC^2$, $y=AD^2$, $z=BC^2$, $t=BD^2$. In the corresponding sum of the four expressions for $16S_i^2$, only $2(zt+xy-yt-xz)=2(x-t)(y-z)$ remains.

The equilibrium of the face area vectors gives
\[
S_A^2+S_B^2-S_C^2-S_D^2
=2S_AS_B\cos\theta_{CD}-2S_CS_D\cos\theta_{AB}.
\]
By \eqref{eq:edge}, both $S_AS_B\sin\theta_{CD}$ and $S_CS_D\sin\theta_{AB}$ equal $3V\ell/2$. Substitution proves \eqref{eq:cotangent}. The last assertion follows because cotangent is injective on $(0,\pi)$.
\end{proof}

Together with Theorem~\ref{thm:twopairs}, the proposition shows that Conjecture~\ref{conj:onepair} would follow from
\begin{equation}
\M(ABCD)\text{ regular and }AB=CD\implies\theta_{AB}=\theta_{CD}.
\label{eq:missing}
\end{equation}
Conversely, suppose that $AB=CD$ and that the tetrahedron has four equal cross edges with respect to some pair of opposite edges. Then it has a second pair of equal opposite edges, and Proposition~\ref{prop:cotangent} gives $\theta_{AB}=\theta_{CD}$. Hence, for tetrahedra with $AB=CD$, Conjecture~\ref{conj:onepair} is equivalent to \eqref{eq:missing}; the implication itself is not proved here. With $A,B,C,D=A_1,A_2,A_3,A_4$, it asks for $q_{12}=q_{34}$.

\section{Near the regular tetrahedron}\label{sec:nearregular}

\begin{theorem}
\label{thm:localregular}
There is a neighborhood of the regular shape in the space of tetrahedra modulo similarities in which a regular Morley tetrahedron implies that the original tetrahedron is regular.
\end{theorem}

\begin{proof}
Let
\[
R=\begin{pmatrix}1&1&1\\1&-1&-1\\-1&1&-1\\-1&-1&1\end{pmatrix},
\]
with rows $r_i$. Then $R^TR=4I_3$ and $RR^T=4I_4-\one\one^T$. If $X$ is the $4\times3$ matrix of the vertices of a tetrahedron with centroid $0$, then $\one^TX=0$, so $X=RN$ with $N=R^TX/4$, and the polar decomposition $N=G^{1/2}O$ shows that $XO^{-1}=RG^{1/2}$ with $G=NN^T$ positive definite. So every tetrahedron is similar to the one with vertex matrix $RG^{1/2}$ for some symmetric positive definite $G$ with $\tr G=3$. The tetrahedron $RG^{1/2}$ is regular exactly when $RGR^T$ is a multiple of $4I_4-\one\one^T$, that is, when $G=I_3$. Hence tetrahedra whose shape is close to the regular one correspond to matrices $G$ close to $I_3$. Put
\[
\mu_i=(r_iG^{-1}r_i^T)^{1/2}\ \text{ and }\ H_{ij}=\frac{r_iG^{-1}r_j^T}{\mu_i\mu_j},
\]
and obtain $K$ from $H$ as in Section~\ref{sec:gram}. The $\mu_i$ are proportional to the face areas. The barycentric coordinate matrix of the Morley vertices is
\[
P(G)=\diag\bigl((K\mu)_i^{-1}\bigr)K\diag(\mu_i).
\]
Consequently, if $M=P(G)X$ and $Y=R^TM/4$, the normalized shape of the Morley tetrahedron is
\begin{equation}
\Phi(G)=\frac{3YY^T}{\tr(YY^T)}.
\label{eq:shapemap}
\end{equation}
The map is real analytic near $I_3$, and the Morley tetrahedron is regular exactly when $\Phi(G)=I_3$.

Let $q=2\cos\bigl(\frac13\arccos\frac13\bigr)$ be the root of $3q^3-9q-2=0$ in $(1,2)$, and put $\lambda=1+3q$, $h=(1-q)/\lambda$. At $G=I_3$ we have $H_0=(4I_4-\one\one^T)/3$, $K_0=(1-q)I_4+q\one\one^T$ and $Y_0=hI_3$.
Let $E$ be a symmetric matrix with trace zero. Write it as the sum of a diagonal matrix and a matrix whose diagonal entries are zero. Differentiation gives
\begin{equation}
D\Phi_{I_3}(E)=\kappa_2 E_{\mathrm{diag}}+\kappa_3 E_{\mathrm{off}},
\label{eq:differential}
\end{equation}
where
\begin{equation}
\kappa_2=\frac{8-9(q-1)^2}{9(q-1)^2(q+1)}\ \text{ and }\ \kappa_3=-\frac{8-9(q-1)^2}{27(q-1)^2(1+3q)}.
\label{eq:kappas}
\end{equation}

We record the intermediate differentials. Set $a_i=r_iEr_i^T$, $v_i=-a_i/6$, and $\eta=-2/[3(q^2-1)]$. Then $\delta\mu_i/\sqrt3=v_i$,
\[
\delta H_{ij}=-\frac{r_iEr_j^T}{3}+(H_0)_{ij}\frac{a_i+a_j}{6},
\]
and $\delta K_{ij}=\eta\delta H_{ij}$ off the diagonal, with zero diagonal. Also,
\[
\begin{gathered}
\delta P=\frac{\delta K+K_0\diag v}{\lambda}-\frac{\diag(\delta K\one+K_0v)K_0}{\lambda^2},\\
\delta Y=\frac14R^T\delta P R+\frac h2E.
\end{gathered}
\]
Substitution into \eqref{eq:shapemap}, evaluated on $\diag(1,-1,0)$ and on the symmetric matrix with entries $E_{23}=E_{32}=1$ and all others zero, gives respectively
\[
1-\frac{16}{9(q-1)^2(q+1)}\ \text{ and }\ \frac{q+3}{3(1+3q)}-\frac{32}{27(q-1)^2(1+3q)}.
\]
Using $3q^3-9q-2=0$ reduces these to \eqref{eq:kappas}; coordinate permutations give the remaining basis directions and establish \eqref{eq:differential}.

The cubic is strictly increasing on $(1,2)$ and changes sign between $9/5$ and $19/10$. Hence $8-9(q-1)^2>0$, so $\kappa_2>0$ and $\kappa_3<0$. The two tangent subspaces have dimensions $2$ and $3$, respectively, and the differential is invertible on their direct sum. The inverse function theorem now implies that $\Phi(G)=I_3$ has only the solution $G=I_3$ in a neighborhood of $I_3$.
\end{proof}

\section{The conjecture for special tetrahedra}
\label{sec:special}

Three classical families of tetrahedra are defined by a relation between the three pairs of opposite edges. For orthocentric tetrahedra the three sums of squares of opposite edges are equal, for isodynamic tetrahedra the three products are equal, and in the third family the three sums are equal. We add two more families, the tetrahedra with three equal edges at one vertex and the tetrahedra with an equilateral face. Each family has three parameters up to similarity, and its members have no symmetry in general, so Corollary~\ref{cor:symmetric} does not apply to them. Neither exceptional tetrahedron belongs to any of the five families; the numbers are given below. By Corollary~\ref{cor:symmetric} and Theorem~\ref{sym:thm:classification}, Conjecture~\ref{conj:four} would therefore imply that the regular tetrahedron is the only member of these families with a regular Morley tetrahedron. We prove this without using the conjecture. In Section~\ref{sec:acute} the same method handles all tetrahedra whose dihedral angles are at most $95^\circ$, a class with five parameters.

\begin{theorem}
\label{thm:special}
Suppose that the Morley tetrahedron of $A_1A_2A_3A_4$ is regular, and that one of the following holds:
\begin{enumerate}[label=\textup{(\roman*)}]
\item $A_1A_2^2+A_3A_4^2=A_1A_3^2+A_2A_4^2=A_1A_4^2+A_2A_3^2$, that is, the tetrahedron is orthocentric;
\item $A_1A_2\cdot A_3A_4=A_1A_3\cdot A_2A_4=A_1A_4\cdot A_2A_3$, that is, the tetrahedron is isodynamic;
\item $A_1A_2+A_3A_4=A_1A_3+A_2A_4=A_1A_4+A_2A_3$;
\item $A_4A_1=A_4A_2=A_4A_3$;
\item the face $A_1A_2A_3$ is equilateral.
\end{enumerate}
Then $A_1A_2A_3A_4$ is regular.
\end{theorem}

Since trirectangular tetrahedra are orthocentric, none of them has a regular Morley tetrahedron. Theorem~\ref{thm:acute} below covers all tetrahedra whose dihedral angles are at most $95^\circ$, so what Theorem~\ref{thm:special} adds concerns the members of the five families with a larger dihedral angle. The proof of (v) uses Theorem~\ref{thm:acute}.

The five proofs are computer assisted and follow one pattern. In each family the normal Gram matrix has a closed form in three parameters (Lemmas~\ref{lem:orthogram}, \ref{lem:isogram}, \ref{lem:midgram}, \ref{lem:apexgram} and~\ref{lem:eqgram}). Regularity of the Morley tetrahedron means five equations, so on a three-parameter family one expects no solutions apart from the regular one. A proof has to deal with the degenerate configurations, where the tetrahedron becomes flat or a vertex recedes to infinity, and $K$ becomes singular. Sections~\ref{sec:rowscaled} and~\ref{sec:bb} describe the common tools, Sections~\ref{sec:orthocentric}--\ref{sec:apex} treat the first four families, Section~\ref{sec:acute} the tetrahedra with dihedral angles of at most $95^\circ$, and Section~\ref{sec:eqface} the fifth family.

\subsection{A well-conditioned regularity test}
\label{sec:rowscaled}

Let $\mathcal N$ be the $4\times3$ matrix whose rows are $n_1,\dots,n_4$, and let $\widetilde m_i$ be the rows of $K^{-1}\mathcal N$. As in the proof of Proposition~\ref{prop:regularity}, $n_j=\sum_it_iq_{ij}\nabla\mu_i$, where $\mu_i$ are the barycentric coordinates with respect to the Morley tetrahedron, so $\widetilde m_i=t_i\nabla\mu_i$ and the Gram matrix of the $\widetilde m_i$ is $B=K^{-1}HK^{-1}$. Thus the Morley tetrahedron is regular if and only if
\begin{equation}
\rho_{ij}:=9G_{ij}^2-G_{ii}G_{jj}=0\ \text{ and }\ G_{ij}<0\ \text{ for all }i<j,
\label{or:test}
\end{equation}
where $G=B$ or $G$ is any positive multiple of $B$.

For an invertible matrix $R$ the rows $\widetilde m_i$ are also the rows of $\widetilde K^{-1}\widetilde{\mathcal N}$, where $\widetilde K=RK$ and $\widetilde{\mathcal N}=R\mathcal N$. So
\begin{equation}
G=\operatorname{adj}(\widetilde K)\,\Gamma\,\operatorname{adj}(\widetilde K)^T,\ \text{ where }\ \Gamma=\widetilde{\mathcal N}\widetilde{\mathcal N}^T,
\label{or:rowscaled}
\end{equation}
equals $(\det\widetilde K)^2B$, a positive multiple of $B$. The entries of $\Gamma$ are scalar products of combinations of the normals, hence explicit in terms of $H$. With $R=I$ we get $G=\operatorname{adj}(K)H\operatorname{adj}(K)$. Near a flat tetrahedron $K$ is close to a singular matrix, and $G$ computed in this way is a difference of large terms. We then choose $R$ so that it replaces some rows of $K$ by differences of nearly equal rows, divided by the rate at which they become equal, and treats the rows of $\mathcal N$ in the same way. Two identities make this explicit. If $q=1+\eta$ and $H=\ph(q)$, then $2(1-H)=\eta^2(3+\eta)$, so $\eta$ is an analytic function of $|n_i-n_j|=\sqrt{2(1-H_{ij})}$ that vanishes linearly with it. If $q=2-\xi$, then $2(1+H)=\xi(9-6\xi+\xi^2)$, so $\xi$ is an analytic function of $|n_i+n_j|^2$.

The test is only applied at points that come from tetrahedra. At such a point $K$ is invertible by Lemma~\ref{lem:negative}, and so is each matrix $R$ that we use, since it only subtracts rows and divides by positive numbers. So $G$ is a positive multiple of $B$ there, and it does not matter whether $\widetilde K$ stays invertible on the boundary of a chart.

\subsection{The branch and bound}
\label{sec:bb}

For each family the parameter space, reduced by symmetry, is covered by a few charts. On each chart the entries of $H$, $K$, $\widetilde K$ and $\Gamma$ are given by formulas in which the vanishing factors have been cancelled, so that they stay analytic on the chart up to its degenerate boundary. On each chart a script runs a branch and bound. For a box it computes Taylor models of the $\rho_{ij}$ and $G_{ij}$ in floating-point interval arithmetic whose endpoints are rounded outward after every operation. A first-order Taylor model is an affine function of the coordinates, with interval coefficients, together with an interval bounding the remainder. A second-order one has a quadratic polynomial with interval coefficients in place of the affine function.

The numbers $q_{ij}$, $\eta$, $\xi$ and a few similar quantities are roots of cubic equations. Each root is enclosed between two numbers at which the cubic has verified opposite signs, and the script also checks that the bracket lies on the intended branch. For the root $q\in[1,2]$ of $q^3-3q=2t$, for instance, the cubic increases for $q\ge1$, so it suffices that the upper number is at least $1$. Derivatives of the roots are bounded by the implicit function theorem; for this root, $d^3q/dt^3=16(5q^2+1)/(27(q^2-1)^5)$.

A box is discarded if some $\rho_{ij}$ is bounded away from $0$ on it, if some $G_{ij}$ is positive on it, or if it contains no point of the part of the parameter space assigned to the chart; otherwise it is bisected. These tests are done in interval arithmetic as well. A box may contain points that do not come from a tetrahedron, where some formula is undefined; the computed enclosure is then NaN, and a NaN never leads to discarding a box. The initial boxes of a chart form a grid whose neighboring boxes share their faces exactly, so together they cover the chart.

The margins are small, because the Morley map contracts strongly; compare the eigenvalues $\kappa_2,\kappa_3$ in Section~\ref{sec:nearregular}. Where they are smallest, first-order Taylor models would need very small boxes, and second-order ones are used instead. All boxes are discarded except those contained in a small cube around the regular point of one chart. The cube may extend beyond the fundamental domain and beyond the initial boxes, but the formulas of the chart are analytic on the whole cube. On this cube the script encloses the Jacobian matrix $J$ of $F=(\rho_{12},\rho_{13},\rho_{14})$ by Taylor models of its entries and checks $\|I-CJ\|_\infty<1$ for a fixed matrix $C$. Then $F$ is injective on the cube: if $F(x)=F(y)$, then $0=C(F(x)-F(y))=CM(x-y)$, where $M$ is the mean of $J$ along the segment from $x$ to $y$. The matrix $M$ lies in the interval enclosure of $J$, so $\|I-CM\|_\infty<1$, the matrix $CM$ is invertible, and $x=y$. The regular tetrahedron is a solution in the cube, so it is the only one there.

\subsection{Orthocentric tetrahedra}
\label{sec:orthocentric}

A tetrahedron is orthocentric if its four altitudes pass through one point $O$, its orthocenter. Equivalently, opposite edges are perpendicular, or the condition in (i) holds. For $T_2$ we have $AB^2+CD^2=2$, while $AC^2+BD^2=2(21+4\sqrt6)/45<2$. For $T_1$, Theorem~\ref{thm:T1} gives $AB^2+CD^2>1+2.35^2>6.5$ and $AC^2+BD^2=2\lambda_*^2<2\cdot1.59^2<5.1$.

\begin{lemma}
\label{lem:orthogram}
Let $A_1A_2A_3A_4$ be orthocentric with orthocenter $O$, and put $a_i=A_i-O$. There is a number $c$ with $a_i\cdot a_j=c$ for all $i\ne j$. With $e_i=|a_i|^2-c$ we have $|A_iA_j|^2=e_i+e_j$, at most one $e_i$ is not positive, and if $e_i\ne0$ then the inward unit normal of $F_i$ is $n_i=\operatorname{sign}(e_i)\,a_i/|a_i|$.
\begin{enumerate}[label=\textup{(\alph*)}]
\item If all $e_i>0$, then
\begin{equation}
H_{ij}=-y_iy_j\ \text{ for }i\ne j,\ \text{ where } y_i>0\ \text{ and }\ \sum_{i=1}^4\frac{y_i^2}{1+y_i^2}=1.
\label{or:acute}
\end{equation}
\item If $e_4<0$, then
\begin{equation}
\begin{gathered}
H_{ij}=y_iy_j\ \text{ for }1\le i<j\le3\ \text{ and }\ H_{i4}=-y_iy_4\ \text{ for }1\le i\le3,\\
\text{where } 0<y_i<1 \text{ for } i\le3,\ \ y_4^2=1+\frac1W\ \text{ and }\ W=\sum_{i=1}^3\frac{y_i^2}{1-y_i^2}.
\end{gathered}
\label{or:obtuse}
\end{equation}
Moreover the face $A_1A_2A_3$ is an acute triangle and the foot of the altitude from $A_4$ is its orthocenter. Let $h$ be that altitude, let $p_i$ be the distance from the orthocenter of $A_1A_2A_3$ to the side opposite $A_i$, and put $t_i=h/p_i$ (in this subsection $t_i$ does not denote the distances of \eqref{eq:distances}). Then
\begin{equation}
H_{ij}=\frac{N}{r_ir_j}\ \text{ for }1\le i<j\le3\ \text{ and }\ H_{i4}=-\frac1{r_i}\ \text{ for }1\le i\le3,
\label{or:flat}
\end{equation}
where $r_i=\sqrt{1+t_i^2}$ and $N=1-t_1t_2t_3/\lambda$, with $\lambda$ the positive root of $\lambda^3-(t_1^2+t_2^2+t_3^2)\lambda-2t_1t_2t_3=0$. Moreover $0<N<1$.
\item If $e_4=0$, the three face angles at $A_4$ are right angles. Then $H_{ij}=0$ for $1\le i<j\le3$ and $H_{i4}=-\nu_i$ for $1\le i\le3$, where $\nu$ is a unit vector with positive coordinates.
\end{enumerate}
\end{lemma}

\begin{proof}
The altitude from $A_i$ passes through $O$, so $a_i\cdot(a_j-a_k)=0$ whenever $i,j,k$ are distinct. Hence $a_i\cdot a_j$ does not depend on $j\ne i$, and by symmetry it does not depend on $i$ either. Then $|A_iA_j|^2=|a_i|^2+|a_j|^2-2c=e_i+e_j$. The Gram matrix $cJ+\diag(e_i)$ of the $a_i$ is positive semidefinite. If $e_3,e_4\le0$, then $c\ge\max(-e_3,-e_4)$ and $(c+e_3)(c+e_4)\ge c^2$, that is, $e_3e_4\ge c(-e_3-e_4)$; this is impossible unless $e_3=e_4=0$, and then $|A_3A_4|=0$. If $e_j\neq0$, the function $\lambda_j(X)=((X-O)\cdot a_j-c)/e_j$ is affine and takes the values $\delta_{jk}$ at the vertices $A_k$, so it is the $j$-th barycentric coordinate; its gradient $a_j/e_j$ points into the tetrahedron, which gives $n_j$.

Four vectors in $\R^3$ are dependent, so $\det(cJ+\diag e)=e_1e_2e_3e_4\bigl(1+c\sum_i1/e_i\bigr)=0$ when no $e_i$ vanishes, and $c=-1/\sum_i e_i^{-1}$. In case (a), $c<0$, and $y_i=\sqrt{-c}/|a_i|$ gives $H_{ij}=c/(|a_i||a_j|)=-y_iy_j$ and $y_i^2/(1+y_i^2)=-c/e_i$, whose sum is $1$. In case (b), positive semidefiniteness gives $c\ge-e_4>0$. Put $y_i=\sqrt c/|a_i|$. Then $H_{ij}=y_iy_j$ and $H_{i4}=-y_iy_4$, $y_i^2/(1-y_i^2)=c/e_i$ for $i\le3$, so $W=c\sum_{i\le3}e_i^{-1}$, and $-c/e_4=1+W$ gives $y_4^2=c/(c+e_4)=1+1/W$.

The angle of $A_1A_2A_3$ at $A_1$ is acute because $(A_2-A_1)\cdot(A_3-A_1)=|a_1|^2-c=e_1>0$. The foot $P$ of the altitude from $A_4$ is the orthocenter of $A_1A_2A_3$, because the plane through $A_4$, $A_1$ and $P$ is perpendicular to $A_2A_3$. Place $P$ at the origin of the plane $z=0$ containing $A_1A_2A_3$ and $A_4=(0,0,h)$. Let $m_i$ be the inward unit normal, in that plane, of the side opposite $A_i$. Then $n_4=(0,0,1)$ and $n_i=(hm_i,-p_i)/\sqrt{h^2+p_i^2}$, so $H_{i4}=-1/r_i$ and $H_{ij}=(1-t_it_j\cos C_k)/(r_ir_j)$, where $C_k$ is the angle of $A_1A_2A_3$ at $A_k$ and $m_i\cdot m_j=-\cos C_k$. With $R$ the circumradius of $A_1A_2A_3$ one has $p_i=2R\cos C_j\cos C_k$, so $t_i=\lambda\cos C_i$ with $\lambda=h/(2R\cos C_1\cos C_2\cos C_3)$ and $t_it_j\cos C_k=t_1t_2t_3/\lambda$. The identity $\sum\cos^2C_i+2\prod\cos C_i=1$ becomes the cubic for $\lambda$, which has exactly one positive root. Finally $N>0$ because $H_{12}>0$ by \eqref{or:obtuse}.

In case (c), $(A_i-A_4)\cdot(A_j-A_4)=e_4=0$ for distinct $i,j\le3$. The three faces through $A_4$ are then pairwise perpendicular, so $H_{ij}=0$ for $i<j\le3$. The rows of $HS=0$ give $H_{i4}=-S_i/S_4$ for $i\le3$ and $\sum_{i\le3}(S_i/S_4)^2=1$.
\end{proof}

The two descriptions in (b) are related by $y_i=\sqrt N/r_i$ for $i\le3$ and $y_4=1/\sqrt N$. In case (a) we put $w_i=y_i^2/(1+y_i^2)$, so that $w_1+w_2+w_3+w_4=1$. The trirectangular tetrahedra of case (c) are limits of both families: in case (a) when one $w_i$ tends to $1$, and in case (b) when $W$ tends to $0$. In case (b) the three dihedral angles at the edges through $A_4$ are obtuse and the other three are acute. The largest one tends to $\pi$ exactly when two of the $t_i$ tend to $0$; the tetrahedron then becomes flat, and $K$ tends to a singular matrix.

In case (b) with $t_1,t_2,t_3\to0$ the first three rows of $K$ tend to $(1,1,1,2)$. Here we take $R$ that replaces the second and third rows of $K$ by $(\mathrm{row}_2-\mathrm{row}_1)/(bt_3)$ and $(\mathrm{row}_3-\mathrm{row}_1)/t_3$, where $t_2=bt_3$, and the second and third rows of $\mathcal N$ by $(n_2-n_1)/(bt_3)$ and $(n_3-n_1)/t_3$. Every entry of $\widetilde K$ and of $\Gamma$ is then an analytic function of $(t_3,b,t_1/t_2)$ up to $t_3=0$ and $b=0$. In \eqref{or:flat} all the differences that occur have an explicit factor. For example, with $r'=t_1/t_2$,
\[
\frac{1-H_{12}}{b^2t_3^2}=\frac1{r_1r_2}\Bigl(\frac{1+r'^2+r'^2b^2t_3^2}{r_1r_2+1}+\frac{r'}{\widehat\lambda}\Bigr)
\ \text{ and }\
H_{13}-H_{23}=\frac{N}{r_3}\cdot\frac{t_2^2-t_1^2}{r_1r_2(r_1+r_2)},
\]
where $\lambda=t_3\widehat\lambda$ and $\widehat\lambda$ is the root $\ge1$ of $\widehat\lambda^3-(1+b^2+r'^2b^2)\widehat\lambda-2r'b^2=0$. The complete list of formulas is in the script \texttt{check\_orthocentric.py}.

By symmetry we may label the vertices so that $w_1\ge w_2\ge w_3\ge w_4$ in case (a), and so that $A_4$ is the vertex with $e_4<0$ and $t_1\le t_2\le t_3$ in case (b). In case (b), $\theta_{12}$, the angle at the edge $A_3A_4$, is then the largest dihedral angle. The script covers these sets by six charts. The last column of the table gives the number of boxes examined.
\begin{center}\small
\begin{tabular}{llp{.42\textwidth}r}\toprule
Chart&Coordinates&Region&Boxes\\\midrule
A&$(u_2,u_3,u_4)$, $u_i=\sqrt{w_i}$&case (a), $w_2+w_3+w_4\ge0.01$&$53\,874$\\
C&$(s,\nu_3,\nu_4)$, $w_1=1-s$, $w_j=s\nu_j^2$&case (a), $s\le0.0101$, near the trirectangular tetrahedra&$1\,768$\\
O&$(y_1,y_2,y_3)$&case (b), $W\ge0.01$, $H_{12}\le2017/2048$&$140\,688$\\
OC&$(r,\nu_2,\nu_3)$, $y_i^2/(1-y_i^2)=r^2\nu_i^2$&case (b), $W=r^2\le0.0102$&$1\,768$\\
U&$(t_3,b,t_1/t_2)$, $t_2=bt_3$&case (b), $t_3\le1$, $t_2\le0.18$&$6\,243$\\
U2&$(t_2,t_1/t_2,1/r_3)$&case (b), $t_3\ge1$, $t_2\le0.18$&$628$\\\bottomrule
\end{tabular}
\end{center}
Here $\nu$ is a unit vector with decreasing coordinates. In chart C the right corner of the limiting trirectangular tetrahedron is $A_1$, and in chart OC it is $A_4$. Charts A, C, O and OC use $G=\operatorname{adj}(K)H\operatorname{adj}(K)$, and charts U and U2 use \eqref{or:rowscaled} with the rows scaled as described above (in U2 only the second row is scaled, by $t_2$). Together the six charts cover cases (a), (b) and (c). Chart O stops at $H_{12}=2017/2048$, that is, at $\theta_{12}\approx170.02^\circ$; the dyadic number makes the cut exact in floating point. If $H_{12}>2017/2048$ in case (b), then $(1+t_1^2+t_2^2)^{-1/2}\ge H_{12}>2017/2048$, because $N\le1$ and $r_1r_2\ge(1+t_1^2+t_2^2)^{1/2}$, and therefore $t_2^2<(2048/2017)^2-1<0.031<0.18^2$. The initial boxes contain the charts, since $w_2\le\frac12$, $\nu_3^2\le\frac12$ in chart C and $\nu_2^2\le\frac12$ in chart OC.

All charts use first-order Taylor models. All boxes are discarded except the cube $|u_k-\frac12|\le10^{-3}$ around the regular tetrahedron in chart A, on which $\|I-CJ\|_\infty<0.62$.

\begin{proof}[Proof of Theorem~\ref{thm:special}\textup{(i)}]
By Lemma~\ref{lem:orthogram}, every orthocentric tetrahedron belongs, after relabeling, to case (a), case (b) or case (c), and case (c) is the common boundary $s=0$ of chart C and $r=0$ of chart OC. The computation shows that \eqref{or:test} fails at every point of the six charts except the regular tetrahedron.
\end{proof}

\begin{remark}
The trirectangular tetrahedron with three equal edges at the right corner has a Morley tetrahedron whose normalized Gram entries differ from $-\frac13$ by less than $0.004$. Near flat tetrahedra the margin is larger, at least $0.016$ in our samples, but there $K$ is badly conditioned, so charts U and U2 are needed.
\end{remark}

\subsection{Isodynamic tetrahedra}
\label{sec:isodynamic}

An inversion with center $O$ and radius $r$ multiplies $|XY|$ by $r^2/(|OX|\,|OY|)$, so it multiplies the three products in (ii) by the same factor. Hence the image of a regular tetrahedron under an inversion whose center is not on its circumsphere is isodynamic. Neither exceptional tetrahedron is isodynamic: for $T_2$ we have $AB\cdot CD=1$ and $AC\cdot BD=(21+4\sqrt6)/45<1$, and for $T_1$ we have $AB\cdot CD=\mu_*<2.36$ and $AC\cdot BD=\lambda_*^2>1.57^2>2.46$. The boundary of this family is more complicated than in the orthocentric case: an isodynamic tetrahedron can become flat in a two-parameter family of ways, and one vertex can recede to infinity.

\begin{lemma}
\label{lem:isogram}
Let $A_1A_2A_3A_4$ be isodynamic. There is a vector $v\in\R^4$ with $\sum_iv_i>0$ and
\begin{equation}
Q=\Bigl(\sum_{i=1}^4v_i\Bigr)^2-\sum_{i=1}^4v_i^2>0
\label{iso:Q}
\end{equation}
such that, after scaling, $|A_iA_j|=(u_iu_j)^{-1/2}$ with $u_i=\sum_{k\ne i}v_k>0$. The normal Gram matrix is
\begin{equation}
H_{ij}=\frac{3v_iv_j-Q}{\sqrt{E_iE_j}}\ \text{ for }i\ne j,\ \text{ where }\ E_i=2Q+3v_i^2,
\label{iso:H}
\end{equation}
and the face areas are proportional to $u_i\sqrt{E_i}$.
\end{lemma}

\begin{proof}
Write $\ell_{ij}=|A_iA_j|$ and put $w_1^2=\ell_{12}\ell_{13}/\ell_{23}$, and similarly for $w_2,w_3,w_4$. The equality of the three products gives $\ell_{ij}=w_iw_j$ for all $i\ne j$. Let $u_i=w_i^{-2}$, $\sigma=\sum_iu_i$ and $v_i=\sigma/3-u_i$. Then $\sum_iv_i=\sigma/3>0$ and $u_i=\sum_{k\ne i}v_k$. With $U=\diag(u_i)$ and $M=J-I$, where $J$ is the matrix of ones, the matrix of squared edges is $D=U^{-1}MU^{-1}$. Moreover $M^{-1}=\frac13J-I$, $M^{-1}u=v$ and $u^TM^{-1}u=Q$. The Cayley--Menger determinant is
\[
\det\begin{pmatrix}0&\one^T\\\one&D\end{pmatrix}=\det U^{-2}\det\begin{pmatrix}0&u^T\\u&M\end{pmatrix}=3Q\prod_iw_i^4,
\]
and it equals $288V^2>0$, so $Q>0$.

Let $g_i=\nabla\lambda_i$ be the gradients of the barycentric coordinates and $G=(g_i\cdot g_j)$. From $\sum_k\lambda_k=1$ and $\sum_k\lambda_k(X)A_k=X$ we get $\sum_kg_k=0$ and $\sum_kA_kg_k^T=I$. Expanding $\ell_{ik}^2=|A_i|^2-2A_i\cdot A_k+|A_k|^2$ and using these two identities, one finds that $D(-\frac12G)+\one c^T=I$ for some vector $c$; since also $G\one=0$, the matrix $-\frac12G$ is the upper-left block of the inverse of $\left(\begin{smallmatrix}D&\one\\\one^T&0\end{smallmatrix}\right)$. Inverting by blocks,
\[
G=-2U\Bigl(M^{-1}-\frac{M^{-1}uu^TM^{-1}}{u^TM^{-1}u}\Bigr)U,\ \text{ that is, }\ G_{ij}=2u_iu_j\Bigl(\frac{v_iv_j}{Q}+\delta_{ij}-\frac13\Bigr).
\]
Thus $G_{ij}=\frac{2u_iu_j}{3Q}(3v_iv_j-Q)$ for $i\ne j$ and $G_{ii}=\frac{2u_i^2}{3Q}E_i$. Since $n_i=g_i/|g_i|$ and the area of $F_i$ is $3V|g_i|$, the lemma follows.
\end{proof}

Formula \eqref{iso:H} is unchanged when $v$ is multiplied by a positive number, and a relabeling of the vertices permutes the $v_i$. We may therefore assume $v_1=1\ge v_2\ge v_3\ge v_4$. Then $\omega=v_2+v_3+v_4$ is positive: $\omega>-1$ because $\sum_iv_i>0$, and $Q=2\omega+\omega^2-v_2^2-v_3^2-v_4^2\le\omega(2+\omega)$, so $\omega\le0$ would give $Q\le0$. Put $s=\sqrt\omega$ and $b_j=v_j/s$ for $j=2,3,4$. Then $b_2+b_3+b_4=s$,
\[
\widehat q=\frac{Q}{s^2}=2+s^2-b_2^2-b_3^2-b_4^2,
\]
$0<s\le\sqrt3$ because $v_j\le1$, and $sb_2=v_2\le1$. In these coordinates \eqref{iso:H} becomes, for $2\le j<k\le4$,
\begin{equation}
H_{1j}=\frac{3b_j-s\widehat q}{\sqrt{(3+2s^2\widehat q)\widehat E_j}}\ \text{ and }\ H_{jk}=\frac{3b_jb_k-\widehat q}{\sqrt{\widehat E_j\widehat E_k}},\ \text{ where }\ \widehat E_j=2\widehat q+3b_j^2.
\label{iso:Hsb}
\end{equation}
These expressions are analytic also at $s=0$. As $s\to0$ the edges at $A_1$ become long compared with the others, since $u_1=s^2$, and the tetrahedron tends to an infinite prism. As $\widehat q\to0$ it becomes flat. The regular tetrahedron is the point $s=\sqrt3$, $b_j=1/\sqrt3$, where $\widehat q=4$.

Suppose $\widehat q\le0.36$. Then $b_2>0>b_4$. Indeed, $\sum_jb_j^2=2+s^2-\widehat q>s^2=(\sum_jb_j)^2$, which is impossible if all $b_j$ are nonnegative, and if all are nonpositive then $s=0$ and $b=0$, which contradicts $\sum_jb_j^2>0$. As $\widehat q\to0$, \eqref{iso:Hsb} shows that $H_{1j}$ tends to the sign of $b_j$ and $H_{jk}$ to the sign of $b_jb_k$. So $q_{ij}$ tends to $1$ when the two signs agree and to $2$ otherwise. If $b_3>0$, the limit is a triangle with a point inside and the first three rows of $K$ become equal. If $b_3<0$, it is a convex quadrilateral, and rows $1,2$ and rows $3,4$ become equal.

All differences that occur have an explicit factor $\widehat q$. If $b_jb_k>0$, then $1-H_{jk}=\widehat q\,g_{jk}$, where
\[
g_{jk}=\frac1{R_{jk}}\Bigl(\frac{4\widehat q+6b_j^2+6b_k^2}{R_{jk}+3b_jb_k}+1\Bigr)\ \text{ and }\ R_{jk}=\sqrt{\widehat E_j\widehat E_k},
\]
and if $b_jb_k<0$, then $1+H_{jk}=\widehat q\,h_{jk}$, where $h_{jk}$ is obtained from $g_{jk}$ by changing the signs of $3b_jb_k$ and of the final $1$. For the pairs $(1,j)$ the same formulas hold with $R_{1j}=\sqrt{(3+2s^2\widehat q)\widehat E_j}$, with $3b_j$ in place of $3b_jb_k$, with $4s^2\widehat q+6+6s^2b_j^2$ in the numerator, and with $\pm s$ in place of $\pm1$. Finally, if $b_j>0$, then
\[
H_{jk}-H_{1k}=\widehat q\;\frac{6b_k(s^2b_j^2-1)\big/\bigl(b_j\sqrt{3+2s^2\widehat q}+\sqrt{\widehat E_j}\bigr)-\sqrt{3+2s^2\widehat q}+s\sqrt{\widehat E_j}}{\sqrt{(3+2s^2\widehat q)\widehat E_j\widehat E_k}} .
\]
Put $\widehat\sigma=\sqrt{\widehat q}$. By the two identities for $2(1\mp H)$ in Section~\ref{sec:rowscaled}, $q_{jk}=1+\widehat\sigma e_{jk}$ or $q_{jk}=2-\widehat\sigma^2x_{jk}$ with $e_{jk}$, $x_{jk}$ analytic up to $\widehat\sigma=0$.

We then use \eqref{or:rowscaled}. If $b_3\ge0.15$ (chart T), the second and third rows of $K$ and of $\mathcal N$ are replaced by their differences with the first row, divided by $\widehat\sigma$. If $b_3\le-0.15$ (chart Q), the second row is replaced by $(\mathrm{row}_2-\mathrm{row}_1)/\widehat\sigma$ and the fourth by $(\mathrm{row}_4-\mathrm{row}_3)/\widehat\sigma$. All entries of $\widetilde K$ and $\Gamma$ are then analytic up to $\widehat\sigma=0$; in chart T, for example, $\Gamma_{22}=2g_{12}$, $\Gamma_{23}=g_{12}+g_{13}-g_{23}$ and $\Gamma_{24}=\widehat\sigma(h_{24}-h_{14})$. Near the arc $b_3=0$ we use polar coordinates $b_3=\rho\sin\psi$, $\widehat\sigma=\rho\cos\psi$. The entries with index $3$ have limits as $\rho\to0$ that depend on $\psi$; for instance
\[
H_{13}=\frac{3\sin\psi-s\rho\cos^2\psi}{\sqrt{(3+2s^2\widehat q)(2+\sin^2\psi)}},
\]
and the last identity shows that $(H_{23}-H_{13})/(\rho\cos^2\psi)$ is analytic. Three charts cover this region. In chart AM, where $|\sin\psi|\le0.8$, only the second row is rescaled, by $\widehat\sigma$. Charts AT and AQ, where $\cos\psi\le0.8$, rescale the second row in the same way. In addition, chart AT ($\sin\psi>0$) replaces the third row by $(\mathrm{row}_3-\mathrm{row}_1)/\cos\psi$, and chart AQ ($\sin\psi<0$) replaces the fourth row by $(\mathrm{row}_4-\mathrm{row}_3)/\cos\psi$.

The six charts, and the number of boxes the script examines on each, are as follows.
\begin{center}\small
\begin{tabular}{lllr}\toprule
Chart&Coordinates&Region&Boxes\\\midrule
I&$(s,b_2,b_3)$&$\widehat q\ge0.25$&$189\,416$\\
T&$(s,b_3,\widehat\sigma)$&$\widehat q\le0.36$, $b_3\ge0.15$&$82\,025$\\
Q&$(s,b_3,\widehat\sigma)$&$\widehat q\le0.36$, $b_3\le-0.15$&$20\,116$\\
AM&$(s,\rho,\sin\psi)$&$\rho\le0.62$, $|\sin\psi|\le0.8$&$22\,180$\\
AT&$(s,\rho,\cos\psi)$&$\rho\le0.62$, $\sin\psi>0$, $\cos\psi\le0.8$&$6\,299$\\
AQ&$(s,\rho,\cos\psi)$&$\rho\le0.62$, $\sin\psi<0$, $\cos\psi\le0.8$&$5\,717$\\\bottomrule
\end{tabular}
\end{center}
In the last five charts $b_2$ and $b_4$ are computed from $s$, $b_3$ and $\widehat q$ as the larger and the smaller root of a quadratic equation. Together the six charts cover the parameter domain, since a point with $\widehat q\le0.36$ and $|b_3|<0.15$ has $\rho<(0.36+0.15^2)^{1/2}<0.62$.

The initial boxes contain all the points of the charts. In chart I they cover $0\le s\le\sqrt3$, $0\le b_2\le1.65$ and $-0.85\le b_3\le1.2$. Here $b_2\ge s/3\ge0$, because $b_2$ is the largest of three numbers with sum $s$. From $b_2\le1/s$ and $b_2^2+(s-b_2)^2/2\le|b|^2<2+s^2$ one gets $b_2<1.47$. If $b_3\ge1.05$, then $b_2\ge1.05$, so $s\le1/b_2<0.96$ and $|b|^2\ge2\cdot1.05^2+(2.1-s)^2>2+s^2$; if $b_3\le-0.6$, then $b_2=s-b_3-b_4\ge s+1.2$ and $|b|^2\ge(s+1.2)^2+0.72>2+s^2$. Hence $-0.6<b_3<1.05$. In the flat charts the boxes cover $0\le s\le1.05$, and in charts T and Q also $0.15\le b_3\le1.05$ and $-0.6\le b_3\le-0.15$. The bound on $s$ suffices: if $\widehat q\le0.36$, then $|b|^2\ge1.64+s^2$, while three numbers that are at most $1/s$ and have sum $s$ have sum of squares at most $2/s^2+(s-2/s)^2$, which is less than $1.64+s^2$ when $s\ge1.05$.

Chart I uses $G=\operatorname{adj}(K)H\operatorname{adj}(K)$ and second-order Taylor models, and the other charts use \eqref{or:rowscaled} and first-order models. All boxes are discarded except the cube $\norm{x-x_0}_\infty\le10^{-3}$ around the regular point $x_0=(\sqrt3,1/\sqrt3,1/\sqrt3)$ of chart I, on which $\|I-CJ\|_\infty<0.33$. In all about $3.3\cdot10^5$ boxes are examined.

\begin{proof}[Proof of Theorem~\ref{thm:special}\textup{(ii)}]
By Lemma~\ref{lem:isogram} and a relabeling, every isodynamic tetrahedron corresponds to a point of chart I or of one of the five flat charts with $\widehat\sigma>0$. The computation shows that \eqref{or:test} fails at every such point except the regular tetrahedron.
\end{proof}

\begin{remark}
The smallest margins occur near the tetrahedra with $v_2=v_3=v_4$, which have a threefold axis. For the limiting infinite prism over an equilateral triangle ($s=0$, $b=0$), the normalized Gram entries of the Morley tetrahedron differ from $-\frac13$ by at most $0.0047$. In the flat limit $\widehat q=0$ the largest difference is at least $0.017$ in our samples.
\end{remark}

\subsection{Equal sums of opposite edges}
\label{sec:midsphere}

Consider a tetrahedron with (iii), and put $t_i=\frac12(\ell_{ij}+\ell_{ik}-\ell_{jk})$, where $\ell_{ij}=|A_iA_j|$ and $j,k$ are two indices different from $i$. (In this subsection $t_i$ does not denote the distances of \eqref{eq:distances}.) The equality of the three sums shows that $t_i$ does not depend on the choice of $j,k$, and then $\ell_{ij}=t_i+t_j$; the triangle inequality gives $t_i>0$. Conversely, $\ell_{ij}=t_i+t_j$ makes the three sums equal to $t_1+t_2+t_3+t_4$. So these are the tetrahedra whose vertices are the centers of four mutually tangent balls, with radii $t_i$. Let $O$ be the point with the same power $p$ with respect to the four spheres; with $\Delta$ as in Lemma~\ref{lem:midgram} below, $p=4/\Delta>0$. A short computation shows that the foot of the perpendicular from $O$ to $A_iA_j$ is the point of contact of the $i$th and $j$th spheres, at distance $\sqrt p$ from $O$. Hence these are also the tetrahedra with a sphere tangent to all six edges.

Neither exceptional tetrahedron belongs to this family. For $T_2$ we have $AB+CD=2$ and $AC+BD=2\bigl((21+4\sqrt6)/45\bigr)^{1/2}<2$, and for $T_1$ we have $AB+CD=1+\mu_*>3.35$ and $AC+BD=2\lambda_*<3.18$.

\begin{lemma}
\label{lem:midgram}
Let $|A_iA_j|=t_i+t_j$ with $t_i>0$, and put $k_i=1/t_i$, $\sigma=\sum_ik_i$ and $m_i=\sigma-2k_i$. Then
\begin{equation}
288V^2=32\Bigl(\prod_it_i\Bigr)^2\Delta,\ \text{ where }\ \Delta=\sigma^2-2\sum_ik_i^2,
\label{mid:Delta}
\end{equation}
so $\Delta>0$. The normal Gram matrix is
\begin{equation}
H_{ij}=\frac{m_im_j-\Delta}{\sqrt{(\Delta+m_i^2)(\Delta+m_j^2)}}\ \text{ for }i\ne j,
\label{mid:H}
\end{equation}
and the face areas are proportional to $k_i\sqrt{\Delta+m_i^2}$.
\end{lemma}

\begin{proof}
We argue as in the proof of Lemma~\ref{lem:isogram}. The matrix of squared edges is $D=a\one^T+\one a^T+D'$, where $a_i=t_i^2$, $D'=2tt^T-4T^2$ and $T=\diag(t_i)$. Adding $a\one^T+\one a^T$ to $D$ changes neither the determinant of $\left(\begin{smallmatrix}D&\one\\\one^T&0\end{smallmatrix}\right)$ nor the upper-left block of its inverse, as one sees by row and column operations with the border. By the Sherman--Morrison formula, $D'^{-1}=-\frac14\,\diag(k_i^2)+\frac18kk^T$ and $\det D'=-256\prod_it_i^2$. Hence $D'^{-1}\one=\frac18p$ with $p_i=k_im_i$, and $\one^TD'^{-1}\one=\Delta/8$. The Cayley--Menger determinant equals $-\det D'\cdot\one^TD'^{-1}\one$, which gives \eqref{mid:Delta}. The Gram matrix of the gradients of the barycentric coordinates is
\[
G=-2\Bigl(D'^{-1}-\frac{D'^{-1}\one\one^TD'^{-1}}{\one^TD'^{-1}\one}\Bigr)=\frac14\Bigl(2\diag(k_i^2)-kk^T+\frac{pp^T}{\Delta}\Bigr),
\]
that is, $G_{ij}=k_ik_j(m_im_j-\Delta)/(4\Delta)$ for $i\ne j$ and $G_{ii}=k_i^2(\Delta+m_i^2)/(4\Delta)$. This proves \eqref{mid:H} and the statement on the areas.
\end{proof}

The relation $\Delta=0$ is Descartes' relation for four mutually tangent circles in the plane. As $\Delta\to0$ the tetrahedron becomes flat, and the four balls become four tangent disks, one of them in the gap between the other three.

Formula \eqref{mid:H} is unchanged when $k$ is multiplied by a positive number, so we may assume $\sum_ik_i=1$, and after relabeling $k_1\le k_2\le k_3\le k_4$. Put $\varepsilon=\sqrt{k_1+k_2}$, $x=k_1/\varepsilon^2$ and $d=(k_4-k_3)/\varepsilon$. Then $k_1=x\varepsilon^2$, $k_2=(1-x)\varepsilon^2$, $k_{3,4}=\frac12(1-\varepsilon^2\mp\varepsilon d)$, and
\begin{equation}
\begin{gathered}
m=\bigl(1-2x\varepsilon^2,\;1-2(1-x)\varepsilon^2,\;\varepsilon a_3,\;\varepsilon a_4\bigr),\ \text{ where }\ a_3=\varepsilon+d\ \text{ and }\ a_4=\varepsilon-d,\\
\Delta=\varepsilon^2\widehat\Delta,\ \text{ where }\ \widehat\Delta=2-d^2-\varepsilon^2\bigl(1+2x^2+2(1-x)^2\bigr).
\end{gathered}
\label{mid:coords}
\end{equation}
The fundamental domain is given by $0\le x\le\frac12$, $d\ge0$, $(3-2x)\varepsilon^2+\varepsilon d\le1$ (that is, $k_2\le k_3$) and $\widehat\Delta>0$, and the regular tetrahedron is the point $(\varepsilon,x,d)=(1/\sqrt2,\frac12,0)$. With $E_i=\Delta+m_i^2$ for $i=1,2$ and $e_j=\widehat\Delta+a_j^2$ for $j=3,4$, formula \eqref{mid:H} becomes
\[
\begin{gathered}
H_{12}=\frac{m_1m_2-\Delta}{\sqrt{E_1E_2}},\ \ H_{34}=\frac{a_3a_4-\widehat\Delta}{\sqrt{e_3e_4}},\\
H_{ij}=\frac{m_ia_j-\varepsilon\widehat\Delta}{\sqrt{E_ie_j}}\ \text{ for }i\le2<j.
\end{gathered}
\]
These expressions are analytic wherever $\widehat\Delta>0$, also at $\varepsilon=0$ and at $x=0$. At $x=0$ the vertex $A_1$ has receded to infinity, but no rescaling is needed there. At $\varepsilon=0$ the vertices $A_1$ and $A_2$ have receded to infinity with $t_3=t_4$; there $m_1=m_2$, and the first two rows of $K$ coincide.

Chart I covers $\widehat\Delta\ge0.25$, $\varepsilon\ge0.1$ and uses $G=\operatorname{adj}(K)H\operatorname{adj}(K)$. Chart S covers $\widehat\Delta\ge0.25$, $\varepsilon\le0.12$ and replaces the second rows of $K$ and $\mathcal N$ by $(\mathrm{row}_2-\mathrm{row}_1)/\varepsilon$. Two identities make its entries analytic up to $\varepsilon=0$. First, $1-H_{12}=\Delta g_{12}$, where
\[
g_{12}=\frac1R\Bigl(\frac{\Delta+m_1^2+m_2^2}{R+m_1m_2}+1\Bigr)\ \text{ and }\ R=\sqrt{E_1E_2}.
\]
Second, for $l=3,4$,
\[
H_{2l}-H_{1l}=\Delta\,(m_2^2-m_1^2)\Bigl(\frac{m_l}{m_2\sqrt{E_1}+m_1\sqrt{E_2}}+\frac1{\sqrt{E_1}+\sqrt{E_2}}\Bigr)\frac1{\sqrt{E_1E_2E_l}},
\]
where $E_l=\Delta+m_l^2=\varepsilon^2e_l$ and $m_2^2-m_1^2=-2\varepsilon^2(1-2x)(m_1+m_2)$. So $H_{2l}-H_{1l}$ has the explicit factor $\varepsilon^3\widehat\Delta$, and $q_{2l}-q_{1l}=(H_{2l}-H_{1l})/\ph[q_{1l},q_{2l}]$, where $\ph[a,b]=\frac12(3-a^2-ab-b^2)$ is a divided difference of $\ph$.

Chart F covers the flat region $\widehat\Delta\le0.36$, with coordinates $(\varepsilon,x,\widehat\sigma)$, $\widehat\sigma=\widehat\Delta^{1/2}$, and $d$ obtained from \eqref{mid:coords}. There $m_1,m_2,m_3>0>m_4$, that is, $d>\varepsilon$. Indeed, if $d\le\varepsilon$, then
\[
(3-2x)\varepsilon^2+\varepsilon d\ge(3-2x)\varepsilon^2+d^2=2-\widehat\Delta+2x(1-2x)\varepsilon^2\ge1.64,
\]
contrary to $k_2\le k_3$. So the flat limit is a triangle $A_1A_2A_3$ with $A_4$ inside, and the first three rows of $K$ become equal. Rows $2$ and $3$ are replaced by $(\mathrm{row}_2-\mathrm{row}_1)/(\varepsilon\widehat\sigma)$ and $(\mathrm{row}_3-\mathrm{row}_1)/\widehat\sigma$. As in Section~\ref{sec:isodynamic}, the differences $1-H_{13}$, $1-H_{23}$, $1+H_{14}$, $1+H_{24}$ and $1+H_{34}$ have an explicit factor $\widehat\Delta$, for instance
\[
\begin{gathered}
1-H_{13}=\widehat\Delta\;\frac{(\Delta+m_1^2+\varepsilon^2a_3^2)/(\sqrt{E_1e_3}+m_1a_3)+\varepsilon}{\sqrt{E_1e_3}},\\
1+H_{34}=\widehat\Delta\;\frac{(\widehat\Delta+a_3^2+a_4^2)/(\sqrt{e_3e_4}-a_3a_4)-1}{\sqrt{e_3e_4}},
\end{gathered}
\]
and the entries of $\widetilde K$ and $\Gamma$ in \eqref{or:rowscaled} are analytic up to $\widehat\sigma=0$.

The branch and bound gives the following counts.
\begin{center}\small
\begin{tabular}{lllr}\toprule
Chart&Coordinates&Region&Boxes\\\midrule
I&$(\varepsilon,x,d)$&$\widehat\Delta\ge0.25$, $\varepsilon\ge0.1$&$28\,088$\\
S&$(\varepsilon,x,d)$&$\widehat\Delta\ge0.25$, $\varepsilon\le0.12$&$5\,028$\\
F&$(\varepsilon,x,\widehat\sigma)$&$\widehat\Delta\le0.36$&$11\,729$\\\bottomrule
\end{tabular}
\end{center}
The initial boxes cover $0\le\varepsilon\le0.7072$, $0\le x\le\frac12$, $0\le d\le1.42$ and $0\le\widehat\sigma\le0.6$. They contain the fundamental domain, because $2\varepsilon^2=2(k_1+k_2)\le k_1+k_2+k_3+k_4=1$ and $d^2<2$. Chart I uses second-order Taylor models, and charts S and F first-order ones. All boxes are discarded except the cube of half-width $10^{-3}$ around the regular point $(\varepsilon,x,d)=(1/\sqrt2,\frac12,0)$, on which $\|I-CJ\|_\infty<0.67$.

\begin{proof}[Proof of Theorem~\ref{thm:special}\textup{(iii)}]
By Lemma~\ref{lem:midgram} and a relabeling, every tetrahedron with equal sums of opposite edges corresponds to a point of one of the three charts, with $\varepsilon>0$ and $\widehat\sigma>0$. The computation shows that \eqref{or:test} fails at every such point except the regular tetrahedron.
\end{proof}

\begin{remark}
This computation is much smaller than the one in Section~\ref{sec:isodynamic}. At $\varepsilon=0$ the normalized Gram entries of the Morley tetrahedron are $-\frac12$ for the pairs $\{1,2\}$, $\{3,4\}$ and $-\frac14$ for the other pairs, for all $x$ and $d$ that we tested, so the margin there is $\frac16$. In the flat limit the largest difference from $-\frac13$ is at least $0.019$ in our samples.
\end{remark}

In the families of Sections~\ref{sec:orthocentric}--\ref{sec:midsphere} the edges have the forms $\ell_{ij}^2=e_i+e_j$, $\ell_{ij}=w_iw_j$ and $\ell_{ij}=t_i+t_j$, and the normal Gram matrices are, up to diagonal scaling, $\diag(1+y_i^2)-yy^T$ in the acute orthocentric case, $3vv^T+Q(3I-J)$ in the isodynamic case and $mm^T+\Delta(2I-J)$ in the third family. The closed forms rely on this structure, and the method should apply to other three-parameter families of the same kind.

\subsection{Three equal edges at a vertex}
\label{sec:apex}

Now suppose that $A_4A_1=A_4A_2=A_4A_3$. Then the foot of the altitude from $A_4$ is the circumcenter $O$ of the face $A_1A_2A_3$. We scale so that this face has circumradius $1$, write $C_i$ for its angle at $A_i$ and $c_i=\cos C_i$, and put $h=|A_4O|$. The family has three parameters up to similarity, and its members have no symmetry unless two of the $C_i$ are equal. Neither exceptional tetrahedron belongs to it. The three edges at a vertex of $T_2$ have lengths $1,\lambda,\lambda$ with $\lambda^2=(21+4\sqrt6)/45<1$, and those at a vertex of $T_1$ have lengths $1,\lambda_*,\lambda_*$ or $\lambda_*,\lambda_*,\mu_*$, where $1.57<\lambda_*<1.59$ and $\mu_*>2.35$.

\begin{lemma}
\label{lem:apexgram}
With $r_i=\sqrt{h^2+c_i^2}$ for $i\le3$, the normal Gram matrix is
\begin{equation}
H_{i4}=-\frac{c_i}{r_i}\ \text{ and }\ H_{ij}=\frac{c_ic_j-h^2c_k}{r_ir_j},\ \text{ where }\ \{i,j,k\}=\{1,2,3\}.
\label{apex:H}
\end{equation}
\end{lemma}

\begin{proof}
Place $O$ at the origin of the plane $z=0$ and $A_4$ at $(0,0,h)$. Let $m_i$ be the unit normal, in that plane, of the side of $A_1A_2A_3$ opposite $A_i$, pointing towards $A_i$. The distance from $O$ to that side is $|c_i|$, and $O$ lies on the same side of it as $A_i$ exactly when $C_i$ is acute, so the side lies on the line $m_i\cdot p=-c_i$. The face $F_i$ lies in the plane $h\,m_i\cdot p-c_iz+c_ih=0$, and the left-hand side is positive at $A_i$. Hence $n_i=(hm_i,-c_i)/r_i$ for $i\le3$, and $n_4=(0,0,1)$. Since $m_i\cdot m_j=-\cos C_k$, this gives \eqref{apex:H}.
\end{proof}

After relabeling $A_1,A_2,A_3$ we may assume $C_1\le C_2\le C_3$, so that $C_2<\pi/2$. Put $x_i=\tan(C_i/2)$ and $w=(1+h^2)^{-1/2}$. The fundamental domain is
\[
0<x_1\le x_2<1,\ \ x_2^2+2x_1x_2\le1,\ \ 0<w<1,
\]
where the second condition says $C_2\le C_3$. Here $c_i=(1-x_i^2)/(1+x_i^2)$ for $i=1,2$ and $c_3=-\cos(C_1+C_2)$, and \eqref{apex:H} becomes
\[
H_{i4}=-\frac{c_iw}{\sqrt{D_i}},\ \ H_{ij}=\frac{c_ic_jw^2-(1-w^2)c_k}{\sqrt{D_iD_j}},\ \text{ where }\ D_i=1-w^2+w^2c_i^2.
\]
The regular tetrahedron is the point $x_1=x_2=w=1/\sqrt3$. These expressions are analytic for $0\le w<1$. At $w=0$ the apex has receded to infinity and the tetrahedron tends to an infinite prism, and at $x_1=0$ the side $A_2A_3$ has shrunk to a point and $H_{23}=-1$. No rescaling is needed there. The matrix $K$ becomes singular in two other ways.

The first is the needle $x_2\to0$. Then $C_1$, $C_2$ and $\pi-C_3$ tend to $0$, the base shrinks to a point of its circumcircle, $H_{12}\to1$, and the first two rows of $K$ become equal. Chart N uses the coordinates $(x_2,r,w)$ with $r=x_1/x_2$, and replaces the second rows of $K$ and $\mathcal N$ by $(\mathrm{row}_2-\mathrm{row}_1)/x_2$. With $\widehat s_i=\sin C_i/x_2$, that is, $\widehat s_1=2r/(1+x_1^2)$ and $\widehat s_2=2/(1+x_2^2)$, one has $1-H_{12}=x_2^2g_{12}$, where
\[
g_{12}=\frac{1-w^2}{R}\Bigl(\frac{\widehat s_1^2+\widehat s_2^2-(1+w^2)x_2^2\widehat s_1^2\widehat s_2^2}{R+c_1c_2}+\widehat s_1\widehat s_2\Bigr)\ \text{ and }\ R=\sqrt{D_1D_2}.
\]
For $l=3,4$ the entries $H_{1l}$ and $H_{2l}$ are the values of one expression at $(c,c')=(c_1,c_2)$ and at $(c,c')=(c_2,c_1)$. For instance $H_{13}=\Phi(c_1,c_2)$ and $H_{23}=\Phi(c_2,c_1)$, where $\Phi(c,c')=(cc_3w^2-(1-w^2)c')/\sqrt{(1-w^2+w^2c^2)D_3}$. So $H_{2l}-H_{1l}$ is $c_2-c_1$ times a divided difference. The script computes it by the rules $[uv]=[u]\,v_2+u_1[v]$, $[1/u]=-[u]/(u_1u_2)$ and $[\sqrt u\,]=[u]/(\sqrt{u_1}+\sqrt{u_2})$, where $u_1,u_2$ are the two values of $u$ and $[c]=1$, $[c']=-1$. Since $c_2-c_1=-2(x_2^2-x_1^2)/((1+x_1^2)(1+x_2^2))$ has the factor $x_2^2$, all entries of $\widetilde K$ and $\Gamma$ are analytic up to $x_2=0$.

The second is the flat limit $h\to0$. If $C_3<\pi/2$, the apex falls inside the base, $H_{ij}\to1$ for $i<j\le3$, and the first three rows of $K$ become equal, as in chart T of Section~\ref{sec:isodynamic}. If $C_3>\pi/2$, the limit is a convex quadrilateral, and rows $1,2$ and rows $3,4$ become equal, as in chart Q. Near $C_3=\pi/2$ the limits depend on the direction of approach, as near the arc $b_3=0$ there. We use the coordinates $(c_2,a,b)$ with $c_3=ac_2$ and $h=bc_2$. Then $-1\le a\le1$, with $a=-1$ for $C_1=0$ and $a=1$ for $C_2=C_3$, the condition $C_1\le C_2$ becomes $ac_2+2c_2^2\le1$, and $c_1=\sqrt{(1-c_2^2)(1-a^2c_2^2)}-ac_2^2$. Here $r_1=\sqrt{b^2c_2^2+c_1^2}$, and with $B=\sqrt{1+b^2}$ and $A=\sqrt{a^2+b^2}$ formula \eqref{apex:H} becomes
\begin{equation}
\begin{gathered}
H_{14}=-\frac{c_1}{r_1},\ \ H_{24}=-\frac1B,\ \ H_{34}=-\frac aA,\\
H_{12}=\frac{c_1-ab^2c_2^2}{r_1B},\ \ H_{13}=\frac{ac_1-b^2c_2^2}{r_1A},\ \ H_{23}=\frac{a-b^2c_1}{BA}.
\end{gathered}
\label{apex:flat}
\end{equation}
These expressions are analytic also at $c_2=0$. This matters, because the formulas in $(x_1,x_2,w)$ break down at the corner $C_1=0$, $C_2=C_3=\pi/2$, $h=0$ of the domain, where $c_2=c_3=0$. The parameter $b$ plays the role of $\widehat\sigma$ in Section~\ref{sec:isodynamic}, and $a$ that of $b_3$. The charts FT, FQ, AM, AT and AQ are built like the charts T, Q, AM, AT and AQ there, with $a=\rho\sin\psi$ and $b=\rho\cos\psi$ in the last three. The differences have explicit factors $b^2$, for instance
\[
1-H_{12}=\frac{b^2}{r_1B}\Bigl(\frac{c_1^2+c_2^2+b^2c_2^2}{r_1B+c_1}+ac_2^2\Bigr)\ \text{ and }\ 1+H_{14}=\frac{b^2c_2^2}{r_1(r_1+c_1)}.
\]
Chart W, where $0.5\le b\le4$, and chart Z, where $c_2\le0.3h$ and the coordinates are $(c_2/h,a,h)$, need no rescaling.

Finally, chart NF covers the needle and the flat limit together, in the coordinates $(x_2,r,h)$. The limit configuration is a convex quadrilateral, and rows $2$ and $4$ are replaced by $(\mathrm{row}_2-\mathrm{row}_1)/(x_2h)$ and $(\mathrm{row}_4-\mathrm{row}_3)/h$. Here we use $P_{il}=(1+H_{il})/h^2$ for $i\le2<l$, which has an explicit form, for instance $P_{14}=1/(r_1(r_1+c_1))$, and compute $P_{2l}-P_{1l}$ with divided differences as in chart N. The complete formulas are in the script \texttt{check\_apex.py}.

The ten charts, and the number of boxes the script examines on each, are as follows.
\begin{center}\small
\begin{tabular}{lllr}\toprule
Chart&Coordinates&Region&Boxes\\\midrule
I&$(x_1,x_2,w)$&$x_2\ge0.25$, $w\le0.98$&$60\,512$\\
N&$(x_2,r,w)$&$x_2\le0.3$, $w\le0.93$&$59\,358$\\
NF&$(x_2,r,h)$&$x_2\le0.3$, $h\le0.4$&$6\,876$\\
W&$(c_2,a,b)$&$0.5\le b\le4$&$34\,108$\\
Z&$(c_2/h,a,h)$&$c_2\le0.3h$&$960$\\
FT&$(c_2,a,b)$&$b\le0.6$, $a\ge0.15$&$95\,864$\\
FQ&$(c_2,a,b)$&$b\le0.6$, $a\le-0.15$&$12\,466$\\
AM&$(c_2,\rho,\sin\psi)$&$\rho\le0.6235$, $|\sin\psi|\le0.8$&$15\,496$\\
AT&$(c_2,\rho,\cos\psi)$&$\rho\le0.6235$, $\sin\psi>0$, $\cos\psi\le0.8$&$13\,463$\\
AQ&$(c_2,\rho,\cos\psi)$&$\rho\le0.6235$, $\sin\psi<0$, $\cos\psi\le0.8$&$2\,215$\\\bottomrule
\end{tabular}
\end{center}
The last seven charts are restricted to $c_2\le0.84$ and $h\le0.21$. Together the charts cover the fundamental domain. Indeed, if $x_2\le0.3$, a point lies in chart N, or in chart NF when $w>0.93$, since then $h<0.4$. If $x_2\ge0.3$ and $w\le0.98$, it lies in chart I. If $x_2\ge0.3$ and $w>0.98$, then $c_2\le0.91/1.09<0.84$ and $h<0.21$. Such a point lies in FT or FQ if $b\le0.6$ and $|a|\ge0.15$, in one of the three arc charts if $b\le0.6$ and $|a|<0.15$, because then $\rho<0.62$, in W if $0.5\le b\le4$, and in Z if $b>4$. The initial boxes contain the charts, because $3x_1^2\le x_2^2+2x_1x_2\le1$ gives $x_1\le1/\sqrt3<0.578$ in chart I, and $0\le r\le1$ in charts N and NF.

Charts I, W and Z use $G=\operatorname{adj}(K)H\operatorname{adj}(K)$, and the other charts use \eqref{or:rowscaled}. Chart I uses second-order Taylor models and the other charts first-order ones. All boxes are discarded except the cube $\norm{x-x_0}_\infty\le5\cdot10^{-4}$ around the regular point $x_0=(1/\sqrt3,1/\sqrt3,1/\sqrt3)$ of chart I, on which $\|I-CJ\|_\infty<0.6$. In all about $3.0\cdot10^5$ boxes are examined.

\begin{proof}[Proof of Theorem~\ref{thm:special}\textup{(iv)}]
By Lemma~\ref{lem:apexgram} and a relabeling, every tetrahedron with $A_4A_1=A_4A_2=A_4A_3$ corresponds to a point of the fundamental domain, and hence to a point of one of the ten charts. The computation shows that \eqref{or:test} fails at every such point except the regular tetrahedron.
\end{proof}

\begin{remark}
The smallest margins occur along the line $x_1=x_2=1/\sqrt3$ of tetrahedra with an equilateral base. These have a threefold axis, and so has their Morley tetrahedron. Its normalized Gram entries are therefore $-\frac13+u$ for the pairs $i<j\le3$ and $-\frac13+v$ for the pairs $(i,4)$. The face normals of the Morley tetrahedron, weighted by the face areas, add up to $0$, which gives $1+2(-\frac13+u)=3(-\frac13+v)^2$, that is, $u=-v+\frac32v^2$. So only one condition is left on this line, and $u$ vanishes only at $h=\sqrt2$. But $u$ stays small: $|u|$ is about $0.0017$ at $h=1$, $0.0009$ at $h=1.8$ and $0.0047$ in the limit of the infinite prism, the prism of the remark in Section~\ref{sec:isodynamic}, and it is $0.017$ in the flat limit. At the needle the margin is $\frac16$.
\end{remark}

\subsection{Dihedral angles of at most \texorpdfstring{$95^\circ$}{95 degrees}}
\label{sec:acute}

Both exceptional tetrahedra have obtuse dihedral angles. In $T_2$ the dihedral angles at $AB$ and $CD$ are about $98.69^\circ$, and in $T_1$ the dihedral angle at $AB$ is about $103.43^\circ$. The next theorem shows that a counterexample to Conjecture~\ref{conj:four} must have a dihedral angle larger than $95^\circ$.

\begin{theorem}
\label{thm:acute}
If no dihedral angle of $A_1A_2A_3A_4$ is larger than $95^\circ$ and its Morley tetrahedron is regular, then $A_1A_2A_3A_4$ is regular. In particular, a tetrahedron with six acute dihedral angles and a regular Morley tetrahedron is regular.
\end{theorem}

Because of $T_2$, the bound $95^\circ$ cannot be replaced by a number larger than $98.69^\circ$.

Unlike the families above, the tetrahedra in Theorem~\ref{thm:acute} have five parameters up to similarity. In the variables $q_{ij}$ the problem is nevertheless simpler. Let $\theta_{ij}$ be the dihedral angle between $F_i$ and $F_j$. Then $q_{ij}=2\cos(\theta_{ij}/3)$ and $H_{ij}=\ph(q_{ij})=-\cos\theta_{ij}$, so $\theta_{ij}\le95^\circ$ exactly when $q_{ij}\ge q_{\min}$, where $q_{\min}=2\cos(19\pi/108)=1.70223\ldots$; a right angle corresponds to $q_{ij}=\sqrt3$. We use the six coordinates $q=(q_{12},q_{13},q_{14},q_{23},q_{24},q_{34})$ on the cube $[q_{\min},2]^6$. There $K$, $H=\ph[K]$, $G=\operatorname{adj}(K)H\operatorname{adj}(K)$ and the $\rho_{ij}$ of \eqref{or:test} are polynomials in $q$, so no root of a cubic has to be enclosed. The rescalings of the previous sections were needed only where some $H_{ij}$ tends to $1$, that is, where a dihedral angle tends to $\pi$. This cannot happen here, and all formulas are defined on the closed cube.

Not every point of the cube comes from a tetrahedron, and two properties of tetrahedra help to discard boxes. First, $\det H=0$, because the four normals are linearly dependent. Second, $H$ is positive semidefinite of rank $3$ and its kernel is spanned by the vector $S$ of face areas, so $\operatorname{adj}(H)=\mu SS^T$ with $\mu>0$, and all entries of $\operatorname{adj}(H)$ are positive. A box is discarded if \eqref{or:test} fails on it, if $\det H$ has no zero on it, or if some entry of $\operatorname{adj}(H)$ is negative on it. A relabeling of the vertices permutes the $q_{ij}$. So we may assume that $q_{12}$ is the largest of the six numbers. The four relabelings that map $\{1,2\}$ to itself permute the pairs $\{1,3\}$, $\{1,4\}$, $\{2,3\}$, $\{2,4\}$ simply transitively, so we may also assume $q_{13}\ge q_{14},q_{23},q_{24}$. Boxes on which these inequalities fail are discarded too.

The script \texttt{check\_acute.py} starts from a grid of $4^6$ boxes on $[1.70223,2]^6$, which contains $[q_{\min},2]^6$. On each box it tries first-order Taylor models, and second-order ones if the box survives. All boxes are discarded except the cube $\norm{q-q_0}_\infty\le5\cdot10^{-4}$ around the regular point $q_0$, all of whose coordinates equal $2\cos(\frac13\arccos\frac13)=1.83398\ldots$. On this cube the map $q\mapsto(\rho_{12},\rho_{13},\rho_{14},\rho_{23},\rho_{24},\rho_{34})$ from $\R^6$ to $\R^6$ is injective, by the test of Section~\ref{sec:bb} with $\|I-CJ\|_\infty<0.35$. In all $153\,171$ boxes are examined.

\begin{proof}[Proof of Theorem~\ref{thm:acute}]
Let $A_1A_2A_3A_4$ be a tetrahedron whose dihedral angles are at most $95^\circ$ and whose Morley tetrahedron is regular. After a relabeling, its point $q$ lies in $[q_{\min},2)^6$ and satisfies the two inequalities above, as well as \eqref{or:test}, $\det H=0$ and $\operatorname{adj}(H)>0$. The computation shows that $q_0$ is the only such point. Since $H$ determines the tetrahedron up to similarity, $A_1A_2A_3A_4$ is regular.
\end{proof}

\subsection{An equilateral face}
\label{sec:eqface}

Finally, let the face $A_1A_2A_3$ be equilateral. Neither exceptional tetrahedron has an equilateral face: the faces of $T_2$ have sides $1,\lambda,\lambda$, and those of $T_1$ have sides $1,\lambda_*,\lambda_*$ or $\lambda_*,\lambda_*,\mu_*$. In this family the dihedral angles along the sides of $A_1A_2A_3$ are convenient coordinates.

\begin{lemma}
\label{lem:eqgram}
Let $A_1A_2A_3$ be equilateral, and let $\theta_i$ be the dihedral angle at its side opposite $A_i$. Then
\[
H_{i4}=-\cos\theta_i\ \text{ and }\ H_{ij}=\cos\theta_i\cos\theta_j-\tfrac12\sin\theta_i\sin\theta_j\ \text{ for }i<j\le3.
\]
Conversely, three numbers $\theta_1,\theta_2,\theta_3\in(0,\pi)$ are the dihedral angles along the sides of an equilateral face of a tetrahedron if and only if $\cot\theta_1+\cot\theta_2+\cot\theta_3>0$, and then they determine the tetrahedron up to similarity.
\end{lemma}

\begin{proof}
Place $A_1A_2A_3$ in the plane $z=0$ with inradius $\frac12$, and $A_4$ in the half-space $z>0$. Let $m_i$ be the unit normal, in that plane, of the side opposite $A_i$, pointing towards $A_i$. Then $n_4=(0,0,1)$ and $n_i=(\sin\theta_i\,m_i,-\cos\theta_i)$ for $i\le3$, and $m_i\cdot m_j=-\cos\frac\pi3=-\frac12$; this gives the formulas for $H$. Let $d_i(p)$ be the signed distance from a point $p$ of the plane to the side opposite $A_i$, positive on the side of $A_i$, so that $d_1+d_2+d_3=\frac32$. A point $(p,h)$ with $h>0$ lies on the plane through that side which makes the dihedral angle $\theta_i$ with $A_1A_2A_3$ exactly when $d_i(p)=h\cot\theta_i$. Adding the three conditions gives $h(\cot\theta_1+\cot\theta_2+\cot\theta_3)=\frac32$. So the apex exists if and only if the sum of the cotangents is positive, and then it is unique.
\end{proof}

A relabeling of $A_1,A_2,A_3$ permutes the $\theta_i$. Put $\tau_i=\tan(\theta_i/2)$; we may assume $\tau_1\le\tau_2\le\tau_3$. Then
\begin{equation}
H_{i4}=\frac{\tau_i^2-1}{\tau_i^2+1}\ \text{ and }\ H_{ij}=\frac{(1-\tau_i^2)(1-\tau_j^2)-2\tau_i\tau_j}{(1+\tau_i^2)(1+\tau_j^2)},
\label{eq:eqface}
\end{equation}
the condition of the lemma becomes $e_2>e_1e_3$, where $e_1,e_2,e_3$ are the elementary symmetric functions of $\tau_1,\tau_2,\tau_3$, and the regular tetrahedron is $\tau_1=\tau_2=\tau_3=1/\sqrt2$. The differences that matter have explicit factors:
\[
1-H_{ij}=\frac{2(\tau_i^2+\tau_i\tau_j+\tau_j^2)}{(1+\tau_i^2)(1+\tau_j^2)},\ \ 1+H_{i4}=\frac{2\tau_i^2}{1+\tau_i^2},\ \ 1-H_{i4}=\frac2{1+\tau_i^2}.
\]
In these variables two kinds of degeneration cause no trouble. At $\tau_1=0$ the apex has moved to $A_1$, in a direction given by $\theta_2$ and $\theta_3$, and on the surface $e_2=e_1e_3$ it has receded to infinity; in both cases no $H_{ij}$ tends to $1$. By the identities above, $H_{ij}\to1$ only if $\tau_i,\tau_j\to0$ or $\tau_i,\tau_j\to\infty$, and $H_{i4}\to1$ only if $\tau_i\to\infty$. Since $e_2>e_1e_3$ forces $\tau_1\to0$ when $\tau_3\to\infty$, this happens in five ways: $\tau_1,\tau_2\to0$, so that $H_{12}\to1$; all $\tau_i\to0$, when the apex falls into the face and the first three rows of $K$ become equal; $\tau_3\to\infty$, so that $H_{34}\to1$; $\tau_1,\tau_2\to0$ and $\tau_3\to\infty$, when the apex falls beyond the side $A_1A_2$ and the limit is a convex quadrilateral; and $\tau_2,\tau_3\to\infty$, when the apex falls beyond the vertex $A_1$, which then lies inside the triangle $A_2A_3A_4$, and rows $2$, $3$, $4$ become equal.

With $\sigma_i=1/\tau_i$, the six charts and the number of boxes examined on each are as follows.
\begin{center}\small
\begin{tabular}{lllr}\toprule
Chart&Coordinates&Region&Boxes\\\midrule
I&$(\tau_1,\tau_2,\tau_3)$&$\tau_2\ge0.25$, $\tau_3\le4$&$9\,661$\\
E1&$(\tau_2,\tau_1/\tau_2,\tau_3)$&$\tau_2\le0.25$, $0.3\le\tau_3\le4$&$13\,874$\\
T&$(\tau_3,\tau_2/\tau_3,\tau_1/\tau_2)$&$\tau_3\le0.3$&$8\,324$\\
E2&$(\sigma_3,\tau_2,\tau_1)$&$\sigma_3\le0.25$, $0.25\le\tau_2\le4$&$2\,396$\\
Q&$(\tau_2,\tau_1/\tau_2,\sigma_3)$&$\tau_2\le0.25$, $\sigma_3\le0.25$&$157$\\
T$'$&$(\sigma_2,\sigma_3/\sigma_2,\tau_1)$&$\sigma_2\le0.25$, $\tau_1\le0.3$&$141$\\\bottomrule
\end{tabular}
\end{center}
Chart I uses $G=\operatorname{adj}(K)H\operatorname{adj}(K)$ and second-order Taylor models, and the other charts use \eqref{or:rowscaled} and first-order models. Chart E1 replaces the second rows of $K$ and $\mathcal N$ by $(\mathrm{row}_2-\mathrm{row}_1)/\tau_2$, chart E2 replaces the fourth rows by $(\mathrm{row}_4-\mathrm{row}_3)/\sigma_3$, and chart Q does both. Chart T replaces the second and third rows by $(\mathrm{row}_2-\mathrm{row}_1)/\tau_2$ and $(\mathrm{row}_3-\mathrm{row}_1)/\tau_3$, as chart U of Section~\ref{sec:orthocentric} does, and chart T$'$ replaces the second and fourth rows by $(\mathrm{row}_2-\mathrm{row}_3)/\sigma_2$ and $(\mathrm{row}_4-\mathrm{row}_3)/\sigma_3$. All the differences that occur have explicit factors, for instance
\[
H_{23}-H_{13}=-2(\tau_2-\tau_1)\frac{(1-\tau_3^2)(\tau_1+\tau_2)+\tau_3(1-\tau_1\tau_2)}{(1+\tau_1^2)(1+\tau_2^2)(1+\tau_3^2)}\ \text{ and }\ H_{j4}-H_{j3}=2\sigma_3\frac{\sigma_3(\tau_j^2-1)+\tau_j}{(1+\tau_j^2)(1+\sigma_3^2)};
\]
the complete list is in the script \texttt{check\_eqface.py}.

Together the charts cover the domain. If $\tau_3\le4$, a point lies in chart I when $\tau_2\ge0.25$, in chart E1 when $\tau_2\le0.25$ and $\tau_3\ge0.3$, and in chart T when $\tau_3\le0.3$. If $\tau_3\ge4$, it lies in chart Q, E2 or T$'$ according as $\tau_2\le0.25$, $0.25\le\tau_2\le4$ or $\tau_2\ge4$. In the last case, dividing $e_2>e_1e_3\ge\tau_3e_3$ by $e_3$ gives $1/\tau_1>\tau_3-1/\tau_2-1/\tau_3$, hence $\tau_1<\sigma_3/(1-\sigma_2\sigma_3-\sigma_3^2)<0.29$.

No cube around the regular point is needed. In chart I a box is also discarded if all six $H_{ij}$ are at most $0.0871$ on it. Since $-\cos95^\circ>0.0871$, all dihedral angles are then less than $95^\circ$, and Theorem~\ref{thm:acute} applies. In all $34\,553$ boxes are examined.

\begin{proof}[Proof of Theorem~\ref{thm:special}\textup{(v)}]
By Lemma~\ref{lem:eqgram} and a relabeling, a tetrahedron with an equilateral face $A_1A_2A_3$ corresponds to a point $\tau_1\le\tau_2\le\tau_3$ with $e_2>e_1e_3$, and this point lies in one of the six charts. If all its dihedral angles are at most $95^\circ$ and its Morley tetrahedron is regular, it is regular by Theorem~\ref{thm:acute}. Otherwise no box containing the point is discarded because of its angles, and the computation shows that \eqref{or:test} fails at the point.
\end{proof}

\begin{remark}
The smallest margin occurs near $\theta_1=\theta_2\approx51.3^\circ$, $\theta_3\approx102.2^\circ$, where all six normalized Gram entries of the Morley tetrahedron are within about $0.0019$ of $-\frac13$; these tetrahedra have a reflection symmetry. The computation is much smaller than the one in Section~\ref{sec:apex}, for two reasons: Theorem~\ref{thm:acute} covers the neighborhood of the regular point, and the variables $\tau_i$ absorb the degenerations in which the apex approaches a vertex of the face.
\end{remark}

\section{Tetrahedra with an isosceles Morley tetrahedron}
\label{sec:isomorley}

By the direct theorem of \cite{TranActa} (Corollary~\ref{cor:inherited}(c)), the Morley tetrahedron of an isosceles tetrahedron is isosceles, and Theorem~\ref{thm:curve} gives a curve of nonisosceles tetrahedra with the same property. In this section we describe all such tetrahedra near the three known solutions, and we show that the set also contains tetrahedra without any symmetry.

A tetrahedron is isosceles if and only if its four faces have the same area. The faces of an isosceles tetrahedron are congruent. Conversely, if $S$ is constant, then $HS=0$ gives $\sum_{j\ne i}H_{ij}=-1$ for each $i$. Adding these equations for $i=1,2$ and subtracting them for $i=3,4$ gives $H_{12}=H_{34}$; in the same way $H_{13}=H_{24}$ and $H_{14}=H_{23}$, and \eqref{eq:edge} shows that opposite edges are equal.

Let $C$ be the matrix with entries $C_{ij}=B_{ij}/\sqrt{B_{ii}B_{jj}}$, where $B=K^{-1}HK^{-1}$. By the remark after Proposition~\ref{prop:regularity}, $C$ is the normal Gram matrix of the Morley tetrahedron. So the Morley tetrahedron is isosceles if and only if $C\one=0$. Only three of these four conditions are independent: if $S'$ is the positive vector of face areas of the Morley tetrahedron, then $CS'=0$, so $S'^TC\one=0$, and if three entries of $C\one$ vanish, so does the fourth. Up to similarity a tetrahedron has five parameters, so we expect the solutions to form surfaces. The isosceles tetrahedra form one of them.

\begin{proposition}
\label{prop:isolocal}
Label $T_1$ and $T_2$ as in Section~\ref{sec:intervals}.
\begin{enumerate}[label=\textup{(\alph*)}]
\item A tetrahedron close enough to the regular tetrahedron or to $T_2$ whose Morley tetrahedron is isosceles is itself isosceles.
\item The tetrahedra close enough to $T_1$ whose Morley tetrahedra are isosceles form a real analytic surface through $T_1$. None of them is isosceles, and all of them are invariant under the half-turn $(1\,4)(2\,3)$. Those with a further symmetry form, after relabeling, the curve of Theorem~\ref{thm:curve}.
\end{enumerate}
\end{proposition}

\begin{proof}
For a symmetric $4\times4$ matrix $C$ and $x=(q_{ij},d_i)$ as in Section~\ref{sec:intervals}, put $F^C(x)=KD(3C)DK-\ph[K]$, where $\ph[K]$ has the diagonal entries $\ph(1)=1$. For $C_0=\frac13(4I-J)$ this is the map $F$ of \eqref{eq:tenequations}. For a tetrahedron, put $d_i=(B_{ii}/3)^{1/2}$; then $B=3DCD$, and $F^C(x)=0$ because $H=KBK$. Section~\ref{sec:intervals} shows that $DF$ is invertible at the zeros $x_*$ that belong to the regular tetrahedron, to $T_1$ and to $T_2$. By the implicit function theorem there are neighborhoods $W$ of $x_*$ and $V$ of $C_0$ and an analytic map $\xi$ on $V$ such that, for $C\in V$, the only zero of $F^C$ in $W$ is $\xi(C)$. Two tetrahedra close to the same one of the three solutions and with the same matrix $C$ therefore have the same $q_{ij}$, hence the same $H$, and they are similar with corresponding labels. In other words, the Morley construction is injective near each of the three solutions.

Let $X$ be close to one of them, $X_0$, and suppose that the Morley tetrahedron of $X$ is isosceles. Each double transposition $\pi$ of the labels is induced by a half-turn of an isosceles tetrahedron, so $\pi C\pi^T=C$. By Lemma~\ref{lem:equivariance}, the relabeled tetrahedron $\pi X$ has the matrix $\pi C\pi^T=C$. If $\pi$ also fixes $X_0$, then $\pi X$ is close to $X_0$, and injectivity shows that $\pi X$ and $X$ are similar with corresponding labels. Since both have the same edges, $\pi$ preserves the edge lengths of $X$ and is induced by an isometry. The three double transpositions fix the regular tetrahedron and $T_2$, and a tetrahedron invariant under all three is isosceles. This proves (a). Of the three, only $(1\,4)(2\,3)$ fixes $T_1$.

For (b), note that the normal Gram matrices of isosceles tetrahedra close to $C_0$ form a real analytic surface $\mathcal D$. Indeed, such a tetrahedron has equal dihedral angles $\alpha$, $\beta$, $\gamma$ at the three pairs of opposite edges, and $H\one=0$ becomes $\cos\alpha+\cos\beta+\cos\gamma=1$. For $C\in\mathcal D\cap V$ let $(q,d)=\xi(C)$. The matrix $H=\ph[K]=KD(3C)DK$ is positive semidefinite of rank $3$, and its kernel is spanned by $K^{-1}D^{-1}\one$, which is positive near $T_1$ by \eqref{eq:intervalgeometry}. As in the proof of Proposition~\ref{prop:regularity}, $H$ is the normal Gram matrix of a tetrahedron, and $B=K^{-1}HK^{-1}=3DCD$ shows that its Morley tetrahedron has the normal Gram matrix $C$. So the tetrahedra near $T_1$ with isosceles Morley tetrahedra are the images of $\mathcal D\cap V$ under an analytic map whose inverse $X\mapsto C$ is also analytic, and they form an analytic surface. No tetrahedron close to $T_1$ is isosceles. A symmetry of a tetrahedron $X$ close to $T_1$ induces a permutation $\pi$ of the labels with $\pi X=X$. Then $\pi T_1$ is close to $\pi X=X$ and hence to $T_1$, so $\pi$ fixes $T_1$ and is one of $(1\,4)$, $(2\,3)$, $(1\,4)(2\,3)$. If $X$ is on the surface and has one of the reflections, it has both, so it has four equal cross edges and belongs, after relabeling, to the family of Section~\ref{sec:family}. There the tetrahedra near $T_1$ with isosceles Morley tetrahedra form the curve of Theorem~\ref{thm:curve}.
\end{proof}

Away from the three known solutions the picture changes, and a symmetry is no longer forced.

\begin{theorem}
\label{thm:asymmetric}
There is a real analytic two-parameter family of tetrahedra with six different edge lengths, and hence without any symmetry, whose Morley tetrahedra are isosceles but not regular. More precisely, put $A_1=(0,0,0)$, $A_2=(1,0,0)$, $A_3=(c_1,c_2,0)$ and $A_4=(d_1,d_2,d_3)$. For every $(c_1,d_1)$ with $|c_1-0.2284|\le2\cdot10^{-5}$ and $|d_1+0.5025|\le2\cdot10^{-5}$ there is exactly one $(c_2,d_2,d_3)$ with
\[
|c_2-0.54717|\le4\cdot10^{-4},\ \ |d_2-0.10601|\le4\cdot10^{-4},\ \ |d_3-1.50055|\le4\cdot10^{-4}
\]
for which the Morley tetrahedron of $A_1A_2A_3A_4$ is isosceles, and it depends analytically on $(c_1,d_1)$.
\end{theorem}

\begin{proof}
A symmetry of a tetrahedron permutes the vertices and preserves the edge lengths. If the six lengths are different, it maps every edge to itself and therefore fixes every vertex.

The inward normals of the faces are proportional to explicit vectors in $c_1,c_2,d_1,d_2,d_3$, for instance $(0,d_3,-d_2)$ for $F_3$ and $(0,0,1)$ for $F_4$. We use $G=\operatorname{adj}(K)H\operatorname{adj}(K)$, a positive multiple of $B$. Then $C\one=0$ is equivalent to $E_i=\sum_jG_{ij}/\sqrt{G_{jj}}=0$ for $i=1,2,3$. The script \texttt{check\_isomorley.py} applies the Krawczyk test \cite{Krawczyk,Neumaier} to $E=(E_1,E_2,E_3)$ as a function of $z=(c_2,d_2,d_3)$, uniformly in the parameters. Let $Z$ be the cube $\norm{z-z_0}_\infty\le5\cdot10^{-4}$ around $z_0=(0.54717056,0.10601489,1.50054912)$, let $P$ be a parameter rectangle slightly larger than the one in the theorem, and let $Y$ be a fixed matrix close to $(\partial E/\partial z)(z_0)^{-1}$. With an interval enclosure $\mathcal J$ of $\partial E/\partial z$ on $Z\times P$, the script checks that $z_0-YE(z_0,p)+(I-Y\mathcal J)(Z-z_0)$ lies in the interior of $Z$ for every $p\in P$ and that $\norm{I-Y\mathcal J}_\infty<0.16$. So for each $p\in P$ there is exactly one zero in $Z$, and $\partial E/\partial z$ is invertible there; the implicit function theorem gives the analytic dependence on $p$. The image of the Krawczyk operator lies in the smaller cube of the theorem, which is contained in $Z$. On $Z\times P$ the script also checks that
\[
\begin{gathered}
|A_1A_3|<0.594<0.945<|A_2A_3|<0.947<|A_1A_2|=1\\
<1.585<|A_1A_4|<1.587<1.725<|A_3A_4|<1.728<2.12<|A_2A_4|,
\end{gathered}
\]
so the six edges have different lengths, and that $\rho_{12}<0$ in \eqref{or:test}, so the Morley tetrahedron is not regular.
\end{proof}

At $(c_1,d_1)=(0.2284,-0.5025)$ the Morley tetrahedron has opposite edges of lengths $0.11568$, $0.11474$ and $0.11359$.

\begin{remark}
\label{rem:isoglobal}
Numerically the picture is as follows. We solved $C\one=0$ together with two random linear equations in $(c_1,c_2,d_1,d_2,d_3)$ from $3000$ random starting points, and kept the solutions whose shortest edge is at least $2\%$ of the longest. They were of three kinds: $254$ isosceles tetrahedra, $1036$ nonisosceles tetrahedra with a half-turn symmetry and $883$ tetrahedra without symmetry. We started from $50$ of the solutions without symmetry and moved along the solution surface so that the distance to the nearest shape with a half-turn symmetry decreases. Each time we reached a tetrahedron with a half-turn symmetry at which the Jacobian matrix of the three conditions has rank $2$. So the surfaces without symmetry seem to branch off the half-turn surfaces along curves, as in a symmetry breaking bifurcation, and near such a curve their points come in pairs exchanged by the half-turn. For the family of Theorem~\ref{thm:asymmetric} the branch point lies near the tetrahedron with $A_1A_2=A_3A_4=1$, $A_1A_4=A_2A_3\approx0.930$, $A_1A_3\approx0.414$ and $A_2A_4\approx1.579$, which is invariant under $(1\,3)(2\,4)$. The Morley tetrahedra of all these solutions are close to regular. For the nonisosceles ones the longest and the shortest Morley edge differ by less than $2\%$ of the longest in half of the cases and by at most about $10\%$ in all of them; this reflects the strong contraction of the Morley map.

These observations, and those in Section~\ref{sec:isoproduct}, are not used in any proof. They can be reproduced with the script \texttt{explore\_isomorley.py} (Appendix~\ref{sec:reproducibility}).
\end{remark}

\subsection{Two conjectures on opposite edges}
\label{sec:isoproduct}

For a tetrahedron $ABCD$ put
\[
P=(AB^2-CD^2)(AC^2-BD^2)(AD^2-BC^2).
\]
A relabeling of the vertices permutes the three factors and changes the signs of an even number of them, so $P$ depends only on the tetrahedron. A double transposition changes the signs of exactly two factors. Hence, if no two opposite edges are equal, the vertices can be labeled so that the three factors have the same sign. If this sign is positive, every edge at $A$ is longer than the opposite edge, and if it is negative, every edge at $A$ is shorter than the opposite edge. So $P<0$ exactly when some vertex has three edges each shorter than the opposite edge, and $P>0$ exactly when some vertex has three edges each longer than the opposite edge. For the family of Theorem~\ref{thm:asymmetric} the vertex $A_4$ is of the second kind.

\begin{conjecture}
\label{conj:isoproduct}
Suppose that the Morley tetrahedron of $ABCD$ is isosceles. Then
\begin{enumerate}[label=\textup{(\roman*)}]
\item $P=(AB^2-CD^2)(AC^2-BD^2)(AD^2-BC^2)\ge0$;
\item if $AB=CD$, then $AC=BD$ or $AD=BC$.
\end{enumerate}
\end{conjecture}

Since $P=0$ exactly when two opposite edges are equal, the two parts together say that $P\ge0$, with equality only when two pairs of opposite edges are equal, that is, when $ABCD$ has a half-turn symmetry. In (ii) both equalities may hold, as they do for isosceles tetrahedra. Part (ii) cannot ask for a third pair: $T_1$ can be labeled so that $AB=CD$ and $AD=BC$, and it is not isosceles, and the same holds for the tetrahedra near $T_1$ in Proposition~\ref{prop:isolocal}(b).

Neither part implies the other, but they are closely linked. Suppose that a solution $X_0$ has $AB=CD$ and no other pair of equal opposite edges. Near $X_0$ the factors $AC^2-BD^2$ and $AD^2-BC^2$ keep their signs, so $P$ has the sign of $AB^2-CD^2$ there. If the solutions near $X_0$ form a surface that crosses the hypersurface $AB=CD$, then $P$ takes both signs near $X_0$, against (i). Conversely, on a connected set of solutions without symmetry $P$ can change sign only at a solution with exactly one pair of equal opposite edges. So (ii) implies that $P$ has a constant sign on each such set, and (i) would follow from the sign of $P$ at one point of each set.

The evidence is numerical. In all our experiments the tetrahedra with an isosceles Morley tetrahedron and no symmetry had $P>0$. These are the $883$ solutions of Remark~\ref{rem:isoglobal} and $520$ more from a second search that also admits nearly degenerate shapes, with shortest edge down to about $0.6\%$ of the longest. We minimized a scale-invariant normalization of $P$ over the solution set from $300$ starting points. The optimizer converged in $166$ cases, and each time to a tetrahedron that is isosceles or has a half-turn symmetry, where $P=0$. A search for solutions with $AB=CD$ in which the half-turn solutions were removed by deflation found nothing from $3000$ starting points. With $AB=1$ and $AB=CD$ replaced by $CD=AB+0.02$, the same search found $124$ solutions from $600$ starting points, and on them $\min(|AC-BD|,|AD-BC|)$ was as large as $0.144$. For the gap $0.002$ it found $43$ solutions, and the largest value of this minimum dropped to $0.045$; so a second pair of opposite edges becomes nearly equal as the gap shrinks. For a tetrahedron with $AC=BC$ and $AD=BD$ we have $P=-(AB^2-CD^2)(AC^2-AD^2)^2$, so the conjecture predicts $AB<CD$ when $AC\ne AD$. We found $347$ solutions of this kind with $AC\ne AD$, and all of them satisfy $AB<CD$.

To describe the solutions with $AB=CD$, put
\[
\tau_{AB}=\frac{AB^2-CD^2}{AB^2+CD^2},
\]
and define $\tau_{AC}$, $\tau_{AD}$ in the same way. Together with two coordinates for the three sums $AB^2+CD^2$, $AC^2+BD^2$, $AD^2+BC^2$, normalized to total $1$, they are coordinates on the shape space of labeled tetrahedra, in which the isosceles tetrahedra are given by $\tau=0$. Denote by $A'$, $B'$, $C'$, $D'$ the vertices of the Morley tetrahedron associated with the faces opposite $A$, $B$, $C$, $D$, and let $\tau'_{AB}$, $\tau'_{AC}$, $\tau'_{AD}$ be the corresponding quantities of $A'B'C'D'$.

On the tetrahedra with $AB=CD$ the coordinates are $(y,\tau_{AC},\tau_{AD})$, where $y$ stands for the two normalized sums. The relabeling $(A\,D)(B\,C)$ preserves $AB=CD$, fixes $y$ and $\tau_{AD}$, and changes the sign of $\tau_{AC}$; by Lemma~\ref{lem:equivariance} it changes the signs of $\tau'_{AB}$ and $\tau'_{AC}$ and fixes $\tau'_{AD}$. The relabeling $(A\,C)(B\,D)$ acts in the same way with the roles of $AC$ and $AD$ exchanged. Hence, for tetrahedra with $AB=CD$,
\[
\tau'_{AC}=\tau_{AC}\,a,\ \ \tau'_{AD}=\tau_{AD}\,b,\ \ \tau'_{AB}=\tau_{AC}\tau_{AD}\,k
\]
with analytic functions $a$, $b$, $k$ of $(y,\tau_{AC},\tau_{AD})$. A tetrahedron with $AB=CD$ and no other pair of equal opposite edges has an isosceles Morley tetrahedron exactly when $a=b=k=0$. These are three conditions in a four-parameter family, so one would expect a curve of such tetrahedra; part (ii) of Conjecture~\ref{conj:isoproduct} says that there is none. The first two conditions say that $A'C'=B'D'$ and $A'D'=B'C'$, so that the Morley tetrahedron has a half-turn symmetry exchanging $A'$ with $B'$ and $C'$ with $D'$. On the set where this happens, the numerical evidence suggests a sign rule.

\begin{conjecture}
\label{conj:isosign}
Suppose that $AB=CD$ and that the Morley tetrahedron satisfies $A'C'=B'D'$ and $A'D'=B'C'$. Then
\[
(A'B'-C'D')(AC-BD)(AD-BC)\ge0,
\]
with equality only if $AC=BD$ or $AD=BC$.
\end{conjecture}

Conjecture~\ref{conj:isosign} implies part (ii) of Conjecture~\ref{conj:isoproduct}: if the Morley tetrahedron is isosceles, then $A'B'=C'D'$, the product vanishes, and so $AC=BD$ or $AD=BC$. In the notation above it says that $k>0$ wherever $a=b=0$ and $\tau_{AC}\tau_{AD}\ne0$. Note that $k$ takes both signs on the tetrahedra with $AB=CD$; the conjecture concerns only those whose Morley tetrahedron has the half-turn symmetry. We solved $a=b=0$ with $AB=CD$, $AC>BD$ and $AD>BC$ from $3000$ random starting points and obtained $83$ solutions. All of them have $A'B'>C'D'$, and $k$ lies between $0.040$ and $0.221$ on them. The smallest values occur near the tetrahedron $X_*$ described below.

Finally we look at the neighborhood of the isosceles tetrahedra. At an isosceles tetrahedron $X_0$ let $g_{AB}(X_0)$ be the partial derivative of $\tau'_{AB}$ with respect to $\tau_{AB}$, and define $g_{AC}$, $g_{AD}$ in the same way.

\begin{proposition}
\label{prop:isonormal}
Let $X_0$ be an isosceles tetrahedron with $g_{AC}(X_0)\ne0$ and $g_{AD}(X_0)\ne0$. Then every tetrahedron close enough to $X_0$ whose Morley tetrahedron is isosceles has $AC=BD$ and $AD=BC$; in particular it satisfies Conjecture~\ref{conj:isoproduct}.
\end{proposition}

\begin{proof}
Write the coordinates as $(y,\tau_{AB};\tau_{AC},\tau_{AD})$ and let $\Phi=(\tau'_{AC},\tau'_{AD})$. The relabeling $(A\,B)(C\,D)$ fixes $y$ and $\tau_{AB}$ and changes the signs of $\tau_{AC}$ and $\tau_{AD}$. By Lemma~\ref{lem:equivariance} it also changes the signs of $\tau'_{AC}$ and $\tau'_{AD}$. So $\Phi$ is odd in $(\tau_{AC},\tau_{AD})$ and vanishes when $\tau_{AC}=\tau_{AD}=0$. The relabeling $(A\,C)(B\,D)$ fixes $X_0$ and $\tau_{AC}$, $\tau'_{AC}$, and changes the sign of $\tau_{AD}$; hence $\partial\tau'_{AC}/\partial\tau_{AD}=0$ at $X_0$, and in the same way $\partial\tau'_{AD}/\partial\tau_{AC}=0$. The derivative of $\Phi$ with respect to $(\tau_{AC},\tau_{AD})$ at $X_0$ is therefore $\diag(g_{AC}(X_0),g_{AD}(X_0))$, which is invertible. By the implicit function theorem, near $X_0$ the equation $\Phi=0$ has no solutions other than $\tau_{AC}=\tau_{AD}=0$.
\end{proof}

At the regular tetrahedron the three numbers $g$ are equal to $-0.0143$, so the Morley map reduces a small difference between opposite edges by a factor of about $70$ and reverses its sign. Numerically, two of the three numbers $g$ vanish simultaneously at only one isosceles tetrahedron up to relabeling. We parametrized the isosceles tetrahedra by the angles of their faces and computed the numbers $g$ on a grid of step $1^\circ$. The only grid cells in which two of them change sign lie around the tetrahedron $X_*$ with
\[
AB=CD\approx0.8803\ \text{ and }\ AC=AD=BC=BD=1,
\]
where $g_{AC}=g_{AD}=0$. If this is so, Proposition~\ref{prop:isonormal} gives Conjecture~\ref{conj:isoproduct} near every isosceles tetrahedron that is not close to $X_*$. On the tetrahedra with $AC=BD$ and $AD=BC$ the function $\tau'_{AB}$ is odd in $\tau_{AB}$, so nonisosceles solutions with this half-turn symmetry can leave the isosceles tetrahedra only where $g_{AB}$ vanishes. The computations suggest that the surfaces of Remark~\ref{rem:isoglobal} with a half-turn symmetry meet the isosceles tetrahedra along such curves, and that the surfaces without symmetry reach them at $X_*$. Near $X_*$ these solutions satisfy $\tau_{AB}\approx h\,\tau_{AC}\tau_{AD}$ with $h\approx1.257$, so $P>0$ there. At $X_*$ we find $k\approx0.039>0$. Both observations agree with the conjectures.

\section{Four equal Morley edges in a skew quadrilateral}
\label{sec:skew}

If a pair of opposite edges is removed from a tetrahedron, the remaining four edges form a skew quadrilateral. In this section we consider tetrahedra whose Morley tetrahedron satisfies
\begin{equation}
A'B'=B'C'=C'D'=D'A'.
\label{eq:skewmorley}
\end{equation}
A tetrahedron has four equal edges in the skew quadrilateral $A'B'C'D'$ exactly when it has a reflection interchanging $A'$ and $C'$ and a reflection interchanging $B'$ and $D'$. By Corollary~\ref{cor:inherited}, \eqref{eq:skewmorley} therefore holds whenever $AB=BC=CD=DA$. Near the regular tetrahedron there are no other solutions.

\begin{proposition}
\label{prop:skewlocal}
There is a neighborhood of the regular tetrahedron in which every tetrahedron satisfying \eqref{eq:skewmorley} has $AB=BC=CD=DA$.
\end{proposition}

\begin{proof}
The proof of Proposition~\ref{prop:isolocal} shows that the Morley construction is injective near the regular tetrahedron: two tetrahedra close to it whose Morley tetrahedra have the same normalized normal Gram matrix are similar with corresponding labels. We may take the neighborhood $U$ in this statement invariant under relabelings, for instance as the set of tetrahedra whose six dihedral angles differ from $\arccos\frac13$ by less than some $\varepsilon>0$. Let $X\in U$ satisfy \eqref{eq:skewmorley}, and let $\pi X$ be the tetrahedron $X$ with the labels $A$ and $C$ interchanged. Then $\pi X\in U$. By Lemma~\ref{lem:equivariance}, the Morley tetrahedron of $\pi X$ is that of $X$ with $A'$ and $C'$ interchanged. By \eqref{eq:skewmorley} this relabeling does not change its edge lengths, so the two Morley tetrahedra have the same normalized normal Gram matrix, and $\pi X$ and $X$ are similar with corresponding labels. They have the same edge lengths, so the similarity is a congruence, and $AB=CB$, $AD=CD$. Interchanging $B$ and $D$ instead gives $AB=AD$ and $CB=CD$.
\end{proof}

Away from the regular tetrahedron a symmetry is no longer forced. Condition \eqref{eq:skewmorley} consists of three equations, so its solutions should form surfaces. Numerical searches for them need some care near flat tetrahedra. When $ABCD$ flattens to a convex quadrilateral with diagonals $AC$ and $BD$, its Morley tetrahedron tends, numerically, to an isosceles tetrahedron. Near flat shapes two of the three conditions in \eqref{eq:skewmorley} therefore hold up to a small error, and the equations have approximate solutions that are not close to any solution. Our searches discard shapes with $6V$ smaller than $10^{-4}$ times the cube of the longest edge. A random search from $600$ starting points gave $218$ solutions: $17$ with $AB=BC=CD=DA$, $49$ with only one of the two reflections, $14$ with only the half-turn that interchanges $A$ with $C$ and $B$ with $D$, and $138$ without any symmetry. We have not found a sign condition for them similar to Conjecture~\ref{conj:isoproduct}.

The picture changes if we also require $AC=BD$. The four conditions should then define curves, and every solution we found has a symmetry.

\begin{conjecture}
\label{conj:skew}
Suppose that the Morley tetrahedron of $ABCD$ satisfies \eqref{eq:skewmorley} and that $AC=BD$. Then $ABCD$ has a nontrivial symmetry. More precisely, one of the following holds.
\begin{enumerate}[label=\textup{(\alph*)}]
\item $ABCD$ has a half-turn that interchanges the edges $AC$ and $BD$, that is, $AB=CD$ or $AD=BC$.
\item $ABCD$ has no such half-turn, but it has a plane of symmetry through $AC$ or through $BD$, that is, $AB=AD$ and $CB=CD$, or $AB=BC$ and $AD=CD$. Then $A'C'\ne B'D'$, and the shorter of these two edges joins the two Morley vertices interchanged by the reflection.
\end{enumerate}
\end{conjecture}

In case (a), Corollary~\ref{cor:inherited} gives $A'C'=B'D'$. Then $A'B'C'D'$ has four equal edges in a skew quadrilateral and its other two edges are equal, so its symmetry group has order eight. If the conjecture holds, the two cases can therefore be told apart from the Morley tetrahedron alone. The first assertion of the conjecture is the main point; the second only lists the symmetries that occur. Case (a) can also be recognized from the dihedral angles: Proposition~\ref{prop:cotangent}, with $B$ and $C$ interchanged, shows that a tetrahedron with $AC=BD$ has $AB=CD$ or $AD=BC$ if and only if $\theta_{AC}=\theta_{BD}$. Some tetrahedra have symmetries of both kinds. A tetrahedron with $AB=BC=CD=DA$ and $AC=BD$, for example, has two planes of symmetry, but it belongs to case (a).

The evidence is again numerical. A random search from $1500$ starting points gave $660$ solutions. In $552$ of them $ABCD$ has a half-turn as in (a); $133$ of these are isosceles, and $55$ of those have $AB=BC=CD=DA$. The other $108$ have only a plane of symmetry. No solution without symmetry occurred. All $552$ solutions of the first kind have $A'C'=B'D'$, and all $108$ of the second kind satisfy the rule in (b).

We also followed $40$ of the $138$ solutions of \eqref{eq:skewmorley} without symmetry found above along their solution surfaces, so that $|AC-BD|$ decreases. In $34$ cases the path ended at a tetrahedron with a half-turn as in (a), where $AC=BD$; the other $6$ ran into nearly degenerate shapes first. Finally, a search for solutions of the four equations in which the symmetric solutions were removed by deflation found none from $1000$ starting points, while the same method applied to \eqref{eq:skewmorley} alone found $23$ solutions from $100$ starting points. As in Section~\ref{sec:isomorley}, none of these computations is used in a proof; they are experiments~9 and~10 of the script \texttt{explore\_isomorley.py}.

\section{Summary of the proved cases}
\label{sec:summary}

The table lists what is known about a tetrahedron $ABCD$ with a regular Morley tetrahedron under various additional hypotheses.

\begin{center}\small
\begin{tabular}{p{.43\textwidth}p{.47\textwidth}}\toprule
Additional hypothesis & Conclusion\\\midrule
$ABCD$ isosceles & regular or $T_2$ (Theorems~\ref{thm:disphenoid}, \ref{t2:thm:classification})\\
$t_1=t_2=t_3=t_4$ & isosceles, hence regular or $T_2$ (Theorem~\ref{thm:equaldistances})\\
two pairs of equal opposite edges & four equal cross edges (Theorem~\ref{thm:twopairs})\\
one reflection symmetry & a second reflection (Theorem~\ref{thm:reflection})\\
four equal cross edges & regular, $T_1$ or $T_2$ (Theorem~\ref{sym:thm:classification})\\
any nontrivial symmetry & regular, $T_1$ or $T_2$ (Corollary~\ref{cor:symmetric})\\
$ABCD$ orthocentric & regular (Theorem~\ref{thm:special}(i))\\
$ABCD$ isodynamic & regular (Theorem~\ref{thm:special}(ii))\\
$AB+CD=AC+BD=AD+BC$ & regular (Theorem~\ref{thm:special}(iii))\\
$AB=AC=AD$ & regular (Theorem~\ref{thm:special}(iv))\\
$ABC$ equilateral & regular (Theorem~\ref{thm:special}(v))\\
all dihedral angles at most $95^\circ$ & regular (Theorem~\ref{thm:acute})\\
shape close to the regular one & regular (Theorem~\ref{thm:localregular})\\
shape close to $T_1$ or $T_2$ & equal to it (Section~\ref{sec:intervals})\\
$AB=CD$ only & open; equivalent to $\theta_{AB}=\theta_{CD}$ (Proposition~\ref{prop:cotangent})\\
none & open (Conjecture~\ref{conj:four})\\\bottomrule
\end{tabular}
\end{center}

What remains open is the case of a tetrahedron with trivial symmetry group that is not close to one of the three known solutions, has a dihedral angle larger than $95^\circ$, and does not belong to one of the five families of Section~\ref{sec:special}. For isosceles Morley tetrahedra, Section~\ref{sec:isomorley} describes the situation near the three known solutions and shows that a symmetry is not forced in general. Conjectures~\ref{conj:isoproduct} and~\ref{conj:isosign} propose a replacement, supported by numerical evidence and by the local statement of Proposition~\ref{prop:isonormal}. For Morley tetrahedra with four equal edges in a skew quadrilateral a symmetry is forced near the regular tetrahedron (Proposition~\ref{prop:skewlocal}), and Conjecture~\ref{conj:skew} predicts one in general when $AC=BD$.

\section{Possible approaches}
\label{tet:directions}

By Corollaries~\ref{sym:cor:equivalence} and~\ref{cor:symmetric}, to prove Conjectures~\ref{conj:four} and~\ref{conj:classification} it suffices to show that a regular Morley tetrahedron forces some symmetry of the original tetrahedron. We describe some possible approaches.

\subsection{Forcing a first reflection}

By Proposition~\ref{prop:defect}, a reflection symmetry corresponds to the vanishing of the nonnegative defect $\varepsilon_{ij}$. One could try to combine the six defects with the positive kernel condition for $H$ and the identity \eqref{eq:regularity}, and to show that at least one defect vanishes; Theorem~\ref{thm:reflection} would then give the second reflection. The bounds of Proposition~\ref{prop:distancebounds} restrict the variables, but are far from sufficient for this.

\subsection{One pair of equal opposite edges}

The case $AB=CD$ is a natural test. Here one has to prove \eqref{eq:missing}, that is, $q_{12}=q_{34}$; Theorem~\ref{thm:twopairs} then finishes the argument. One can try to combine the edge condition \eqref{eq:edge} with some of the entries of \eqref{eq:regularity}. Any such elimination has to keep track of the positivity of the face areas, since polynomial relations alone have roots that do not come from tetrahedra.

\subsection{Larger dihedral angles}

The proof of Theorem~\ref{thm:acute} only uses that the six numbers $q_{ij}$ lie in the cube $[q_{\min},2]^6$. The same computation can be tried on $[c,2]^6$ for smaller $c$. Below $c\approx1.6793$, the value of $q_{ij}$ at the obtuse dihedral angles of $T_2$, small cubes around $T_2$ have to be added, and below $c\approx1.6487$ also cubes around $T_1$. As $c$ approaches $1$ the tetrahedron can become flat, $K$ becomes singular, and rows have to be rescaled as in Sections~\ref{sec:orthocentric}--\ref{sec:apex}, now in many more ways. Carried out down to $c=1$, this would prove Conjecture~\ref{conj:classification}. It is a computation in six variables, instead of the ten variables of the next subsection.

\subsection{A global computation}

The system $F=0$ of \eqref{eq:tenequations} consists of ten polynomial equations of degree at most four in the ten unknowns $q_{ij}$ and $d_i$. It is invariant under the action of $S_4$ and under $d\mapsto-d$, and its admissible solutions lie in a bounded region: $1<q_{ij}<2$, $0<d_i\le1/\sqrt2$ by Proposition~\ref{prop:defect}, and $\max_id_i/\min_id_i<2/\sqrt3$ by Proposition~\ref{prop:distancebounds}. This suggests two computational approaches.

First, one can cover the closure of this region by boxes and exclude zeros by rational interval arithmetic, as in Section~\ref{sec:intervals}, except in small neighborhoods of the three known solutions. Near the boundary this is harder: when some $q_{ij}$ tends to $1$ or $2$ the tetrahedron degenerates, and one needs to know which variables can approach their limits and whether the equations can hold along such a sequence. Our numerical experiments with a collapsing edge found no solutions there, but they do not give a proof. Section~\ref{sec:special} carries out this program for five three-parameter families. There the degenerate configurations could be handled by rescaling rows of the linear system $K\widetilde m=\mathcal N$, and the same device may work in general.

Second, homotopy continuation or monodromy methods could be used to compute all isolated complex solutions of $F=0$, keeping the real solutions that satisfy the conditions of Proposition~\ref{prop:regularity}. The B\'ezout bound $4^{10}$ is large, but the symmetries help, and by Corollary~\ref{cor:symmetric} a new solution would give $24$ labeled tetrahedra, that is, $48$ solutions of $F=0$ because of the symmetry $d\mapsto-d$. Such a computation would give strong evidence; a proof would require certified path tracking or a separate argument for completeness.

\subsection{The \texorpdfstring{$p$}{p}-secting construction}

For the construction of \cite[Theorem 6]{TranActa} with other parameters our methods do not apply directly. Let a $p$-secting plane through an edge with dihedral angle $\theta$ make the angle $\theta/p$ with the first face through that edge. Then the distances of its points to the second and to the first face are in the ratio $\sin((p-1)\theta/p)/\sin(\theta/p)$. For an integer $p$ this ratio is $U_{p-2}(c)$, and $\cos\theta=T_p(c)$, where $c=\cos(\theta/p)$ and $T_p$, $U_{p-2}$ are Chebyshev polynomials. So for integer parameters the regularity conditions are again polynomial equations; for $p=3$ they are those of Section~\ref{sec:gram}. We have not studied Conjecture~\ref{conj:psector} for parameters other than $e=f=g=3$.

\section{Open problems}
\label{sec:openproblems}

\begin{enumerate}[label=\arabic*.]
\item Prove or disprove Conjecture~\ref{conj:four}, or one of its equivalent forms: the classification Conjecture~\ref{conj:classification}, the symmetry Conjecture~\ref{conj:symmetry}, or the matrix form Conjecture~\ref{conj:positivity}.
\item Prove Conjecture~\ref{conj:onepair}, that is, the implication \eqref{eq:missing}.
\item Describe the set of tetrahedra whose Morley tetrahedron is isosceles. Near the regular tetrahedron and $T_2$ it consists of isosceles tetrahedra, near $T_1$ it is a surface of tetrahedra with a half-turn symmetry (Proposition~\ref{prop:isolocal}), and it contains a two-parameter family of tetrahedra without symmetry (Theorem~\ref{thm:asymmetric}). Remark~\ref{rem:isoglobal} suggests that it is a union of surfaces of these three kinds, in which the surfaces without symmetry branch off the half-turn surfaces and reach the isosceles tetrahedra only at $X_*$ (Section~\ref{sec:isoproduct}). Is this the whole picture, and how many such surfaces are there?
\item Prove Conjectures~\ref{conj:isoproduct} and~\ref{conj:isosign}. By the discussion after Conjecture~\ref{conj:isoproduct}, part (ii), together with the sign of $P$ where the surfaces without symmetry meet the half-turn surfaces or $X_*$, would give part (i). Natural first cases are the tetrahedra with a reflection symmetry, where the conjecture says that $AC=BC\ne AD=BD$ forces $AB<CD$, and the neighborhood of $X_*$.
\item Prove Conjecture~\ref{conj:skew}. By Proposition~\ref{prop:cotangent}, case (a) amounts to $\theta_{AC}=\theta_{BD}$. A first step would be the neighborhood of the curves of solutions with a half-turn.
\item Let $\theta_*\approx98.69^\circ$ be the size of the two obtuse dihedral angles of $T_2$. Must a tetrahedron whose dihedral angles are all smaller than $\theta_*$ and whose Morley tetrahedron is regular be regular? Theorem~\ref{thm:acute} gives this for dihedral angles of at most $95^\circ$.
\item Prove the analogue of Theorem~\ref{thm:special} for larger families without symmetry. A natural next case is the four-parameter family of tetrahedra with an isosceles face. It contains $T_1$ and $T_2$, and Conjecture~\ref{conj:classification} predicts that they and the regular tetrahedron are its only members with a regular Morley tetrahedron.
\item Find the analogue of the Gram formulation for the construction of \cite[Theorem 6]{TranActa} with other parameters, and decide Conjecture~\ref{conj:psector}.
\item Is there a geometric reason why $T_2$ is expressible by square roots while $T_1$ is defined by a sextic with Galois group $S_6$?
\item Study the same problems for simplices of higher dimension. The answer then depends on the dimension; see \cite{TranSimplices}.
\end{enumerate}

\begin{appendices}

\section{Polynomial inequality for two opposite pairs}
\label{app:bernstein}

This appendix proves the incompatibility used in Theorem~\ref{thm:twopairs}. Set $x=P$, $t=R/P$ and
\[
z=\frac{wP(P+3)}{R(R+3)}.
\]
By \eqref{eq:wbound}, all three variables lie in $(0,1)$. Write $u=x(1+t)$, $v=tx^2$, $D_0=x+3$, $W=zt(tx+3)$ and $w=W/D_0$.
Define the polynomials
\begin{align*}
N={}&2u(2u+3)D_0^2-6uWD_0-v(W^2+4WD_0+11D_0^2),\\
C_0={}&v(u+5)D_0+v(u+1)W+2Wu(u+3),\\
E={}&\frac{NC_0-6(u+1)(v+2u+4)v(D_0-W)(D_0+W)^2}{x^2},\\
B_0={}&(u^2+3u-6)D_0^2-6WD_0+2v(W^2+WD_0-D_0^2).
\end{align*}
The numerator defining $E$ is divisible by $x^2$. Equation \eqref{eq:Ezero} becomes $E=0$, while \eqref{eq:Bpositive} becomes $B_0>0$.

\begin{lemma}
\label{lem:bernstein}
On $(0,1)^3$, $B_0\ge0$ implies $E<0$.
\end{lemma}

\begin{proof}
Let $b_{n,i}(X)=\binom niX^i(1-X)^{n-i}$, and define
\[
Q(x,t,z)=\sum_{i=0}^2\sum_{j=0}^4\sum_{k=0}^1
q_{ijk}b_{2,i}(x)b_{4,j}(t)b_{1,k}(z).
\]
The nonzero coefficients, all positive, are listed below.

\begin{center}
\begin{tabular}{cc}
\toprule
$(i,j,k)$ & $q_{ijk}$\\
\midrule
$(0,1,1)$ & $30902484965806161/7009635898549648$\\
$(0,2,1)$ & $16399650927464661/7009635898549648$\\
$(0,3,1)$ & $129949208779311819/14019271797099296$\\
$(0,4,0)$ & $48/7$\\
$(0,4,1)$ & $24/7$\\
$(1,1,1)$ & $2483086725481191/7009635898549648$\\
$(1,3,1)$ & $382341243899593387/56077087188397184$\\
$(1,4,0)$ & $869/56$\\
$(2,0,1)$ & $386384/44499$\\
$(2,1,1)$ & $512549/44499$\\
$(2,2,1)$ & $2202979/38142$\\
$(2,4,0)$ & $73/6$\\
\bottomrule
\end{tabular}
\end{center}

Put $R_0=-E-B_0Q$. In the tensor Bernstein basis of degree $(6,9,3)$, all $280$ coefficients of $R_0$ are nonnegative, and $252$ are strictly positive. If $R_0=\sum_\nu a_\nu X^\nu$, its coefficient at the multi-index $\iota$ is
\[
r_\iota=\sum_{\nu\le\iota}a_\nu
\prod_{j=1}^3\frac{\binom{\iota_j}{\nu_j}}{\binom{n_j}{\nu_j}},
\]
where $n=(6,9,3)$, so each sign is a finite computation with the rational numbers above.

Every Bernstein basis function is positive in the open cube. Hence $R_0>0$ there, while $Q\ge0$. The identity $-E=B_0Q+R_0$ proves the lemma and completes the polynomial step of Theorem~\ref{thm:twopairs}.
\end{proof}

\section{Polynomial elimination for one reflection}
\label{app:reflection}

We use the notation from the proof of Theorem~\ref{thm:reflection}. Put $e_1=(\sqrt2,0,1)$, $e_2=(-\sqrt2,0,1)$, $e_3=(0,\sqrt2,-1)$ and $e_4=(0,-\sqrt2,-1)$.
Choose $M_i=Le_i/(2\sqrt2)$. The normals satisfy $n_i=\sum_jq_{ij}d_je_j$, giving
\begin{align*}
n_1&=(\sqrt2a(1-p),\sqrt2(bq-cr),a(1+p)-bq-cr),\\
n_3&=(0,\sqrt2(b-cs),2aq-b-cs),\\
n_4&=(0,\sqrt2(bs-c),2ar-bs-c).
\end{align*}
The difference $n_1-n_2$ gives $a^2=(p+2)/8$.

Assume $b\ne c$, and relabel so that $b>c$. Define $R=(b+c)/a$, $U=(b-c)/a$, $X=R^2$ and $Y=U^2$.
Thus $X>Y>0$. The vectors $n_3+n_4$ and $n_3-n_4$ lie in the $yz$ plane, are perpendicular, and have squared lengths $(s+1)^2(2-s)$ and $(s-1)^2(s+2)$. Comparison of their coordinates yields a signed parameter $T$ such that
\[
T^2=\frac{2-s}{2(s+2)},\ \text{ that is, }\ s=\frac{2(1-2T^2)}{1+2T^2},\ \text{ with }\ 0<|T|<1/\sqrt6,
\]
and $q+r=\alpha R$, $q-r=\beta U$, where
\[
\alpha=\frac{(s+1)(1+2T)}2\ \text{ and }\ \beta=\frac{(s-1)(T^{-1}-1)}2.
\]
The squared lengths also give
\begin{equation}
(p+2)\left(\frac{X}{s+2}+\frac{Y}{2-s}\right)=4.
\label{eq:reflectionlength}
\end{equation}

The sign of $T$ is forced by the positive kernel. Set
\[
D_* = p+1-\frac{\alpha^2X}{s+1}+\frac{\beta^2Y}{s-1}.
\]
Block elimination gives $\det K=(1-p)(1-s^2)D_*$, so $D_*<0$ by Lemma~\ref{lem:negative}. Rescale the positive area vector so that
\[
KS=(1,1,a/b,a/c)^T.
\]
Solving this system and using \eqref{eq:reflectionlength}, its first coordinate is
\[
S_1=S_2=-\frac{8T(s+2)}{(p+2)(X-Y)D_*}.
\]
Since $S_1>0$, we must have $T>0$. The physical parameter interval is therefore $(0,1/\sqrt6)$.

\subsection{The necessary rational system}

Put $\gamma=(s+1)T$ and $\delta=(s-1)/(2T)$. Equation \eqref{eq:reflectionlength} is the case $i=0$ of the following system. The cases $i=1,2$ come from the equations $n_1\cdot(n_3+n_4)=\ph(q)+\ph(r)$ and $n_1\cdot(n_3-n_4)=\ph(q)-\ph(r)$: use $a^2=(p+2)/8$ and \eqref{eq:reflectionlength}, and multiply by $8/R$ and $8/U$ respectively. The system is
\begin{equation}
A_iX+B_iY+C_i=0\ \text{ for }i=0,1,2,
\label{eq:reflectionlinear}
\end{equation}
where $A_0=(p+2)/(s+2)$, $B_0=(p+2)/(2-s)$, $C_0=-4$ and
\begin{align*}
A_1&=-(p+2)\gamma\alpha+\alpha^3,\\
B_1&=(p+2)((\alpha+\beta)(s+1)-\gamma\beta)+3\alpha\beta^2,\\
C_1&=2\gamma(p+2)(p+1)-12\alpha,\\
A_2&=(p+2)(-(\alpha+\beta)(s-1)-\delta\alpha)+3\beta\alpha^2,\\
B_2&=-(p+2)\delta\beta+\beta^3,\\
C_2&=2\delta(p+2)(p+1)-12\beta.
\end{align*}
The last equation uses $U\ne0$. The unit length of $n_1$ gives
\begin{equation}
(p+2)\left[2(p-1)^2+\frac{(\alpha+\beta)^2XY}{2}
+\left(p+1-\frac{\alpha X+\beta Y}{2}\right)^2\right]=8.
\label{eq:reflectionunit}
\end{equation}

For $j=1,2$, set $D_j=A_0B_j-A_jB_0$, $X_j=B_0C_j-B_jC_0$ and $Y_j=A_jC_0-A_0C_j$.
Every solution of \eqref{eq:reflectionlinear} satisfies $X_j=D_jX$ and $Y_j=D_jY$, including when $D_j=0$. It therefore satisfies $F=G_1=G_2=0$, where
\[
F=D_1C_2+A_2X_1+B_2Y_1
\]
and
\begin{align*}
G_j={}&(p+2)\bigl[8(p-1)^2D_j^2+2(\alpha+\beta)^2X_jY_j\\
&+(2(p+1)D_j-\alpha X_j-\beta Y_j)^2\bigr]-32D_j^2.
\end{align*}
No division by $D_j$ has been made, so this necessary system loses no singular case.

\subsection{The two resultants}

For a polynomial in $p$ with coefficients in $\mathbb Q(T)$, let $\pp$ mean the following operation: clear the denominators, divide the coefficients by their polynomial greatest common divisor in $\mathbb Q[T]$, and multiply by a rational constant so that the resulting integer coefficients have greatest common divisor $1$. The result is unique up to sign. Put $\widetilde F=\pp(F)$ and $\widetilde G_j=\pp(G_j)$ for $j=1,2$.
To divide without introducing new denominators, start with $V=\widetilde G_j$ and, while $\deg_p V\ge3$, replace $V$ by
\[
\operatorname{lc}_p(\widetilde F)V
-\operatorname{lc}_p(V)p^{\deg_p V-3}\widetilde F.
\]
Let $H_j=\pp(V)$ at the end of this procedure. Substitution of the preceding formulas gives $\deg_p\widetilde F=3$ and $\deg_pH_j=2$.

The denominators, polynomial factors removed by these operations, and the leading coefficient of $\widetilde F$ have no nonconstant factors other than $T$, $1+2T^2$, $2T^2-3$, $6T^2-1$ and $6T^2-4T-3$.
None vanishes for $0<T<1/\sqrt6$. Consequently every solution under consideration gives a common zero of $\widetilde F,H_1,H_2$.

For a cubic $a_3p^3+a_2p^2+a_1p+a_0$ and a quadratic $b_2p^2+b_1p+b_0$, define
\[
\begin{aligned}
\mathcal A&=a_1b_2^2-a_3b_0b_2-a_2b_1b_2+a_3b_1^2,\\
\mathcal B&=a_0b_2^2-a_2b_0b_2+a_3b_1b_0.
\end{aligned}
\]
Their resultant is the polynomial quotient
\[
\frac{b_0\mathcal A^2-b_1\mathcal A\mathcal B+b_2\mathcal B^2}{b_2^2}.
\]
Use this formula to define $R_j(T)=\Res_p(\widetilde F,H_j)$ for $j=1,2$. Up to a nonzero constant, the factorization of $R_1$ is
\begin{align*}
R_1={}&(T-1)^8(2T+1)^{12}(2T^2+1)^{12}(2T^2-3)^2(6T^2-1)^2\\
&\cdot(6T^2-4T-3)^{10}(2T^2+8T-1)^4\mathcal P(T)\mathcal E(T)^2.
\end{align*}
Here $\mathcal P$ and $\mathcal E$ are specified by their coefficients in ascending order. The coefficients of $\mathcal P$, from $T^0$ through $T^{18}$, are
\[
\begin{gathered}
(-243,2349,-6561,5229,140586,-485388,-1094108,\\
3265740,5130696,-10136384,-10261392,13062960,\\
8752864,-7766208,-4498752,334656,839808,601344,124416).
\end{gathered}
\]
The coefficients of $\mathcal E$, from $T^0$ through $T^{19}$, are
\[
\begin{gathered}
(27,-2430,8190,60228,26352,-164448,-3342112,-8861120,\\
22756256,54892992,-100341696,-121096320,264964864,\\
20152832,-294750720,214062080,-35962112,-18583040,\\
4288000,1664000).
\end{gathered}
\]
The Euclidean algorithm gives $\gcd(\mathcal E,R_2)=1$. All the other factors of $R_1$, apart from $2T^2+8T-1$ and $\mathcal P$, have no zero in $(0,1/\sqrt6)$. Thus a solution must satisfy one of these last two equations.

At the positive root of $2T^2+8T-1$, the polynomial $\widetilde F$ is proportional to
\[
(p-1)(p-9/7)(p-17/3).
\]
Only $p=9/7$ lies in $(1,2)$. It fails $G_1=0$: substitution in the rational expression for $G_1$ gives a numerator with a nonzero linear remainder modulo $2T^2+8T-1$, and a denominator that does not vanish. Since the positive root has degree two over $\mathbb Q$, this excludes the first possibility.

For the second possibility, form the Sturm sequence beginning with $\mathcal P,\mathcal P'$ and continuing by taking the negative remainder at each Euclidean division. Multiplication of a term by a positive constant does not change its signs. The variation counts at the following rational points are
\begin{center}
\begin{tabular}{crrrrrr}
\toprule
$T$ & $0$ & $0.143$ & $0.144$ & $0.295$ & $0.296$ & $0.409$\\
\midrule
Variations & $9$ & $9$ & $8$ & $8$ & $7$ & $7$\\
\bottomrule
\end{tabular}
\end{center}
Since $0.296<1/\sqrt6<0.409$, precisely two roots remain. For the pair $\widetilde F,H_1$, the Euclidean algorithm also gives
\[
\gcd(\mathcal P,b_2)=\gcd(\mathcal P,\mathcal A)=1.
\]
At either root, therefore, the last linear remainder yields $p=-\mathcal B/\mathcal A$. Substitution into $D_1$ gives a nonzero value at both roots, so $X=X_1/D_1$ and $Y=Y_1/D_1$.

It remains to use the positivity of the face areas. Write $h=\ph(p)$, $u=\ph(q)$, $v=\ph(r)$, $w=\ph(s)$ and $k=(1+h)/2$.
The vector $N=(n_1+n_2)/2$ has squared length $k$, and $N\cdot n_3=u$, $N\cdot n_4=v$. The equilibrium equation
\[
2S_1N+S_3n_3+S_4n_4=0
\]
gives the necessary inequality
\begin{equation}
\mathcal C:=uv-wk=\frac{(1-w^2)S_3S_4}{4S_1^2}>0.
\label{eq:reflectionpositive}
\end{equation}
Set
\[
Z=\frac{\alpha^2X-\beta^2Y}{4}=qr\ \text{ and }\ W=\frac{\alpha^2X+\beta^2Y}{2}=q^2+r^2.
\]
Then
\[
4\mathcal C=Z(9-3W+Z^2)-\ph(s)(2+3p-p^3),
\]
a rational expression in $T,p,X,Y$. Start with the two intervals for $T$ in the preceding table and bisect, retaining the unique root counted by the Sturm sequence, until the width is less than $10^{-70}$. Substitution in the formulas for $p,X,Y$ and $4\mathcal C$, using rational intervals, gives
\begin{center}
\begin{tabular}{ccc}
\toprule
Root interval for $T$ & Interval for $p$ & Interval for $4\mathcal C$\\
\midrule
$(0.143,0.144)$ & $(1.69702,1.69703)$ & $(-0.704,-0.703)$\\
$(0.295,0.296)$ & $(1.25214,1.25215)$ & $(-0.168,-0.167)$\\
\bottomrule
\end{tabular}
\end{center}
Both possibilities contradict \eqref{eq:reflectionpositive}. Thus $b\ne c$ is impossible, completing the algebraic step of Theorem~\ref{thm:reflection}.
\section{Verification of the computations}\label{sec:reproducibility}

The computer-assisted steps of the paper are checked by the Python scripts listed below, which are available as ancillary files with the arXiv version. Each script names the statements it checks, in the order in which they appear in the text, and stops with an error if a check fails. The first six scripts in the table use only the standard library and exact rational arithmetic. Polynomial identities in one, two or three variables are checked by evaluation at sufficiently many rational points, using an a priori bound for the degrees.

The next seven scripts use numpy and the module \texttt{interval\_taylor.py}, which contains the interval arithmetic, the enclosures of the roots of the cubic equations, the Taylor models of order one and two, the branch and bound and the injectivity test of Section~\ref{sec:bb}. After every arithmetic operation the endpoints are moved outward by one unit in the last place, so the computed intervals contain the exact values. On our computer the seven scripts take about 40 seconds, a little over two minutes, 25 seconds, two minutes, ten minutes, 20 seconds and a few seconds, respectively.

\begin{center}\small
\begin{tabular}{p{.33\textwidth}p{.59\textwidth}}\toprule
File&Content\\\midrule
\texttt{check\_section6.py}&the bounds \eqref{eq:contractionbounds} and \eqref{eq:intervalgeometry} for the three boxes, the localization \eqref{eq:sharpbox}, the edge bounds and the bound for $\partial f/\partial b$ (Section~\ref{sec:intervals})\\
\texttt{check\_T2.py}&the exact identities in the proof of Theorem~\ref{thm:T2} and the resultant \eqref{t2:eq:resultant}\\
\texttt{check\_T1.py}&the identities, the resultant \eqref{sym:eq:resultant}, the linear subresultant, the bounds of Lemma~\ref{sym:lem:root}, the admissibility intervals and the decimal values in Theorems~\ref{thm:T1} and~\ref{sym:thm:classification}, computed in rational interval arithmetic\\
\texttt{check\_galois\_T1.py}&the factorizations used in Proposition~\ref{prop:T1galois}\\
\texttt{check\_appendixA.py}&the identity leading to \eqref{eq:Ezero} in the proof of Theorem~\ref{thm:twopairs}, and the $280$ Bernstein coefficients of Appendix~\ref{app:bernstein}\\
\texttt{check\_appendixB.py}&the elimination, the factorization of $R_1$, the Sturm counts and the final intervals of Appendix~\ref{app:reflection}\\
\texttt{check\_orthocentric.py}&the branch and bound on the six charts and the injectivity bound of Section~\ref{sec:orthocentric}\\
\texttt{check\_isodynamic.py}&the branch and bound on the six charts and the injectivity bound of Section~\ref{sec:isodynamic}\\
\texttt{check\_midsphere.py}&the branch and bound on the three charts and the injectivity bound of Section~\ref{sec:midsphere}\\
\texttt{check\_apex.py}&the branch and bound on the ten charts and the injectivity bound of Section~\ref{sec:apex}\\
\texttt{check\_acute.py}&the branch and bound in six variables and the injectivity bound of Section~\ref{sec:acute}\\
\texttt{check\_eqface.py}&the branch and bound on the six charts of Section~\ref{sec:eqface}\\
\texttt{check\_isomorley.py}&the Krawczyk test and the edge bounds in the proof of Theorem~\ref{thm:asymmetric}\\
\texttt{interval\_taylor.py}&interval arithmetic, root enclosures, Taylor models, branch and bound (used by the previous seven scripts)\\
\texttt{morley\_tools.py}&polynomial arithmetic, rational intervals, determinants, Sturm sequences, factorization modulo a prime\\
\texttt{explore\_isomorley.py}&the numerical experiments of Remark~\ref{rem:isoglobal} and Sections~\ref{sec:isoproduct} and~\ref{sec:skew} (floating point, not part of any proof)\\\bottomrule
\end{tabular}
\end{center}

The three contraction constants in \eqref{eq:contractionbounds} are less than $0.061689$, $0.059243$ and $0.065397$ for the regular tetrahedron, $T_1$ and $T_2$, respectively. The coefficients $A_i$, $B_i$, $C_i$ in \eqref{eq:reflectionlinear}, the formula for $\det K$ and the formula for $S_1$ in Appendix~\ref{app:reflection} were derived by hand; the script \texttt{check\_appendixB.py} starts from them.

The last script in the table is of a different kind. It repeats the numerical experiments behind Remark~\ref{rem:isoglobal} and Sections~\ref{sec:isoproduct} and~\ref{sec:skew} in ordinary floating point arithmetic with numpy and scipy, and prints the numbers quoted there. Its random searches use fixed seeds, but with other versions of numpy or scipy the counts may differ slightly. Nothing in it is used in a proof. It takes about $25$ minutes.

\end{appendices}

\section*{Acknowledgments}
\phantomsection
\addcontentsline{toc}{section}{Acknowledgments}

The author thanks Tran Manh Dung and Nguyen Xuan Tho for their help with the computations. The AI assistants ChatGPT and Claude were used for exploratory computations, for checking algebra and in preparing the manuscript; the author takes full responsibility for the content.

\phantomsection

\begin{flushright}
Quang Hung Tran,\\
High School for Gifted Students,\\
Vietnam National University, Hanoi, Vietnam,\\
Email: \href{mailto:tranquanghung@hus.edu.vn}{tranquanghung@hus.edu.vn}\\
ORCID: \href{https://orcid.org/0000-0003-2468-4972}{0000-0003-2468-4972}
\end{flushright}

\end{document}